\documentclass[12pt]{amsart}
\usepackage{oldgerm}
\usepackage{latexsym}
\usepackage{amsmath}
\usepackage{amscd}
\usepackage{amssymb} 
\usepackage{amsmath}
\usepackage{amsthm}
\usepackage{latexsym}
\usepackage{arydshln}
\begin{document}
\title [Explicit formula for  Siegel series]{An explicit formula for  the  Siegel series of a quadratic form  over a non-archimedean local field
}
\author{Tamotsu IKEDA}
\address{Graduate school of mathematics, Kyoto University, Kitashirakawa, Kyoto, 606-8502, Japan}
\email{ikeda@math.kyoto-u.ac.jp}
\author{Hidenori KATSURADA}
\address{Muroran Institute of Technology
27-1 Mizumoto, Muroran, 050-8585, Japan}
\email{hidenori@mmm.muroran-it.ac.jp}
\date{27 September, 2017}

\maketitle
\newcommand{\alp}{\alpha}
\newcommand{\bet}{\beta}
\newcommand{\gam}{\gamma}
\newcommand{\del}{\delta}
\newcommand{\eps}{\epsilon}
\newcommand{\zet}{\zeta}
\newcommand{\tht}{\theta}
\newcommand{\iot}{\iota}
\newcommand{\kap}{\kappa}
\newcommand{\lam}{\lambda}
\newcommand{\sig}{\sigma}
\newcommand{\ups}{\upsilon}
\newcommand{\ome}{\omega}
\newcommand{\vep}{\varepsilon}
\newcommand{\vth}{\vartheta}
\newcommand{\vpi}{\varpi}
\newcommand{\vrh}{\varrho}
\newcommand{\vsi}{\varsigma}
\newcommand{\vph}{\varphi}
\newcommand{\Gam}{\Gamma}
\newcommand{\Del}{\Delta}
\newcommand{\Tht}{\Theta}
\newcommand{\Lam}{\Lambda}
\newcommand{\Sig}{\Sigma}
\newcommand{\Ups}{\Upsilon}
\newcommand{\Ome}{\Omega}


\newcommand{\frka}{{\mathfrak a}}    \newcommand{\frkA}{{\mathfrak A}}
\newcommand{\frkb}{{\mathfrak b}}    \newcommand{\frkB}{{\mathfrak B}}
\newcommand{\frkc}{{\mathfrak c}}    \newcommand{\frkC}{{\mathfrak C}}
\newcommand{\frkd}{{\mathfrak d}}    \newcommand{\frkD}{{\mathfrak D}}
\newcommand{\frke}{{\mathfrak e}}    \newcommand{\frkE}{{\mathfrak E}}
\newcommand{\frkf}{{\mathfrak f}}    \newcommand{\frkF}{{\mathfrak F}}
\newcommand{\frkg}{{\mathfrak g}}    \newcommand{\frkG}{{\mathfrak G}}
\newcommand{\frkh}{{\mathfrak h}}    \newcommand{\frkH}{{\mathfrak H}}
\newcommand{\frki}{{\mathfrak i}}    \newcommand{\frkI}{{\mathfrak I}}
\newcommand{\frkj}{{\mathfrak j}}    \newcommand{\frkJ}{{\mathfrak J}}
\newcommand{\frkk}{{\mathfrak k}}    \newcommand{\frkK}{{\mathfrak K}}
\newcommand{\frkl}{{\mathfrak l}}    \newcommand{\frkL}{{\mathfrak L}}
\newcommand{\frkm}{{\mathfrak m}}    \newcommand{\frkM}{{\mathfrak M}}
\newcommand{\frkn}{{\mathfrak n}}    \newcommand{\frkN}{{\mathfrak N}}
\newcommand{\frko}{{\mathfrak o}}    \newcommand{\frkO}{{\mathfrak O}}
\newcommand{\frkp}{{\mathfrak p}}    \newcommand{\frkP}{{\mathfrak P}}
\newcommand{\frkq}{{\mathfrak q}}    \newcommand{\frkQ}{{\mathfrak Q}}
\newcommand{\frkr}{{\mathfrak r}}    \newcommand{\frkR}{{\mathfrak R}}
\newcommand{\frks}{{\mathfrak s}}    \newcommand{\frkS}{{\mathfrak S}}
\newcommand{\frkt}{{\mathfrak t}}    \newcommand{\frkT}{{\mathfrak T}}
\newcommand{\frku}{{\mathfrak u}}    \newcommand{\frkU}{{\mathfrak U}}
\newcommand{\frkv}{{\mathfrak v}}    \newcommand{\frkV}{{\mathfrak V}}
\newcommand{\frkw}{{\mathfrak w}}    \newcommand{\frkW}{{\mathfrak W}}
\newcommand{\frkx}{{\mathfrak x}}    \newcommand{\frkX}{{\mathfrak X}}
\newcommand{\frky}{{\mathfrak y}}    \newcommand{\frkY}{{\mathfrak Y}}
\newcommand{\frkz}{{\mathfrak z}}    \newcommand{\frkZ}{{\mathfrak Z}}


\newcommand{\bfa}{{\mathbf a}}    \newcommand{\bfA}{{\mathbf A}}
\newcommand{\bfb}{{\mathbf b}}    \newcommand{\bfB}{{\mathbf B}}
\newcommand{\bfc}{{\mathbf c}}    \newcommand{\bfC}{{\mathbf C}}
\newcommand{\bfd}{{\mathbf d}}    \newcommand{\bfD}{{\mathbf D}}
\newcommand{\bfe}{{\mathbf e}}    \newcommand{\bfE}{{\mathbf E}}
\newcommand{\bff}{{\mathbf f}}    \newcommand{\bfF}{{\mathbf F}}
\newcommand{\bfg}{{\mathbf g}}    \newcommand{\bfG}{{\mathbf G}}
\newcommand{\bfh}{{\mathbf h}}    \newcommand{\bfH}{{\mathbf H}}
\newcommand{\bfi}{{\mathbf i}}    \newcommand{\bfI}{{\mathbf I}}
\newcommand{\bfj}{{\mathbf j}}    \newcommand{\bfJ}{{\mathbf J}}
\newcommand{\bfk}{{\mathbf k}}    \newcommand{\bfK}{{\mathbf K}}
\newcommand{\bfl}{{\mathbf l}}    \newcommand{\bfL}{{\mathbf L}}
\newcommand{\bfm}{{\mathbf m}}    \newcommand{\bfM}{{\mathbf M}}
\newcommand{\bfn}{{\mathbf n}}    \newcommand{\bfN}{{\mathbf N}}
\newcommand{\bfo}{{\mathbf o}}    \newcommand{\bfO}{{\mathbf O}}
\newcommand{\bfp}{{\mathbf p}}    \newcommand{\bfP}{{\mathbf P}}
\newcommand{\bfq}{{\mathbf q}}    \newcommand{\bfQ}{{\mathbf Q}}
\newcommand{\bfr}{{\mathbf r}}    \newcommand{\bfR}{{\mathbf R}}
\newcommand{\bfs}{{\mathbf s}}    \newcommand{\bfS}{{\mathbf S}}
\newcommand{\bft}{{\mathbf t}}    \newcommand{\bfT}{{\mathbf T}}
\newcommand{\bfu}{{\mathbf u}}    \newcommand{\bfU}{{\mathbf U}}
\newcommand{\bfv}{{\mathbf v}}    \newcommand{\bfV}{{\mathbf V}}
\newcommand{\bfw}{{\mathbf w}}    \newcommand{\bfW}{{\mathbf W}}
\newcommand{\bfx}{{\mathbf x}}    \newcommand{\bfX}{{\mathbf X}}
\newcommand{\bfy}{{\mathbf y}}    \newcommand{\bfY}{{\mathbf Y}}
\newcommand{\bfz}{{\mathbf z}}    \newcommand{\bfZ}{{\mathbf Z}}


\newcommand{\cala}{{\mathcal A}}
\newcommand{\calb}{{\mathcal B}}
\newcommand{\calc}{{\mathcal C}}
\newcommand{\cald}{{\mathcal D}}
\newcommand{\cale}{{\mathcal E}}
\newcommand{\calf}{{\mathcal F}}
\newcommand{\calg}{{\mathcal G}}
\newcommand{\calh}{{\mathcal H}}
\newcommand{\cali}{{\mathcal I}}
\newcommand{\calj}{{\mathcal J}}
\newcommand{\calk}{{\mathcal K}}
\newcommand{\call}{{\mathcal L}}
\newcommand{\calm}{{\mathcal M}}
\newcommand{\caln}{{\mathcal N}}
\newcommand{\calo}{{\mathcal O}}
\newcommand{\calp}{{\mathcal P}}
\newcommand{\calq}{{\mathcal Q}}
\newcommand{\calr}{{\mathcal R}}
\newcommand{\cals}{{\mathcal S}}
\newcommand{\calt}{{\mathcal T}}
\newcommand{\calu}{{\mathcal U}}
\newcommand{\calv}{{\mathcal V}}
\newcommand{\calw}{{\mathcal W}}
\newcommand{\calx}{{\mathcal X}}
\newcommand{\caly}{{\mathcal Y}}
\newcommand{\calz}{{\mathcal Z}}


\newcommand{\scra}{{\mathscr A}}
\newcommand{\scrb}{{\mathscr B}}
\newcommand{\scrc}{{\mathscr C}}
\newcommand{\scrd}{{\mathscr D}}
\newcommand{\scre}{{\mathscr E}}
\newcommand{\scrf}{{\mathscr F}}
\newcommand{\scrg}{{\mathscr G}}
\newcommand{\scrh}{{\mathscr H}}
\newcommand{\scri}{{\mathscr I}}
\newcommand{\scrj}{{\mathscr J}}
\newcommand{\scrk}{{\mathscr K}}
\newcommand{\scrl}{{\mathscr L}}
\newcommand{\scrm}{{\mathscr M}}
\newcommand{\scrn}{{\mathscr N}}
\newcommand{\scro}{{\mathscr O}}
\newcommand{\scrp}{{\mathscr P}}
\newcommand{\scrq}{{\mathscr Q}}
\newcommand{\scrr}{{\mathscr R}}
\newcommand{\scrs}{{\mathscr S}}
\newcommand{\scrt}{{\mathscr T}}
\newcommand{\scru}{{\mathscr U}}
\newcommand{\scrv}{{\mathscr V}}
\newcommand{\scrw}{{\mathscr W}}
\newcommand{\scrx}{{\mathscr X}}
\newcommand{\scry}{{\mathscr Y}}
\newcommand{\scrz}{{\mathscr Z}}


\newcommand{\AAA}{{\mathbb A}} 
\newcommand{\BB}{{\mathbb B}}
\newcommand{\CC}{{\mathbb C}}
\newcommand{\DD}{{\mathbb D}}
\newcommand{\EE}{{\mathbb E}}
\newcommand{\FF}{{\mathbb F}}
\newcommand{\GG}{{\mathbb G}}
\newcommand{\HH}{{\mathbb H}}
\newcommand{\II}{{\mathbb I}}
\newcommand{\JJ}{{\mathbb J}}
\newcommand{\KK}{{\mathbb K}}
\newcommand{\LL}{{\mathbb L}}
\newcommand{\MM}{{\mathbb M}}
\newcommand{\NN}{{\mathbb N}}
\newcommand{\OO}{{\mathbb O}}
\newcommand{\PP}{{\mathbb P}}
\newcommand{\QQ}{{\mathbb Q}}
\newcommand{\RR}{{\mathbb R}}
\newcommand{\SSS}{{\mathbb S}} 
\newcommand{\TT}{{\mathbb T}}
\newcommand{\UU}{{\mathbb U}}
\newcommand{\VV}{{\mathbb V}}
\newcommand{\WW}{{\mathbb W}}
\newcommand{\XX}{{\mathbb X}}
\newcommand{\YY}{{\mathbb Y}}
\newcommand{\ZZ}{{\mathbb Z}}


\newcommand{\tta}{\hbox{\tt a}}    \newcommand{\ttA}{\hbox{\tt A}}
\newcommand{\ttb}{\hbox{\tt b}}    \newcommand{\ttB}{\hbox{\tt B}}
\newcommand{\ttc}{\hbox{\tt c}}    \newcommand{\ttC}{\hbox{\tt C}}
\newcommand{\ttd}{\hbox{\tt d}}    \newcommand{\ttD}{\hbox{\tt D}}
\newcommand{\tte}{\hbox{\tt e}}    \newcommand{\ttE}{\hbox{\tt E}}
\newcommand{\ttf}{\hbox{\tt f}}    \newcommand{\ttF}{\hbox{\tt F}}
\newcommand{\ttg}{\hbox{\tt g}}    \newcommand{\ttG}{\hbox{\tt G}}
\newcommand{\tth}{\hbox{\tt h}}    \newcommand{\ttH}{\hbox{\tt H}}
\newcommand{\tti}{\hbox{\tt i}}    \newcommand{\ttI}{\hbox{\tt I}}
\newcommand{\ttj}{\hbox{\tt j}}    \newcommand{\ttJ}{\hbox{\tt J}}
\newcommand{\ttk}{\hbox{\tt k}}    \newcommand{\ttK}{\hbox{\tt K}}
\newcommand{\ttl}{\hbox{\tt l}}    \newcommand{\ttL}{\hbox{\tt L}}
\newcommand{\ttm}{\hbox{\tt m}}    \newcommand{\ttM}{\hbox{\tt M}}
\newcommand{\ttn}{\hbox{\tt n}}    \newcommand{\ttN}{\hbox{\tt N}}
\newcommand{\tto}{\hbox{\tt o}}    \newcommand{\ttO}{\hbox{\tt O}}
\newcommand{\ttp}{\hbox{\tt p}}    \newcommand{\ttP}{\hbox{\tt P}}
\newcommand{\ttq}{\hbox{\tt q}}    \newcommand{\ttQ}{\hbox{\tt Q}}
\newcommand{\ttr}{\hbox{\tt r}}    \newcommand{\ttR}{\hbox{\tt R}}
\newcommand{\tts}{\hbox{\tt s}}    \newcommand{\ttS}{\hbox{\tt S}}
\newcommand{\ttt}{\hbox{\tt t}}    \newcommand{\ttT}{\hbox{\tt T}}
\newcommand{\ttu}{\hbox{\tt u}}    \newcommand{\ttU}{\hbox{\tt U}}
\newcommand{\ttv}{\hbox{\tt v}}    \newcommand{\ttV}{\hbox{\tt V}}
\newcommand{\ttw}{\hbox{\tt w}}    \newcommand{\ttW}{\hbox{\tt W}}
\newcommand{\ttx}{\hbox{\tt x}}    \newcommand{\ttX}{\hbox{\tt X}}
\newcommand{\tty}{\hbox{\tt y}}    \newcommand{\ttY}{\hbox{\tt Y}}
\newcommand{\ttz}{\hbox{\tt z}}    \newcommand{\ttZ}{\hbox{\tt Z}}

\newcommand{\phm}{\phantom}
\newcommand{\ds}{\displaystyle }
\newcommand{\smallstrut}{\vphantom{\vrule height 3pt }}
\def\bdm #1#2#3#4{\left(
\begin{array} {c|c}{\ds{#1}}
 & {\ds{#2}} \\ \hline
{\ds{#3}\vphantom{\ds{#3}^1}} &  {\ds{#4}}
\end{array}
\right)}
\newcommand{\wtd}{\widetilde }
\newcommand{\bsl}{\backslash }
\newcommand{\GL}{{\mathrm{GL}}}
\newcommand{\SL}{{\mathrm{SL}}}
\newcommand{\GSp}{{\mathrm{GSp}}}
\newcommand{\PGSp}{{\mathrm{PGSp}}}
\newcommand{\SP}{{\mathrm{Sp}}}
\newcommand{\SO}{{\mathrm{SO}}}
\newcommand{\SU}{{\mathrm{SU}}}
\newcommand{\Ind}{\mathrm{Ind}}
\newcommand{\Hom}{{\mathrm{Hom}}}
\newcommand{\Ad}{{\mathrm{Ad}}}
\newcommand{\Sym}{{\mathrm{Sym}}}
\newcommand{\Mat}{\mathrm{M}}
\newcommand{\sgn}{\mathrm{sgn}}
\newcommand{\trs}{\,^t\!}
\newcommand{\iu}{\sqrt{-1}}
\newcommand{\oo}{\hbox{\bf 0}}
\newcommand{\ono}{\hbox{\bf 1}}
\newcommand{\smallcirc}{\lower .3em \hbox{\rm\char'27}\!}
\newcommand{\bAf}{\bA_{\hbox{\eightrm f}}}
\newcommand{\thalf}{{\textstyle{\frac12}}}
\newcommand{\shp}{\hbox{\rm\char'43}}
\newcommand{\Gal}{\operatorname{Gal}}

\newcommand{\bdel}{{\boldsymbol{\delta}}}
\newcommand{\bchi}{{\boldsymbol{\chi}}}
\newcommand{\bgam}{{\boldsymbol{\gamma}}}
\newcommand{\bome}{{\boldsymbol{\omega}}}
\newcommand{\bpsi}{{\boldsymbol{\psi}}}
\newcommand{\GK}{\mathrm{GK}}
\newcommand{\EGK}{\mathrm{EGK}}
\newcommand{\ord}{\mathrm{ord}}
\newcommand{\diag}{\mathrm{diag}}
\newcommand{\ua}{{\underline{a}}}
\newcommand{\ZZn}{\ZZ_{\geq 0}^n}

\theoremstyle{plain}
\newtheorem{theorem}{Theorem}[section]
\newtheorem{lemma}{Lemma}[section]
\newtheorem{proposition}{Proposition}[section]
\newtheorem{corollary}{\bf {Corollary}}[section]
\theoremstyle{definition}
\newtheorem{definition}{Definition}[section]
\newtheorem{conjecture}{Conjecture}[section]
\newtheorem{remark}{{\bf Remark}}[section]
\theoremstyle{remark}
\newtheorem{formula}{\bf {Formula}}[section]

%

\def\mattwono(#1;#2;#3;#4){\begin{array}{cc}
                               #1  & #2 \\
                               #3  & #4
                                      \end{array}}

\def\mattwo(#1;#2;#3;#4){\left(\begin{matrix}
                               #1 & #2 \\
                               #3  & #4
                                      \end{matrix}\right)}
 \def\smallmattwo(#1;#2;#3;#4){\left(\begin{smallmatrix}
                               #1 & #2 \\
                               #3  & #4
                                      \end{smallmatrix}\right)}                                     
                                      
 \def\matthree(#1;#2;#3;#4;#5;#6;#7;#8;#9){\left(\begin{matrix}
                               #1 & #2  & #3\\
                               #4  & #5 & #6\\
                               #7  & #8 &#9 
                                      \end{matrix}\right)}                                     
                                      
\def\mattwo(#1;#2;#3;#4){\left(\begin{matrix}
                               #1 & #2 \\
                               #3  & #4
                                      \end{matrix}\right)}  

\def\rowthree(#1;#2;#3){\begin{matrix}
                               #1   \\
                               #2  \\
                               #3
                                      \end{matrix}}  
\def\columnthree(#1;#2;#3){\begin{matrix}
                               #1   &   #2  &  #3
                                      \end{matrix}}  
                                      
\def\rowfive(#1;#2;#3;#4;#5){\begin{array}{lllll}
                               #1   \\
                               #2  \\
                               #3 \\
                               #4 \\
                               #5                              
                                      \end{array}} 

\def\columnfive(#1;#2;#3;#4;#5){\begin{array}{lllll}
                               #1   &   #2  &  #3 & #4 & #5
                                \end{array}}

\def\mattwothree(#1;#2;#3;#4;#5;#6){\begin{matrix}
                               #1 & #2  & #3  \\
                               #4 & #5  & #6
                                      \end{matrix}}  
\def\matthreetwo(#1;#2;#3;#4;#5;#6){\begin{array}{lc}
                               #1  & #2  \\
                               #3  & #4 \\
                               #5  & #6
                                      \end{array}}  
\def\columnthree(#1;#2;#3){\begin{matrix}
                               #1 & #2 & #3  
                                  \end{matrix}}  
\def\rowthree(#1;#2;#3){\begin{matrix}
                               #1 \\
                                #2 \\
                                #3  
                                  \end{matrix}}  
\def\smallddots{\mathinner
{\mskip1mu\raise3pt\vbox{\kern7pt\hbox{.}}
\mskip1mu\raise0pt\hbox{.}
\mskip1mu\raise-3pt\hbox{.}\mskip1mu}}

\begin{abstract}
Let $F$ be a non-archimedean local field of characteristic $0$, and $\frko$ the ring of integers in $F$. 
We give an explicit formula for the Siegel series of a half-integral  matrix over $\frko$.
This formula expresses the Siegel series of a half-integral matrix $B$  explicitly in terms of the Gross-Keating invariant of $B$ and its related invariants.
\end{abstract}

\section{Introduction}
The Siegel series  is one of the simplest but most important subjects in number theory, and is related with various types of arithmetic theories  of modular forms. It appears in the Fourier coefficients of the Hilbert-Siegel Eisenstein series, and it is also related with the Fourier coefficients of the lift, called the Duke-Imamoglu-Ikeda lift, constructed by the first named author in \cite{Ike1} (see also  \cite{Ike-Ya}). It also plays very important roles in the study of various types of L-functions associated with cusp forms through the pullback formula(cf.  \cite{B}, \cite{B-S},  \cite{Ko}, \cite{Or}, \cite{Sh3}).  Moreover, it is closely related to arithmetic algebraic geometry (cf. \cite{Ku1}, \cite{Ku2}, \cite{K-R-Y}). In all cases, precise information on the Siegel series is necessary. Thus it is very important to give an explicit form of the Siegel series. In \cite{Kat1}, the second named author gave an explicit formula for the Siegel series of a half-integral matrix over $\ZZ_p$ with any prime number $p$ of any degree. The formula  is useful for a practical computation of the Siegel series, and has several interesting applications. Indeed, the formula was used to give   special values of the standard L-functions of Siegel modular forms  and the triple product L-functions of elliptic modular forms exactly  (cf. \cite{D-I-K},\cite{I-K-P-Y},\cite{Ib-Kat}, \cite{Kat2},\cite{Kat3},\cite{Kat-Mi}).  These computations played important roles not only in  confirming several conjectures on such values numerically but also in proposing new conjecture on them.  The formula was also one of key ingredients in proving the conjecture on the period of  the Duke-Imamoglu-Ikeda lift proposed in 
\cite{Ike2} (cf. \cite{Kat-Kaw2}). Moreover, it  was used to relate the local intersection multiplicities on certain Shimura varieties to the derivatives of certain local Whittaker functions in \cite {Ra-Wed}. 

 However, it is not satisfactory in the following reasons. Firstly, the formula is complicated in the case $p=2$, and it seems difficult to unify it with the formula in the case that $p$ is odd as it is. Secondly, it does not seem clear what invariants determine the Siegel series.
Though there are other explicit formulas for local densities (cf. \cite{Hi-Sa},  \cite{Ya}), it seems difficult 
to resolve the above problems using them. 
In \cite{Wed}, Wedhorn reformulated the formula in \cite{Kat1} for the Siegel series of a half-integral matrix of degree three  in terms of  the Gross-Keating invariant in \cite{G-K}. In \cite{Ot}, Otsuka gave an explicit formula for the Siegel series of a half-integral matrix of degree two  over the ring of integers of any non-archimedean local field of characteristic $0$.
In this paper, we give an explicit formula of the Siegel series of a half-integral matrix of any degree over any non-archimedean local field of characteristic $0$.

We explain our main result more precisely. Let $F$ be a  non-archimedean local field of characteristic $0$ with the residue field $\frkk$ and  let $\frko$ be the ring of integers in $F$. Put $q=\#(\frkk)$.  For a non-degenerate half-integral  matrix $B$ of degree $n$ over $\frko$, let $b(B,s)$ be the Siegel series of $B$. Then, as will be explained in Section 2, we obtain a polynomial $\widetilde F(B,X)$ in $X^{1/2}$ and $X^{-1/2}$ attached to $b(B,s)$. Let $\GK(B)$ be the Gross-Keating invariant of $B$. 
We then define a set $\EGK(B)$ of invariants of $B$ (cf. Definition \ref{def.3.5}), which will be called the extended GK datum of $B$. In the non-dyadic case, $B$ has a diagonal Jordan decomposition:
\[B \sim \vpi^{m_1} U_1 \bot \cdots \bot \vpi^{m_r}U_r\]
with $m_1,\ldots,m_r$ non-negative integers such that $m_1<\cdots<m_r$ and $U_i$ a diagonal unimodular matrix  of degree $n_i$ for $i=1,\ldots,r$. Then 
\[\GK(B)=(\underbrace{m_1,\ldots,m_1}_{n_1},\ldots,\underbrace{m_r,\ldots,m_r}_{n_r}).\]
For each $i=1,\ldots,r$, we define $\zeta_i$ as 
\[\zeta_i=\begin{cases}
\xi_{B^{(n_1+\cdots+n_i)}} & \text{ if } \deg B^{(n_1+\cdots+n_i)} \text{ is even}\\
\eta_{B^{(n_1+\cdots+n_i)}} & \text{ if } \deg B^{(n_1+\cdots+n_i)} \text{ is odd},
\end{cases}\]
where $B^{(k)}$ is the upper left $k \times k$ block of $B$, and $\xi_A$ and $\eta_A$ are the invariants of a half-integral matrix $A$, which will be defined in Section 2.
Then $\EGK(B)$ is defined as $(n_1,\ldots,n_r;m_1,\ldots,m_r;\zeta_1,\ldots,\zeta_r)$. In the dyadic case, the invariants  $\GK(B)$ and 
$\EGK(B)$ of $B$ are more elaborately defined.
Then we express $\widetilde F(B,X)$  explicitly in terms of $\EGK(B)$.
This polynomial is universal in the following sense. We define an $\EGK$ datum $G$ of length $n$ as an element $(n_1,\ldots,n_r;m_1,\ldots,m_r;\zeta_1,\ldots,\zeta_r)$ of
$\ZZ^r_{>0} \times \ZZ^r_{\ge 0} \times \{0,1,-1 \}^r$  with $n_1+\cdots +n_r=n$ satisfying certain conditions (cf. Definition \ref{def.4.5}). 
The $\EGK$ datum is defined by axiomatizing some properties of the  extended GK datum of a half-integral matrix, and naturally 
$\EGK(B)$ is an EGK datum (cf. Theorem \ref{th.4.2}). We define a Laurent polynomial 
$\widetilde \calf(G;Y,X)$ in  $X^{1/2},Y$ attached to $G$.  An explicit formula for $\widetilde \calf(G;Y,X)$ will be given in Proposition \ref{prop.4.1}. Then, our main result in this paper is as follows:

\begin{theorem}
\label{th.1.1} Let $B$ be a non-degenerate half-integral matrix of degree $n$ over $\frko$. Then we have 
\[\widetilde F(B,X)=\widetilde \calf(\EGK(B);q^{1/2},X).\]
\end{theorem}

This unifies the formula for $p=2$ with that for an odd prime $p$ in  \cite{Kat1}. Therefore, our result not only gives a generalization of  the main result in \cite{Kat1} but also reformulates it in a satisfactory way, and  is new even in the case $F=\QQ_p$. Our result also shows that
$\widetilde F(B,X)$ is  determined by $\EGK(B)$.
Based on this fact, Cho and Yamauchi \cite{C-Y} give an induction formula for the Siegel series, which is different from that
in the present paper. It gives a description of the local intersection multiplicities of the special cycles on the special fiber of Shimura varieties for $SO(2,n)$ with $n \le 3$ over a finite field in terms of Siegel series directly, which sheds a new light on Kudla's program \cite{Ku2}. Therefore, our result   plays a crucial role also  in arithmetic algebraic geometry.

A proof of Theorem \ref{th.1.1} will be given in Sections 6 and 7. To explain the method of the proof of our main result, first we review the proof of the main result in \cite{Kat1} with some modification.  
For simplicity let $q$ be odd, and let $B$ be a non-degenerate half-integral  matrix of degree $n$ over the ring $\frko$.
Then $B$ has the following diagonal Jordan decomposition 
\[B \sim \vpi^{a_1}u_1 \bot \cdots \bot \vpi^{a_{n}}u_n\]
with $a_1 \le \cdots \le a_n$ and $u_1,\ldots, u_n \in \frko^{\times}$. 
Then, in the case that  $\frko$ is the ring $\ZZ_p$ of $p$-adic integers, by the induction formulas for local densities (cf. Theorem \ref{th.5.1}), and the functional equation of the Siegel series (cf. Proposition \ref{prop.2.1}), we can
express $\widetilde F(B,X)$ in terms of $\widetilde F(B^{(n-1)},X)$ (cf. [\cite{Kat1}, Theorem 4.1)]).
This argument works for any Siegel series over a non-dyadic field. To be more precise, for an integer $1 \le i \le n$, we define $\frke_i$ as
\[\frke_i=
\begin{cases} a_1+\cdots +a_i  & \text{ if  $i$ is odd} \\
2[(a_1+\cdots+a_i)/2] & \text{ if $i$ is even.}
\end{cases}\]
Then we have the following  {\rm (cf. Theorem \ref{th.6.1})} :\\
{\it 
Under the above notation and the assumption, we have
\begin{align*}
\widetilde F(B,X)&=D(\frke_n,\frke_{n-1};\xi_{B^{(n-1)}};q^{1/2},X)\widetilde F(B^{(n-1)},q^{1/2}X) \\
&+\eta_B D(\frke_n, \frke_{n-1},\xi_{B^{(n-1)}};q^{1/2},X^{-1})\widetilde F(B^{(n-1)},q^{1/2}X^{-1})
\end{align*}
if $n$ is odd, and 
\begin{align*}
\widetilde F(B,X)&=C(\frke_n,\frke_{n-1},\xi_B;q^{1/2},X)\widetilde F(B^{(n-1)},q^{1/2}X) \\
&+ C(\frke_n, \frke_{n-1},\xi_B;q^{1/2},X^{-1})\widetilde F(B^{(n-1)},q^{1/2}X^{-1})
\end{align*}
if $n$ is even. In particular if $n=1$ we have
\[\widetilde F(B,X)=\sum_{i=0}^{r_1} X^{i-(r_1/2)},\]
where $\xi_*$ and $\eta_*$ are the invariants stated above, and $C(*,*,*;Y,X)$ and $D(*,*,*;Y,X)$ are rational
 functions in $X^{1/2}$ and $Y^{1/2}$, which will be defined in Definition \ref{def.4.3}.}

Using the induction formulas stated above repeatedly, we get an explicit formula for $\widetilde F(B,X)$. However, in the case $p=2$, we do not necessarily  have a diagonal Jordan decomposition for $B$,  and therefore, the formula for $p=2$ becomes complicated, and it is no hope to generalize it to any dyadic field as it is. 
To overcome this obstacle, we adopt a reduced  decomposition of $B$ (cf. Definition \ref{def.3.5}) instead of a diagonal Jordan decomposition. Let $B$ be a reduced  form of degree $n$. Then, in the case that   $B^{(n-1)}$ is  a reduced  form, and we can express 
$\widetilde F(B,X)$ in terms of $\widetilde F(B^{(n-1)},X)$ as (\ref{eq.1.1}) and (\ref{eq.1.2}) in the proof of Theorem 1.1 in the dyadic case. In the other cases, we can also express $\widetilde F(B,X)$ in terms of  $\widetilde F(B^{(n-2)},X)$ as  (\ref{eq.2.3}) and  (\ref{eq.3.1}) therein.
From the induction formulas we prove the explicit formula stated above. A key ingredient for proving such induction formulas is
a stability of the extended GK datum of a reduced  form (cf. Theorem \ref{th.3.3}), which was essentially proved in \cite{Ike-Kat}. 
Therefore the extended GK datum plays a very important roll not only in formulating main results but also in proving them. 

It seems interesting to consider a Hermitian version of the main result in this paper. 

We would like to thank Takuya Yamauchi and Sungmun Cho for many fruitful discussions and suggestions.  
The research was partially supported by the JSPS KAKENHI Grant Number  26610005, 24540005, 16H03919 and 17H02834.
 This work was supported by the Research Institute for Mathematical Sciences, an International Joint Usage/Research Center located in Kyoto University.

{\bf Notation} Let $R$ be a commutative ring. We denote by $R^{\times}$ the group of units in $R$. We denote by $M_{mn}(R)$ the set of $(m,n)$ matrices with entries in $R$, and especially write $M_n(R)=M_{nn}(R)$. 
We often identify an element $a$ of $R$ and the matrix $(a)$ of degree 1 whose component is $a$. If $m$ or $n$ is 0, we understand an element of $M_{mn}(R)$ is the {\it empty matrix} and denote it by $\emptyset$. Let $GL_n(R)$ be the group consisting of all invertible elements of $M_n(R)$, and ${\rm Sym}_n(R)$ the set of symmetric matrices of degree $n$ with entries in $R$.  For a  semigroup  $S$ we put $S^{\Box}=\{s^2 \ | \ s \in S \}$.
Let $R$ be an integral domain  of characteristic different from $2$, and $K$ its quotient field. We say that an element  $A $ of $\mathrm{Sym}_n(K)$ is non-degenerate if the determinant  $\det A$ of $A$ is non-zero. For a subset $S$ of $\mathrm{Sym}_n(K)$, we denote by
$S^{{\rm{nd}}}$ the subset of $S$ consisting of non-degenerate matrices.  We say that a symmetric matrix $A=(a_{ij})$ of degree $n$ with entries in $K$ is half-integral if $a_{ii} \ (i=1,...,n)$ and $2a_{ij} \ (1 \le i \not= j \le n)$ belong to $R$. We denote by $\calh_n(R)$ the set of half-integral matrices of degree $n$ over $R$. 
We note that $\calh_n(R)={\rm Sym}_n(R)$ if $R$ contains the inverse of $2$. 
We denote by $\ZZ_{> 0}$ and $\ZZ_{\ge 0}$ the set of positive integers and the set of non-negative integers, respectively.  
 For an $(m,n)$ matrix $X$ and an $(m,m)$ matrix $A$, we write $A[X] ={}^tXAX$, where $^t X$ denotes the transpose of $X$.
Let $G$ be a subgroup of $GL_n(K)$. Then we say that two elements $B$ and $B'$ in $\mathrm{Sym}_n(K)$  are $G$-equivalent if there is an element $g$ of $G$ such that $B'=B[g]$.  
For two square matrices $X$ and $Y$ we write $X \bot Y =\mattwo(X;O;O;Y)$. We often write $x \bot Y$ instead of $(x) \bot Y$ if $(x)$ is  a matrix of degree 1. 
For a square matrix $B$ of degree $n$ and integers $1 \le i_1, \ldots, i_r \le n, 1 \le j_1 , \ldots,  j_r \le n$ such that 
$i_k \not=i_l \ (k \not=l)$ and $j_{k'} \not= j_{l'} \ (k' \not= l')$ we denote by
$B(i_1,\ldots,i_r;j_1,\ldots,j_r)$ the matrix obtained from $B$ by deleting its $i_1,\ldots,i_r$-th rows and $j_1,\ldots,j_r$-th columns.
In particular, put  $T^{(k)}=T(k+1,\ldots,n;k+1,\ldots, n)$. We make the convention that $T^{(k)}$ is the empty matrix if $k=0$.  
We denote by $1_m$ the unit matrix of degree $m$ and by $O_{m,n}$ the zero matrix of type $(m,n)$. We sometimes abbreviate $O_{m,n}$ as $O$ if there is no fear of  confusion.
\section{Siegel series}
Let $F$ be a non-archimedean local field of characteristic $0$, and $\frko=\frko_F$ its ring of integers.
The maximal ideal and  the residue field of $\frko$ is denoted by $\frkp$ and $\frkk$, respectively.
We fix a prime element $\vpi$ of $\frko$ once and for all.
The cardinality of $\frkk$ is denoted by $q$.
Let $\ord=\ord_{\frkp}$ denote additive valuation on $F$  normalized so that $\ord(\vpi)=1$. If $a=0$,  We write $\ord(0)=\infty$
and we  make the convention that $\ord(0) > \ord(b)$ for any $b \in F^{\times}$.
We also denote by $|*|_{\frkp}$ denote the valuation on $F$ normalized so that $|\vpi|_{\frkp}=q^{-1}$. 
We put $e_0=\ord_{\frkp}(2)$.
 
For a non-degenerate element $B\in\calh_n(\frko)$, we put $D_B=(-4)^{[n/2]}\det B$.
If $n$ is even, we denote the discriminant ideal of $F(\sqrt{D_B})/F$ by $\frkD_B$.
We also put
\[
\xi_B=
\begin{cases} 
1 & \text{ if $D_B\in F^{\times 2}$,} \\
-1 & \text{ if $F(\sqrt{D_B})/F$ is unramified quadratic,} \\
0 & \text{ if $F(\sqrt{D_B})/F$ is ramified quadratic.} 
\end{cases}
\]
Put 
\[\frke_B=
\begin{cases}
\ord(D_B)-\ord(\frkD_B)   & \text{ if $n$ is even} \\
\ord(D_B)                         & \text{ if $n$ is odd.}
\end{cases}\]

Let $\langle  \ {} \ , \ {} \ \rangle=\langle \ {} \ , \ {} \ \rangle_F$ be the Hilbert symbol on  $F$. Let $B$ be  a non-degenerate symmetric matrix with entries in $F$ of degree $n$. Then $B$ is $GL_n(F)$-equivalent to $b_1 \bot \cdots \bot b_n$ with $b_1,\ldots,b_n \in F^{\times}$. Then we define $\varepsilon_B$ as  
\[\varepsilon_B=\prod_{1 \le i < j \le n} \langle b_i,b_j \rangle.\]
This does not depend on the choice of $b_1,\ldots,b_n$. We also denote by $\eta_B$ the Clifford invariant of $B$
(cf. \cite{Ike3}). 
Then we have 
\[\eta_B=
\begin{cases}
\langle -1,-1 \rangle^{m(m+1)/2} \langle (-1)^m,\det B \rangle \varepsilon_B & \ \text{if $n=2m+1$}\\
\langle -1,-1 \rangle^{m(m-1)/2} \langle (-1)^{m+1},\det B \rangle \varepsilon_B & \ \text{if $n=2m$.}
\end{cases}\]
(cf. [\cite{Ike3}, Lemma 2.1]). 
We make the convention that $\xi_B=1, \frke_B=0$ and $\eta_B=1$ if $B$ is the empty matrix.
Once for all, we fix an additive character $\psi$ of $F$ of order zero, that is, a character such that
\[\frko =\{ a \in F \ | \ \psi(ax)=1 \ \text{ for  any} \ x \in \frko \}.\]
  For  a half-integral matrix $B$ of degree $n$ over $\frko$ define the local Siegel series $b_{\frkp}(B,s)$ by 
\[b_{{\frkp}}(B,s)= \sum_{R} \psi({\rm tr}(BR))\mu(R)^{-s},\]
where $R$ runs over a complete set of representatives of ${\rm Sym}_n(F)/{\rm Sym}_n(\frko)$ and $\mu(R)=[R\frko^n+\frko^n:\frko^n]$. 

Now for a non-degenerate half-integral matrix $B$ of degree $n$ over $\frko $ define a polynomial $\gamma_q(B,X)$ in $X$ by 
\[\gamma_q(B,X)=
\begin{cases}
(1-X)\prod_{i=1}^{n/2}(1-q^{2i}X^2)(1-q^{n/2}\xi_B X)^{-1} & \text{ if $n$ is  even} \\ 
(1-X)\prod_{i=1}^{(n-1)/2}(1-q^{2i}X^2) & \text{ if $n$ is  odd.} \end{cases}\] 
Then it is shown by \cite{Sh1} that there exists a polynomial $F_{\frkp}(B,X)$ in $X$ with coefficients in $\ZZ$ such that 
\[F_{\frkp}(B,q^{-s})={b_{\frkp}(B,s) \over \gamma_q(B,q^{-s})}.\]
 We define a symbol $X^{1/2}$ so that $(X^{1/2})^2=X$.
We define $\widetilde F_{\frkp}(B,X)$ as
\[\widetilde F_{\frkp}(B,X)=X^{-\frke_B/2}F(B,q^{-(n+1)/2}X).\]
We note that $\widetilde F_{\frkp}(B,X) \in \QQ[q^{1/2}][X,X^{-1}]$ if $n$ is even, and
$\widetilde F_{\frkp}(B,X) \in \QQ[X^{1/2},X^{-1/2}]$ if $n$ is odd.
We sometimes write $F_{\frkp}(B,X)$ and $\widetilde F_{\frkp}(B,X)$ as $F(B,X)$ and $\widetilde F(B,X)$, respectively.

\bigskip
The following proposition is due to [\cite{Ike3}, Theorem 4.1]. 

\begin{proposition} 
\label{prop.2.1} 
We have 
\[\widetilde F(B,X^{-1})=\zeta_B\widetilde F(B,X),\]
where $\zeta_{B}=\eta_B$ or 1 according as $n$ is odd or even.
\end{proposition}

\section{The Gross-Keating invariant and  related invariants}
\label{sec:1}

We first recall the definition of the Gross-Keating invariant \cite{G-K} of a quadratic form over $\frko$.

For two matrices $B, B'\in\calh_n(\frko)$, we sometimes write $B\sim B'$ if $B$ and $B'$ are $GL_n(\frko)$-equivalent. 
The $GL_n(\frko)$-equivalence class of $B$ is denoted by $\{B\}$.
Let $B=(b_{ij}) \in \calh_n(\frko)^{\rm nd}$. 
Let $S(B)$ be the set of all non-decreasing sequences $(a_1, \ldots, a_n)\in\ZZn$ such that
\begin{align*}
\ord(b_i)&\geq a_i, \\
\ord(2 b_{ij})&\geq (a_i+a_j)/2\qquad (1\leq i,j\leq n).
\end{align*}
Set
\[
S(\{B\})=\bigcup_{B'\in\{B\}} S(B')=\bigcup_{U\in\GL_n(\frko)} S(B[U]).
\]
The Gross-Keating invariant (or the GK-invariant for short) $\ua=(a_1, a_2, \ldots, a_n)$ of $B$ is the greatest element of $S(\{B\})$ with respect to the lexicographic order $\succ$ on $\ZZn$.
Here, the lexicographic order $\succ$ is, as usual, defined as follows.
For $(y_1, y_2, \ldots, y_n),  (z_1, z_2, \ldots, z_n)\in \ZZ_{\geq 0}^n$, let $j$ be the largest integer such that $y_i=z_i$ for $i<j$.
Then $(y_1, y_2, \ldots, y_n)\succ  (z_1, z_2, \ldots, z_n)$ if $y_j>z_j$.
The Gross-Keating invariant  is denoted by  $\GK(B)$.
A sequence of length $0$ is denoted by $\emptyset$.
When $B$ is a matrix of degree $0$, we understand $\GK(B)=\emptyset$.

By definition, the Gross-Keating invariant $\GK(B)$ is determined only by the $GL_n(\frko)$-equivalence class of $B$.
We say that $B\in\calh_n(\frko)$ is an optimal form if $\GK(B)\in S(B)$.
Let $B \in \calh_n(\frko)$. Then $B$ is $GL_n(\frko)$-equivalent to an optimal form $B'$.
Then we say that $B$ has an optimal decomposition $B'$. 
We say that $B \in \calh_n(\frko)$ is a diagonal Jordan form if $B$ is expressed as 
\[B=\vpi^{a_1} u_1 \bot \cdots \bot \vpi^{a_n}u_n\]
with $a_1 \le \cdots \le a_n$ and $u_1,\cdots,u_n \in \frko^{\times}$.
 Then, in the non-dyadic case,  the diagonal Jordan form $B$ above  is optimal, and
$\GK(B)=(a_1,\ldots,a_n)$.
Therefore, the diagonal Jordan decomposition is an optimal decomposition. However, in the dyadic case, not all half-integral symmetric matrices have a diagonal Jordan decomposition, and  the Jordan decomposition is not necessarily an optimal decomposition.
\begin{definition}
\label{def.3.1} Let $\ua=(a_1,\ldots,a_n)$ be a non-decreasing sequence of non-negative integers. Write $\ua$ as
\[\ua=(\underbrace{m_1,\ldots,m_1}_{n_1},\ldots,\underbrace{m_r,\ldots,m_r}_{n_r})\]
with $m_1<\cdots<m_r$ and $n=n_1+\cdots+n_{r-1}+n_r$.
For $s=1,2,\ldots,r$ put
\[n_s^\ast=\sum_{u=1}^sn_u,\]
and
\[I_s=\{n_{s-1}^\ast+1,n_{s-1}^\ast+2,\ldots,n_s^\ast\}.\]
\end{definition}
\begin{definition} 
\label{def.3.2}
Let $B \in \calh_n(\frko)^{\rm nd}$ with $\GK(B)=(a_1,\ldots,a_n)$, and $n_1,\ldots,n_r,n_1^\ast,\ldots,n_r^\ast$ and $m_1,\ldots,m_r$ be those in Definition \ref{def.3.1}. Take an optimal decomposition $C$ of $B$, and for $s=1,\ldots,r$ we put
\[\zeta_s(C)=\zeta(C^{(n_s^\ast)}),\]
where 
$\zeta(C^{(n_s^\ast)})=\xi_{C^{(n_s^\ast)}}$ or $\zeta(C^{(n_s^\ast)})=\eta_{C^{(n_s^\ast)}}$ according as $n_s^\ast$ is even or odd.
Then $\zeta_s(C)$ does not depend on the choice of $C$ (cf. [\cite{Ike-Kat}, Theorem 0.4]), which will be denoted by $\zeta_s=\zeta_s(B)$.
Then we define $\EGK(B)$ as $\EGK(B)=(n_1,\ldots,n_r;m_1,\ldots,m_r;\zeta_1,\ldots,\zeta_r)$,
and we call it  the extended GK datum of $B$.
\end{definition}
From now on, until the end of this section, we assume that $q$ is even. We denote by $\frkS_n$ the symmetric group of degree $n$. Recall that a permutation $\sig\in \frkS_n$ is an involution if $\sig^2=\mathrm{id}$.

\begin{definition}
\label{def.3.3}
For an involution $\sig \in \frkS_n$ and a non-decreasing sequence $\ua=(a_1,\ldots,a_n)$ of non-negative integers , we set
\begin{align*}
\calp^0&=\calp^0(\sigma)=\{i\,| 1\leq i\leq n,\;  i=\sig(i)\}, \\
\calp^+&=\calp^+(\sigma)=\{i\,| 1\leq i\leq n,\;  a_i>a_{\sig(i)}\}, \\
\calp^-&=\calp^-(\sigma)=\{i\,| 1\leq i\leq n,\;  a_i<a_{\sig(i)}\}. 
\end{align*}
We  say that an involution $\sig\in\frkS_n$ is an $\underline{a}$-admissible involution if the following two conditions are satisfied.
\begin{itemize}
\item[(i)] 
$\calp^0$ has at most two elements.
If $\calp^0$ has two distinct elements $i$ and $j$, then $a_i\not\equiv a_j \text{ mod $2$}$. 
Moreover, if $i \in  I_s\cap \calp^0$, then $i$ is the maximal element of $I_s$, and
\[i=\max\{j \ | \ j \in \calp^0 \cup \calp^+ , a_j \equiv a_i \text{ mod } 2 \}.\]
\item[(ii)]
For $s=1, \ldots, r$,  there is at most one element in $I_s\cap\calp^-$.
If $i \in  I_s\cap\calp^-$, then $i$ is the maximal element of $I_s$ and
\[
\sig(i)=\min\{j\in \calp^+ \,| \, j>i,\, a_j\equiv a_i \text{ mod } 2\}.
\]
\item[(iii)]
For $s=1, \ldots, r$,  there is at most one element in $I_s\cap\calp^+$.
If $i \in  I_s\cap\calp^+$, then $i$ is the minimal element of $I_s$ and
\[
\sig(i)=\max\{j\in \calp^- \,| \, j<i,\, a_j\equiv a_i \text{ mod } 2\}.
\]
\item[(iv)]
If $a_i=a_{\sig(i)}$, then $|i -\sig(i)| \le 1$. 
\end{itemize}
\end{definition}
This is called a standard $\underline{a}$-admissible involution in \cite{Ike-Kat}, but in this paper we omit the word ``standard'', since we do not consider an $\underline{a}$-admissible involution which is not standard.

\begin{definition}
\label{def.3.4}
For $\ua=(a_1, \ldots, a_n)\in\ZZ_{\geq 0}^n$, put
\begin{align*}
\calm(\ua)&=\left\{B=(b_{ij})\in\calh_n(\frko)\,\vrule\, \begin{array}{ll}\ord(b_{ii})\geq a_i, \\ \ord(2 b_{ij})\geq (a_i+a_j)/2 \  (1\leq i < j\leq n)\end{array} \right\},  \\
\calm^0(\ua)&=\left\{B=(b_{ij})\in\calh_n(\frko)\,\vrule\, \begin{array}{ll}\ord(b_{ii}) > a_i, \\ \ord(2 b_{ij}) > (a_i+a_j)/2 \ (1\leq i < j\leq n)\end{array} \right\}.
\end{align*}
\end{definition}

\begin{definition}
\label{def.3.5}
Let $\sig\in\frkS_n$ be an $\ua$-admissible involution.
We say that $B=(b_{ij})\in \calm(\ua)$ is a reduced form with GK-type $(\ua, \sig)$ if the following conditions are satisfied.
\begin{itemize}
\item[(1)]  If $i\notin\calp^0$ and $j=\sig(i)$, then 
\[\ord(2 b_{i \,j})=\frac{a_i+a_j}{2} \]
\item[(2)] If $i\in\calp^0 \cup \calp^- $, then
\[
\ord(b_{ii})=a_i.
\]
\item[(3)] If $j\neq i, \sig(i)$, then
\[
\ord(2 b_{ij})>\frac{a_i+a_j}2.
\]
\end{itemize}
\end{definition}
We often say that $B$ is a reduced form with GK-type $\ua$ without mentioning $\sig$.
We formally think of a matrix of degree $0$ as a reduced form with GK-type $\emptyset$.
The following theorems are fundamental in our theory.
\begin{theorem}
\label{th.3.1}
{\rm ([\cite{Ike-Kat}, Corollary 5.1])} Let $B$ be a reduced form of $\GK$ type $(\ua,\sigma)$. Then we have $\GK(B)=\ua$.
\end{theorem}
\begin{theorem}
\label{th.3.2} 
{\rm  (cf. [\cite{Ike-Kat}, Theorem 4.3])}  
Assume that $\mathrm{GK}(B)=\underline{a}$ for $B\in \calh_n(\frko)^\mathrm{nd}$.
Then $B$ is $GL_n(\frko)$-equivalent to a reduced form of $\GK$ type $(\underline{a}, \sigma)$ for some   $\underline{a}$-admissible involution $\sigma$.
\end{theorem}
By Theorem \ref{th.3.2}, any non-degenerate  half-integral symmetric matrix $B$ over $\frko$ is $GL_n(\frko)$-equivalent to  a reduced form $B'$. Then we say that $B$ has a reduced decomposition $B'$.

The following theorem plays an important role in proving our main result.
\begin{theorem} 
\label{th.3.3}
Let $B$ and $B'$ be  elements of $\calh_n({\mathfrak o})$.  Assume that
 $B$ is a reduced  form of $\GK$-type $(\ua,\sig)$ and that 
$B'-B \in \calm^0(\ua)$. Then $B'$ is also a reduced  form of the same type as $B$.
Moreover we have $\EGK(B')=\EGK(B)$.  
\end{theorem}
\begin{proof} The first assertion can easily be proved. The second assertion follows from [\cite{Ike-Kat}, Proposition 3.3]. 
\end{proof}

\section{Laurent polynomial attached to EGK datum}
We recall the definition of naive EGK datum (cf. \cite{Ike-Kat}).  Let ${\mathcal Z}_3=\{0,1,-1 \}$. 
\begin{definition} 
\label{def.4.1}
An element $(a_1,\ldots,a_n;\vep_1,\ldots,\vep_n)$
of $\ZZ_{\ge 0}^n \times {\mathcal Z}_3^n$ is said to be a naive  EGK datum of length $n$  if the following conditions hold:
\begin{itemize}
\item [(N1)] $a_1 \le \cdots \le a_n$.
\item [(N2)] Assume that $i$ is even. Then $\vep_i \not=0$ if and only if $a_1+\cdots+a_i$ is even.
\item [(N3)] Assume that $i$ is odd. Then $\vep_i \not=0$. 
\item[(N4)]  $\vep_1=1$.
\item[(N5)]  Let $i  \ge 3$ be an odd integer and assume that $a_1+\cdots + a_{i-1}$ is even.  Then $\vep_i=\vep_{i-1}^{a_i+a_{i-1}}\vep_{i-2}$.
\end{itemize}
We denote by  $\mathcal{NEGK}_n$ the set of all naive $\EGK$ data of length $n$. 
\end{definition}
\begin{definition} 
\label{def.4.2} 
For integers $e,\widetilde e$,  a real number $\xi$, and $i=0,1$ define  rational functions $C(e,\widetilde e,\xi;Y,X)$ 
and $D(e,\widetilde e,\xi;Y,X)$ in $Y^{1/2}$ and $X^{1/2}$ by 
\[C(e,\widetilde e,\xi;Y,X)={Y^{\widetilde e/2}X^{-(e- \widetilde e)/2-1}(1-\xi Y^{-1} X)  \over X^{-1}-X} \]
and
\[D(e,\widetilde e,\xi;Y,X)= {Y^{\widetilde e/2}X^{-(e-\widetilde e)/2}   \over 1- \xi X} .\]
\end{definition}
For a positive integer $i$  put 
\[C_i(e,\widetilde e,\xi;Y,X)= \begin{cases}
C(e,\widetilde e,\xi;Y,X) &  \text { if  $i$ is even } \\
D(e,\widetilde e,\xi;Y,X) & \text{ if $i$ is odd.}
\end{cases}.\]
\begin{definition}
\label{def.4.3}
For a sequence $\underline a=(a_1,\ldots,a_n)$ of integers and an integer $1 \le i \le n$, we define $\frke_i=\frke_i(\underline a)$ as
\[\frke_i=
\begin{cases} a_1+\cdots +a_i  & \text{ if  $i$ is odd} \\
2[(a_1+\cdots+a_i)/2] & \text{ if $i$ is even.}
\end{cases}\]
 We also put $\frke_0=0$.
\end{definition}
  
\begin{definition}
\label{def.4.4}
For a naive EGK datum $H=(a_1,\ldots,a_n;\vep_1,\ldots,\vep_n)$ we define a rational function $\calf(H;Y,X)$ in $X^{1/2}$ and $Y^{1/2}$ as follows:
First we define
\[\calf(H;Y,X)=X^{-a_1/2}+X^{-a_1/2+1}+\cdots+X^{a_1/2-1}+X^{a_1/2}\]
if $n=1$. Let  $n>1$. Then $H'= (a_1,\ldots,a_{n-1};\vep_1,\ldots,\vep_{n-1})$ is a naive EGK datum of length $n-1$. 
Assume that $\calf(H';Y,X)$ is defined for $H'$. Then, we define $\calf(H;Y,X)$ as
\begin{align*}
&\calf(H;Y,X)=C_n(\frke_n,\frke_{n-1},\xi;Y,X)\calf(H';Y,YX)\\
&+\zeta C_n(\frke_n,\frke_{n-1},\xi;Y,X^{-1})\calf(H';Y,YX^{-1}),
\end{align*}
where $\xi=\vep_n$ or $\vep_{n-1}$ according as $n$ is even or odd, and $\zeta=1$ or $\vep_n$ according as $n$ is even or odd.
\end{definition}
By the definition of $\calf(H;Y,X)$ we easily see the following.

\begin{proposition}
\label{prop.4.1}
Let $H=(a_1,\ldots,a_n;\vep_1,\ldots,\vep_n)$ be a naive $\EGK$ datum of length $n$. Then we have
\begin{align*}
& \calf(H;Y,X) \\
&=\sum_{(i_1,\ldots,i_n) \in \{\pm 1 \}^n} \eta_n^{(1-i_n)/2} C_{n}(\frke_n,\frke_{n-1},\xi_n;Y,X^{i_n})\\
& \times \prod_{j=1}^{n-1}\eta_j^{(1-i_j)/2} C_{j}(\frke_j,\frke_{j-1},\xi_j;Y,Y^{i_j+i_ji_{j+1}+\cdots+i_ji_{j+1}\cdots i_{n-1}}X^{i_j\cdots i_n}),
\end{align*}
where 
\[\xi_j=\begin{cases} \vep_j & \text{ if $j$ is even} \\
\vep_{j-1} & \text{ if $j$ is odd},
\end{cases}\]
and
\[\eta_j=\begin{cases} 1 & \text{ if $j$ is even} \\
\vep_j & \text{ if $j$ is odd}
\end{cases}\]
 for $1 \le j \le n$.  In particular, 
\[\calf(H;Y,X^{-1})=\eta_n \calf(H;Y,X).\]
\end{proposition}
\bigskip
\bigskip
\begin{proposition} 
\label{prop.4.2} Let $H=(a_1,\ldots,a_n;\vep_1,\ldots,\vep_n)$ be a naive $\EGK$ datum of length $n$. Then
 $\calf(H;Y,X)$ is a  Laurent polynomial in $X^{1/2}$ with coefficients in  $\ZZ[Y,Y^{-1}]$. 
\end{proposition} 

\begin{proof} Put $\calf(H;Y,X)'=X^{\frke_n/2}\calf(H;Y,X)$. It suffices to show that $\calf(H;Y,X)'$ is a polynomial in $X$ with coefficients in $\ZZ[Y,Y^{-1}]$. By definition, we easily see that
$\calf(H;Y,X)'$ belongs to $\QQ(X,Y)$. Moreover, by Proposition \ref{prop.4.1},  $\calf(H;Y,X)'$ can be expressed as
\[ \calf(H;Y,X)'=\frac{P(X,Y)}{Q(X,Y)}\]
with $P(X,Y),Q(X,Y) \in \ZZ[X,Y,Y^{-1}]$ such that $Q(X,Y)$ is a monic polynomial in $X$ with coefficients in $\ZZ[Y,Y^{-1}]$.
By [\cite{Ike-Kat}, Remark 6.1, Proposition 6.3
], and  [\cite{Kat1}, Theorem 4.3], $Q[X,p^{1/2}]$ divides $P[X,p^{1/2}]$ for any odd prime $p$  in $\QQ[X,p^{1/2}]$.
It follows that $Q(X,Y) $ divides $P(X,Y)$ in $\ZZ[X,Y,Y^{-1}]$. This proves the assertion.
\end{proof}

\begin{proposition} 
\label{prop.4.3}
Let $H=(a_1,\ldots,a_n;\vep_1,\ldots,\vep_n)$ be a naive  $\EGK$ datum of length $n$ and
$H''=(a_1,\ldots,a_{n-2};\vep_1,\ldots,\vep_{n-2})$. Then $H''$ is a naive $\EGK$ datum of length $n-2$. Assume that $a_{n-1}=a_n$.
 Then  the following assertions hold.
\begin{itemize}
\item [(1)] Assume that $n$ is odd  and  $a_1+\cdots +a_{n-1}$ is even. Then we have
\begin{align*}
\calf(H;Y,X)&=Y^{\frke_{n-2}-1} \Biggl \{ {X^{(-\frke_n+\frke_{n-2})/2 -1}  \over (YX)^{-1}-YX}\calf(H'';Y,Y^2X)\\
&+{\vep_n X^{(\frke_n-\frke_{n-2})/2 +1}  \over (YX^{-1})^{-1}-YX^{-1}}\calf(H'';Y,Y^2X^{-1})\Biggr \}\\
&+ {Y^{\frke_{n-1}} (Y^2-Y^{-2}) \vep_n \over ((YX)^{-1}-YX)((YX^{-1})^{-1}-YX^{-1})}\calf(H'';Y,X) .
\end{align*}
In particular, $\calf(H;Y,X)$ does not depend on $\vep_{n-1}$.\\
\item [(2)] Assume that  $n$ is even and  $a_1+\cdots +a_n$ is odd. Then we have  
\begin{align*}
\calf(H;Y,X)&=Y^{\frke_{n-2}} \Biggl \{{X^{(-\frke_n+\frke_{n-2})/2 -1}  \over X^{-1}-X}\calf(H'';Y,Y^2X)\\
&+{X^{(\frke_n- { e}_{n-2})/2 +1}  \over X-X^{-1}}\calf(H'';Y,Y^2X^{-1})\Biggr \}.
\end{align*}
In particular, $\calf(H;Y,X)$ does not depend on $\vep_{n-1}$.
\end{itemize}

\end{proposition}
\begin{proof}
We note that $(H')'=H''$, and hence $H''$ is a naive EGK datum of length $n-2$.
Assume that $n$ is odd and $a_1+\cdots +a_{n-1}$ is even. Hence, we have
\begin{align*}
\calf(H;Y,X) &=C_n(\frke_n,\frke_{n-1},\vep_{n-1};Y,X)\\
& \times \Biggl \{C_{n-1}(\frke_{n-1},\frke_{n-2},\vep_{n-1};Y,YX)\calf(H'';Y,Y^2X) \\
& +C_{n-1}(\frke_{n-1},\frke_{n-2},\vep_{n-1};Y,(YX)^{-1})\calf(H'';Y,X^{-1})\Biggr \}\\
&+\vep_n C_n(\frke_n,\frke_{n-1},\vep_{n-1};Y,X^{-1})\\
&\times \Biggl \{C_{n-1}(\frke_{n-1},\frke_{n-2},\vep_{n-1};Y,YX^{-1})\calf(H'';Y,Y^2X^{-1})\\
& +C_{n-1}(\frke_{n-1},\frke_{n-2},\vep_{n-1};Y,(YX^{-1})^{-1})\calf(H'';Y,X)\Biggr \},
\end{align*}
By definition, we have $\vep_{n-2}=\vep_n$, and by Proposition \ref{prop.4.1}, we have
\[\calf(H'',Y,X^{-1})=\vep_n\calf(H'';Y,X).\]
Then the assertion (1) can be proved by a direct calculation.
Similarly the assertion (2) can be proved.
\end{proof}

\bigskip
Now we recall the definition of EGK datum (cf. \cite{Ike-Kat}).

\begin{definition} 
\label{def.4.5}
 Let $G=(n_1,\ldots,n_r;m_1,\ldots,m_r;\zeta_1,\ldots,\zeta_r)$ be an element 
of $\ZZ_{>0}^r \times \ZZ_{\ge 0}^r  \times 
{\mathcal Z}_3^r$. Put $n_s^\ast=\sum_{i=1}^s n_i$ for $s \le r$.  We say that $G$ is an EGK datum of length $n$  if the following conditions hold:
\begin{itemize}
\item [(E1)] $n_r^\ast=n$  and $m_1 <\cdots <m_r$.
\item[(E2)] Assume that $n_s^\ast$ is even. Then $\zeta_s \not=0$ if and only if $m_1n_1+\cdots+m_sn_s$ is even.
\item [(E3)] Assume that $n_s^\ast$ is odd. Then $\zeta_s \not=0$. Moreover we have
\begin{itemize}
\item[(a)] Assume that  $n_i^\ast$  is even for any $i<s$.  Then
\[\zeta_s=\zeta_{s-1}^{m_s+m_{s-1}}\cdots\zeta_2^{m_2+m_1} \zeta_1^{m_2+m_1}.\]
In particular, $\zeta_1=1$ if $n_1$ is odd.
\item[(b)] Assume that $n_1m_1+\cdots+(n_{s-1}-1)m_{s-1} $ is even and that $n_i^\ast$ is odd for some $i<s$. Let $t<s$ be the largest number such that $n_t^\ast$ is odd. Then
\[\zeta_s=\zeta_{s-1}^{m_s+m_{s-1}}\cdots \zeta_{t+2}^{m_{t+3}+m_{t+2}} \zeta_{t+1}^{m_{t+2}+m_{t+1}}  \zeta_t.\]
In particular, $\zeta_s=\zeta_t$ if $t=s-1$.
\end{itemize}
\end{itemize}
We denote by  $\mathcal{EGK}_n$ the set of all $\EGK$ data of length $n$.
\end{definition}                                        
\bigskip

\begin{definition}
\label{def.4.6.}
Let $H=(a_1,\ldots,a_n;\vep_1,\ldots,\vep_n)$ be an naive $\EGK$ datum of length $n$, and
$n_1,\ldots,n_r,n_1^\ast,\ldots,n_r^\ast$ and $m_1,\ldots,m_r$ be those defined in Definition \ref{def.3.1}. 
Then put $\zeta_s=\vep_{{n_s}^\ast}$ for $s=1,\ldots,r$. 
Then $G_H:=(n_1,\ldots,n_r;m_1,\ldots,m_r;\zeta_1,\ldots,\zeta_r)$ is an $\EGK$ datum (cf. [\cite{Ike-Kat}, Proposition 6.2]). We then define a mapping $\Ups_n$ from $\mathcal{NEGK}_n$ to $\mathcal{EGK}_n$ by $\Ups_n(H)=G_H$.
\end{definition}
The  mapping $\Ups_n$ is surjective (cf. [\cite{Ike-Kat}, Proposition 6.3]), but it is  not injective in general.

\begin{proposition} 
\label{prop.4.4}
Let $G\in\mathcal{EGK}_n$ be an $\EGK$ datum.
Choose $H=(a_1,\ldots, a_n;\vep_1,\ldots,\vep_n)\in\mathcal{NEGK}_n$ such that $\Ups_n(H)=G$ and put $H'=(a_1,\ldots,a_{n-1};\vep_1,\ldots,\vep_{n-1})$ and  $G'=\Ups_{n-1}(H')$.
Then $G'$ is uniquely determined by $G$ except in the following two cases:
\begin{itemize}
\item[(Case 1) \hskip -15pt] \hskip 15pt $n\geq 3$ is odd, $a_{n-1}=a_n$, and $a_1+\cdots+a_{n-1}$ is even.
\item[(Case 2) \hskip -15pt] \hskip 15pt $n$ is even, $a_{n-1}=a_n$, and $a_1+\cdots+a_n$ is odd.
\end{itemize}
Moreover, in these exceptional cases, put $H''=(a_1,\ldots,a_{n-2};\vep_1,\ldots,\vep_{n-2})$ and  $G''=\Ups_{n-2}(H'')$.
Then $G''$ is uniquely determined by $G$.
\end{proposition}
\begin{proof}
For the first part, it is enough to prove $\vep_{n-1}$ is uniquely determined by $G$ except in  (Case 1) and (Case 2).
If $a_{n-1}<a_n$, then the assertion is obvious.
Assume that $a_{n-1}=a_n$.
If both $n\geq 3$ and $a_1+\cdots+a_{n-1}$ are odd, then we have $\vep_{n-1}=0$ by (N2).
Assume that both $n$ and $a_1+\cdots+a_n$ are even.
Write $G=(n_1^\ast,\ldots,n_r^\ast;m_1,\ldots,m_r;\zeta_1,\ldots,\zeta_r)$.
Then we have
\[
\vep_{n-1}=
\begin{cases}
\zeta_{r-1} \ & \text{ if $n_r$ is  odd,} \\
\zeta_{r-1}^{m_i+m_{r-1}}\cdots \zeta_1^{m_2+m_1} \ & \text{ if $n_1,\ldots,n_r$ are  even,}\\
\zeta_{r-1}^{m_r+m_{r-1}}\cdots  \zeta_{k+1}^{m_{k+2}+m_{k+1}}\zeta_k \
 & \text{ if $n_r$ is even and}  
\\
{ } &  \  n_{k+1} \text{ is  odd  with  some $k \le  r -2$.} 
\end{cases}
\]
by (E3) and (N5).
The latter part can be proved in a similar way.
\end{proof}

\begin{theorem} 
\label{th.4.1}
Let $G$ be an $\EGK$ datum of length $n$, take
$H\in \Ups_n^{-1}(G)$. Then the Laurent polynomial  $\calf(H;Y,X)$ in $X^{1/2}, Y$ is uniquely determined by
 $G$, and does not depend on the choice of $H$. 
\end{theorem}
 
\begin{proof} We prove the assertion by the induction on $n$. The assertion holds for $n=1$. 
Let $n >1$ and assume that the assertion holds for any $l<n$.  
Write
\[G=(n_1,\ldots,n_r;m_1,\ldots,m_r;\zeta_1,\ldots,\zeta_r), \quad
H=(a_1,\ldots, a_n;\vep_1,\ldots,\vep_n).
\]
We note that $(a_1,\ldots, a_n)$ is uniquely determined by $G$, and $\vep_n=\zeta_r$. 
For a positive integer $i \le n$, let $\frke_i=\frke_i(\widetilde m)$ be that defined in Definition \ref{def.4.3}. 
Then except in (Case 1) and (Case 2) in Proposition \ref{prop.4.4}, we have
\begin{align*}
\calf(H;Y,X)=&C_n(\frke_n,\frke_{n-1},\xi;Y,X)\calf(H';Y,YX)\\
&\quad +\zeta C_n(\frke_n,\frke_{n-1},\xi;Y,X^{-1})\calf(H';Y,YX^{-1}) 
\end{align*}
where $H'$ is as in Proposition \ref{prop.4.4} and 
\begin{align*}
\zeta=&
\begin{cases}
1 & \text{ if $n$ is even} \\
\vep_n & \text{ if $n$ is odd,}
\end{cases} 
\\\
\xi=&
\begin{cases}
\vep_{n-1}  & \text{ if $n$ is odd} \\
\vep_n & \text{ if $n$ is even.}
\end{cases} 
\end{align*}
By the induction assumption, $\calf(H',Y,X)$ is uniquely determined by $G'=\Ups_{n-1}(H')$, and hence $\calf(H;Y,X)$ is uniquely determined by $G$ by Proposition \ref{prop.4.4}.

Next, we consider (Case 1), i.e.,  $n\geq 3$ is odd, $a_{n-1}=a_n$,  and $a_1+\cdots a_{n-1}$ is even. 
Then, by (2) of Proposition \ref{prop.4.3}, we have
\begin{align*}
\calf(H;Y,X)&=Y^{\frke_{n-2}-1} \Biggl \{ {X^{(-\frke_n+\frke_{n-2})/2 -1}  \over (YX)^{-1}-YX}\calf(H'',Y^2X)\\
&+{X^{(\frke_n-\frke_{n-2})/2 +1}  \over (YX^{-1})^{-1}-YX^{-1}}\calf(H'';Y,Y^2X^{-1})\Biggr \}\\
&+ {Y^{\frke_{n-1}}(Y^2-Y^{-2})  \vep_r \over ((YX)^{-1}-YX)((YX^{-1})^{-1}-YX^{-1})}\calf(H'';Y,X).
\end{align*}
By the induction assumption, $\calf(H'',Y,X)$ is uniquely determined by $G''=\Ups_{n-2}(H'')$, and hence $\calf(H;Y,X)$ is uniquely determined by $G$ by Proposition \ref{prop.4.4}.

Now we consider (Case 2), i.e., $n$ is even, $a_{n-1}=a_n$, and $a_1+\cdots +a_n$ is odd.  
Then,  by (1) of Proposition \ref{prop.4.3}, we have
\begin{align*}
\calf(H;Y,X)&=Y^{\frke_{n-2}} \Biggl \{{X^{(-\frke_n+\frke_{n-2})/2-1} \over X^{-1} -X}\calf(H'';Y,Y^2X)\\
&+{X^{(\frke_n-\frke_{n-2})/2+1} \over X -X^{-1}}\calf(H'';Y,Y^2X^{-1}) \Biggr \}.
\end{align*}
By the induction assumption, $\calf(H'',Y,X)$ is uniquely determined by $G''=\Ups_{n-2}(H'')$, and hence $\calf(H;Y,X)$ is uniquely determined by $G$ by Proposition \ref{prop.4.4}.
\end{proof}

For an EGK datum $G$ we define $\widetilde \calf(G;Y,X)$ as $\calf(H;Y,X)$, where $H$ is a naive $\EGK$ datum of length $n$ such that $\Ups_n(H)=G$. For later purpose we recall the following theorem. 

\begin{theorem} 
\label{th.4.2}
{\rm (cf. [\cite{Ike-Kat}, Theorem 6.1])} Let  $B \in \calh_n(\frko )^{\rm{nd}}$. Then $\EGK(B)$ is an $\EGK$ datum of length $n$.
\end{theorem}
The following proposition will be used later.
\begin{proposition}
\label{prop.4.5} Assume that $q$ is even. 
Let $B\in\calh_n(\frko)$ be a reduced form of $\GK$ type $(\ua,\sigma)$, where $\ua=(a_1, \ldots, a_n)$.
\begin{itemize}
\item [(1)] Assume that $B^{(n-1)}$ is a reduced form with $GK(B^{(n-1)})=\ua^{(n-1)}=(a_1, \ldots, a_{n-1})$.
Then there exists $H=(\underline{a};\vep_1, \ldots, \vep_n)\in\mathcal{NEGK}_n$ such that $\Ups_n(H)=\EGK(B)$ and $\Ups_{n-1}(H')=\EGK(B^{(n-1)})$.
\item [(2)] Assume that $\GK(B)$ satisfies the condition either in ${\rm (Case 1)}$ or ${\rm (Case2)}$ of Proposition \ref{prop.4.4} and $\sigma(a_{n-1})=a_n$. Then $B^{(n-2)}$ is a reduced form with $\GK(B^{(n-2)})=\ua^{(n-2)}=(a_1, \ldots, a_{n-2})$, and there exists $H=(\underline{a};\vep_1, \ldots, \vep_n)\in\mathcal{NEGK}_n$ such that $\Ups_n(H)=\EGK(B)$ and $\Ups_{n-2}(H'')=\EGK(B^{(n-2)})$.
\end{itemize}
\end{proposition}
\begin{proof}
Put $\EGK(B)=(n_1^\ast,\ldots,n_r^\ast;m_1,\ldots,m_r;\zeta_1,\ldots,\zeta_r)$.\\
(1) The assertion follows from {\rm (cf. [\cite{Ike-Kat}, Theorem 0.4])} if $a_{n-1}=a_n$.
In the case that  $n\geq 3$ is odd, $a_{n-1}=a_n$, and  $a_1+\cdots+a_{n-1}$ is odd,  we have $\vep_{n-1}=\xi_{B^{(n-1)}}=0$.
Similarly, in the case that  $n$ is even, $a_{n-1}=a_n$, and $a_1+\cdots+a_n$ is even, we have 
\[
\vep_{n-1}=
\begin{cases}
\zeta_{r-1} \ & \text{ if $n_r$ is  odd,} \\
\zeta_{r-1}^{m_i+m_{r-1}}\cdots \zeta_1^{m_2+m_1} \ & \text{ if $n_1,\ldots,n_r$ are  even,}\\
\zeta_{r-1}^{m_r+m_{r-1}}\cdots  \zeta_{k+1}^{m_{k+2}+m_{k+1}}\zeta_k \
 & \text{ if $n_r$ is even and}  
\\
{ } &  \  n_{k+1} \text{ is  odd  with  some $k \le  r -2$.} 
\end{cases}
\]
by Proposition \ref{prop.4.4}.
By (E3), we have $\eta_{B^{(n-1)}}=\vep_{n-1}$.

In (Case 1) of Proposition \ref{prop.4.4}, choose any $H=(\underline{a};\vep_1, \ldots, \vep_n)\in\mathcal{NEGK}_n$ such that $\Ups_n(H)=\EGK(B)$.
Then we have
\[
(\underline{a};\vep_1, \ldots, \vep_{n-2}, 1, \vep_n), \, (\underline{a};\vep_1, \ldots, \vep_{n-2}, -1, \vep_n)\in\mathcal{NEGK}_n.
\]
Thus one can find $H=(\underline{a};\vep_1, \ldots, \vep_n)\in\mathcal{NEGK}_n$ such that $\Ups_n(H)=\EGK(B)$ and $\vep_{n-1}=\xi_{B^{(n-1)}}$.
Thus we have $\Ups_{n-1}(H')=\EGK(B^{(n-1)})$. Similarly, the assertion holds in (Case 2) of Proposition \ref{prop.4.4}.\\
(2) Let $H''=(a_1,\ldots,a_{n-2};\vep_1,\ldots,\vep_{n-2})$ be a naive EGK datum of length $n-2$ such that $\Ups_{n-2}(H'')=\EGK(B^{(n-2)})$. 
Assume that $\GK(B)$ satisfies the condition in (Case 1). Then put $\vep_n=\vep_{n-2}$, and take $\vep_{n-1}= \pm1$ arbitrary. Then $H:=(\ua;\vep_1,\ldots,\vep_n)$ is a naive EGK data. Moreover,
by [\cite{Ike-Kat}, Lemma 6.2 (2)], we have $\vep_n=\eta_{B}$, and hence we have $\Ups_n(H)=\EGK(B)$. This proves the assertion. 
Similarly the assertion holds in (Case 2).

\end{proof}

\section{Induction formulas of the local densities and Siegel series}
Let $dY$ be the Haar measure of ${\rm Sym}_n(F)$ normalized so that 
\[\int_{{\rm Sym}_n(\frko )}dY =1.\]
 For a measurable subset $C$ of ${\rm Sym}_n(F)$, we define the volume ${\rm vol}(C)$ of $C$ as
\[{\rm vol}(C)=\int_C dY.\]
We also normalize the Haar measure $dX$ of $M_{mn}(F)$ so that
\[\int_{M_{mn}(\frko )}dX =1.\]
Let $m,n$ and $r$ be non-negative integers such that $m \ge n-r >0$.
For an element $S \in \calh_m(\frko )^{\rm{nd}}, A \in \calh_r(\frko ), R \in  M_{r,n-r}(\frko )$, and $ T \in \calh_{n-r}(\frko )$, we define the modified local density with respect to $S,T,R$, and $A$ as 
\[\cala_e(S,T,R,A)=\{(X_1,X_2) \in M_{m,n-r}(\frko ) \times M_{r,n-r}(\frko ) \ | \]
\[ S[X_1]-\left(\begin{matrix} T & 2^{-1}{}^tR \\ 2^{-1}R & A \end{matrix}\right) \left[\begin{matrix} 1_{n-r} \\ X_2 \end{matrix}\right] \in \vpi^e\calh_{n-r}(\frko ) \},\]
and 
\[\alpha_{\frkp}(S,T,R,A)=\lim_{e \rightarrow \infty}q^{e(n-r)(n-r+1)/2}{\rm vol}(\cala_e(S,T,R,A)).\]
We make the convention that $\alpha_{\frkp}(S,T,R,A)=1$ if $n=r$. 
As for the existence of the above limit, see Lemma \ref{lem.5.1} and Theorem \ref{th.5.1}. 
The modified local density $\alpha_{\frkp}(S,T,R,A)$ can also be expressed as the following improper integral:
\[\alpha_{\frkp}(S,T,R,A)\]
\[=\lim_{e \rightarrow \infty} \int_{\frkp^{-e}{\rm Sym}_{n-r}(\frko )}\int_{M_{m,n-r}(\frko ) \times 
M_{r,n-r}(\frko )} \psi({\rm tr}(Y(F(X_1,X_2)))) dX_1dX_2dY,\]
where  
\[F(X_1,X_2)=S[X_1]-\left(\begin{matrix} T & 2^{-1}{}^tR \\ 2^{-1}R & A \end{matrix}\right) \left[\begin{matrix} 1_{n-r} \\ X_2 \end{matrix}\right].\]
We sometimes write the above improper integral as 
\[\int_{{\rm Sym}_{n-r}(F)}\int_{M_{m,n-r}(\frko ) \times 
M_{r,n-r}(\frko )} \psi({\rm tr}(Y(F(X_1,X_2)))) dX_1dX_2dY.\]
In the case $r=0$, we write $\alpha_{\frkp}(S,T,R,A)$ as $\alpha_{\frkp}(S,T)$, and call it the local density representing $T$ by $S$. We have
\[\alpha_{\frkp}(S,T)=\lim_{e \rightarrow \infty} q^{en(n+1)/2}{\rm vol}(\cala_e(S,T)),\]
where
\[\cala_e(S,T) = \{ X =(x_{ij}) \in M_{mn}(\frko ) \ | \ S[X]  - T \in \vpi^e\calh_n(\frko ) \}.\]
The local density can also be expressed as 
\[\alpha_{\frkp}(S,T)=\int_{{\rm Sym}_n(F)}\int_{M_{mn}(\frko )}\psi({\rm tr}(Y(S[X]-T)))dXdY.\]
An element $X$ of $M_{mn}(\frko )$ with $m \ge n$ is said to be primitive if there is
a matrix $Y \in M_{m,m-n}(\frko )$ such that $\left(\begin{matrix} X & Y \end{matrix}\right) \in GL_m(\frko )$.
We denote by $M_{mn}(\frko )^\mathrm{{prm}}$ the subset of $M_{mn}(\frko )$ consisting of all primitive matrices.
Let $m,n$ and $l$ be non-negative integers such that $m \ge n \ge l \ge 1$. 
For  $S \in \calh_m(\frko )^{\rm{nd}}$ and $ T \in \calh_n(\frko )$,
put 
\[\calb_e(S,T)^{(l)} = \{ X =(x_{ij}) \in \cala_e(S,T) \ |  \ (x_{ij})_{1 \le i \le m,n-l+1 \le j \le n} \text{ is  primitive} \},\]
and we define the modified primitive local density  
\[ \beta_{\frkp}(S,T)^{(l)} = \lim_{e \rightarrow \infty} q^{en(n+1)/2}{\rm vol}(\calb_e(S,T)^{(l)}).\]
In particular put  
\[\beta_{\frkp}(S,T)=\beta_{\frkp}(S,T)^{(n)},\]
and call it the primitive local density. We make the convention that $\calb_e(S,T)^{(0)}=\cala_e(S,T)$ and $\beta_{\frkp}(S,T)^{(0)}=\alpha_{\frkp}(S,T)$.  We can also express  $\beta_{\frkp}(S,T)^{(l)}$ as
\[
\beta_{\frkp}(S,T)^{(l)} =
\int_{\mathrm{Sym}_n(F)}
\int_{M_{m,n-l}(\frko )}
\int_{M^\mathrm{prm}_{m,l}(\frko )}
\psi(\mathrm{tr} (Y(S\left[\left(\begin{matrix} X_1 & X_2 \end{matrix}\right) \right]-T)))\, dX_2\,dX_1\, dY.
\]

\bigskip

Now let $H_k = \overbrace{H \bot \cdots \bot H}^k$ with 
$H =\left(\begin{matrix}
             0        & 1/2  \\
             1/2        & 0 
                                   \end{matrix}\right)$.
We note that 
\[b_{\frkp}(B,k)=\alpha_{\frkp}(H_k,B)\]
for any positive integer $k \ge n$. 

\bigskip
\begin{definition}
\label{def.5.1}
Let $T \in \calh_{n-r}(\frko )^{\rm{nd}}$.  For $X=(x_{ij}) \in M_{r,n-r}(\frko )$ and $A \in \calh_r(\frko ), R \in  M_{r,n-r}(\frko )$  we define a matrix $T(R,A,X)$ by 
\[T(R,A,X)=
\begin{pmatrix} T & {}^t\!R/2 \\  R/2 & A\end{pmatrix}\left[
\begin{pmatrix} 1_{n-r} \\ X \end{pmatrix} \right]=
T+ A[X]+ \frac{1}2 \left({}^t\! RX+ {}^t\!X R\right).\]
\end{definition}

\bigskip

For $B \in \calh_n(\frko)^{\rm nd}$, we denote by $\frkl_B$ the least integer such that $p^{\frkl_B}B^{-1} \in \calh_n(\frko)$.

\begin{lemma} 
\label{lem.5.1}
Let $S \in \calh_m(\frko)^{{\rm nd}}$. 
\begin{itemize}
\item[(1)]Let $n$ and $r$ be non-negative integers such that $m \ge n-r \ge 0$. Let $T \in \calh_{n-r}(\frko )$, and $A \in \calh_r(\frko ), R \in  M_{r,n-r}(\frko )$. 
Assume that there is a positive integer $l_0=l_0(T,R,A)$ depending only on $T,R,A$ such that 
\[\ord(\det (2T(R,A,X))) \le l_0(T,R,A)\]
for any $X \in M_{r,n-r}(\frko )$. Put $m_0=2l_0+1$. 
Then the limit $\alpha_\frkp(S,T,R,A)$  exists, and 
\[\alpha_\frkp(S,T,R,A)=q^{-er(n-r)}\sum_{X \in M_{r,n-r}(\frko )/\frkp^{e}M_{r,n-r}(\frko )} \alpha_{\frkp}(S,T(R,A,X))\]
for any $e \ge m_0$. Here we use the same symbol $X$ to denote the coset of $X \in M_{r,n-r}(\frko)$ modulo ${\frkp}^e$.
\item[(2)] Let $B \in \calh_n(\frko)$ with $m \ge n$.
\begin{itemize}
\item[(2.1)] Assume that $B$ is non-degenerate. Then for any $0 \le r \le n$ the limit $\beta_\frkp(S,B)^{(r)}$ exists and
\[\beta_\frkp(S,B)^{(r)}=q^{en(n+1)/2}{\rm vol}(\calb_{e}(S,B)^{(r)})\]
for any integer $e \ge 2\ord(\det (2B))+1$.
\item[(2.2)] Assume that $2S \in GL_m(\frko)$. Then $\beta_\frkp(S,B)$ exists and
\[\beta_\frkp(S,B)=q^{en(n+1)/2}{\rm vol}(\calb_{e}(S,B))\]
 for any $e \ge 1$.
\item[(2.3)] Assume that $B$ is non-degenerate. Then,  there is a positive integer $\lambda_{\frkl_B}$ depending on $\frkl_B$ satisfying the following condition:\\
If $A \in \calh_m(\frko)$ satisfies $A \equiv B \ {\rm mod} \ \vpi^{\lambda_{\frkl_B}} S_m(\frko)$, 
$A$ is $GL_m(\frko)$-equivalent to $B$.

In particular if $q$ is odd, we can take $\frkl_B+1$ as $\lambda_{\frkl_B}$.
 \end{itemize}
 \end{itemize}
\end{lemma}

\begin{proof}  
(1) For each $e$  put ${\bf A}_e= q^{e(n-r)(n-r+1)/2}{\rm vol}(\cala_e(S,T,R,A))$. Then we have
\[{\bf A}_e=\int_{M_{r,n-r}(\frko)} q^{e(n-r)(n-r+1)/2}{\rm vol}(\cala_e(S,T(R,A,X_2)))dX_2.\]
If $X_2 \equiv X_2' \in {\frkp}^eM_{nr}(\frko)$, then ${\rm vol}(\cala_e(S,T(R,A,X_2)))={\rm vol}(\cala_e(S,T(R,A,X_2')))$. Hence we have
\[{\bf A}_e=\sum_{X \in M_{r,n-r}(\frko)/{\frkp}^e M_{r,n-r}(\frko)}\int_{X+{\frkp}^eM_{r,n-r}(\frko)} q^{e(n-r)(n-r+1)/2}{\rm vol}(\cala_e(S,T(R,A,X)))dX\]
\[=q^{-er(n-r)}\sum_{X \in M_{r,n-r}(\frko)/{\frkp}^e M_{r,n-r}(\frko)}q^{e(n-r)(n-r+1)/2}{\rm vol}(\cala_e(S,T(R,A,X))).\]
Let $e \ge m_0$. Then, by using the same argument as in the proof of [\cite{Sh1}, Proposition 14.3], we can prove that
\[\alpha_{\frkp}(S,T(R,A,X))=q^{e(n-r)(n-r+1)/2}{\rm vol}(\cala_e(S,T(R,A,X))).\]
Hence we have
\[{\bf A}_e=q^{-er(n-r)}\sum_{X \in M_{r,n-r}(\frko )/\frkp^{e}M_{r,n-r}(\frko )} \alpha_{\frkp}(S,T(R,A,X)).\]
We also have
\[{\bf A}_{e+1}=q^{-(e+1)r(n-r)}\sum_{X \in M_{r,n-r}(\frko )/\frkp^{e+1}M_{r,n-r}(\frko )} \alpha_{\frkp}(S,T(R,A,X)).\]
We have  $\alpha_{\frkp}(S,T(R,A,X))= \alpha_{\frkp}(S,T(R,A,X'))$ if $X \equiv X' \text{ mod } \frkp^eM_{r,n-r}(\frko)$.
Hence
\[{\bf A}_{e+1}=q^{-(e+1)r(n-r)}q^{r(n-r)} \sum_{X \in M_{r,n-r}(\frko )/\frkp^{e}M_{r,n-r}(\frko )} \alpha_{\frkp}(S,T(R,A,X))={\bf A}_e.\]
This proves the assertion.\\
(2) The assertion (2.1) can be proved by using the same argument as in the proof of [\cite{Sh1}, Proposition 14.3].
The assertions (2.2) and (2.3) can easily be proved.
\end{proof}
\begin{lemma}
\label{lem.5.2}  
Let $S \in \calh_m(\frko)^{{\rm nd}}, T \in \calh_n(\frko)$ with $m \ge n$, and let $l$ be a non-negative integer such that $l \le n$.  Assume that 
the $\beta_{\frkp}(S,T)^{(l)}$ exists. Then, for any $W \in M_n(\frko)^{{\rm nd}}$,  we have
\begin{align*}
&\int_{{\rm Sym}_n(F)} \int_{M_{m,n-l}(\frko)} \int_{M_{ml}^{\mathrm{prm}}(\frko)} \psi({\rm tr}(YW(S[\left(\begin{matrix} X_1 & X_2 \end{matrix}\right)]-T)))dX_2dX_1dY
\\
&=|\det W|_{\frkp}^{-n-1}\beta_{\frkp}(S,T)^{(l)}.
\end{align*}
\end{lemma}
\begin{proof}
For a subset $\cals$ of ${\rm Sym}(F)$ and an element $V \in GL_n(F)$ put
\[\cals[V]=\{ X[V] \ | \ X \in \cals \}.\]
Then we have
\begin{align*}
& \int_{{\rm Sym}_n(F)} \int_{M_{m,n-l}(\frko)} \int_{M_{ml}^{\mathrm{prm}}(\frko)} \psi({\rm tr}(YW(S[\left(\begin{matrix} X_1 & X_2 \end{matrix}\right)]-T)))dX_2dX_1dY \\ 
& = \lim_{e \rightarrow \infty}  \int_{{\frkp}^{-e} {\rm Sym}_n(\frko)} \int_{M_{m,n-l}(\frko)} \int_{M_{ml}^{\mathrm{prm}}(\frko)} \psi({\rm tr}(YW(S[\left(\begin{matrix} X_1 & X_2 \end{matrix}\right)]-T)))dX_2dX_1dY \\
&=|\det W|_{\frkp}^{-n-1} \\
& \times \lim_{e \rightarrow \infty}  \int_{{\frkp}^{-e} {\rm Sym}_n(\frko)[W] } \int_{M_{m,n-l}(\frko)} \int_{M_{ml}^{\mathrm{prm}}(\frko)} \psi({\rm tr}(Y(S[\left(\begin{matrix} X_1 & X_2 \end{matrix}\right)]-T)))dX_2dX_1dY \\
&=|\det W|_{\frkp}^{-n-1} \\
& \times  \int_{{\rm Sym}_n(F) } \int_{M_{m,n-l}(\frko)} \int_{M_{ml}^{\mathrm{prm}}(\frko)} \psi({\rm tr}(Y(S[\left(\begin{matrix} X_1 & X_2 \end{matrix}\right)]-T)))dX_2dX_1dY.
\end{align*}
This proves the assertion.
\end{proof}

\begin{proposition} 
\label{prop.5.1}
 Let $A$ and $B$ be non-degenerate half-integral matrices of degree $m$ and $n$, respectively, over $\frko $ such that $m \ge n$. Then  we have
\[\alpha_{\frkp}(A,B)=\sum_{W \in GL_l(\frko ) \backslash M_l(\frko)^{{\rm nd}}} q^{\ord(\det W) (-m+n+1)}\beta_{\frkp}(A,B[1_{n-l} \bot W^{-1}])^{(l)}.\]
\end{proposition}
\begin{proof}
We have 
\begin{align*}
&\alpha_{\frkp}(A,B) \\
=&\int_{\mathrm{Sym}_n(F)} \int_{M_{m,n-l}(\frko )} \int_{M_{m,l}(\frko )} \psi(\mathrm{tr} (Y(A\left[\left(\begin{matrix} X_1 & X_2 \end{matrix}\right) \right]-B)))dX_2dX_1dY \\
=&\int_{\mathrm{Sym}_n(F)} \sum_{W\in GL_l(\frko )\backslash M_l(\frko)^{{\rm nd}}} |\det W|_{\frkp}^m  \\
&\times \int_{M_{m,n-l}(\frko )} \int_{M^\mathrm{prm}_{m,l}(\frko )} \psi(\mathrm{tr} (Y\!A\left[\left(\begin{matrix} X_1 & X_2\!W \end{matrix}\right) \right]- Y\!B)) dX_2dX_1dY \\
=& \sum_{W\in GL_l(\frko )\backslash M_l(\frko)^{{\rm nd}}} |\det W|_{\frkp}^{m}   \\
&\times \int_{\mathrm{Sym}_n(F)} \int_{M_{m,n-l}(\frko )}\int_{M^\mathrm{prm}_{m,l}(\frko )}\psi\left(\mathrm{tr}(Y[1_{n-l} \bot \,{}^t W]A\left[\left(\begin{matrix} X_1 & X_2 \end{matrix}\right) \right]-Y\!B)\right)
dX_2dX_1dY \\
=& \sum_{W\in GL_l(\frko )\backslash M_l(\frko)^{{\rm nd}}}|\det W|_{\frkp}^{m}     \\
&\times \int_{\mathrm{Sym}_n(F)} \int_{M_{m,n-l}(\frko )}\int_{M^\mathrm{prm}_{m,l}(\frko )}\psi\left(\mathrm{tr}( Y[1_{n-l} \bot \,{}^t W]g(X_1,X_2))\right)dX_2dX_1dY ,
\end{align*}
with $g(X_1,X_2)=A\left[\left(\begin{matrix} X_1 & X_2 \end{matrix}\right) \right]-B[1_{n-l} \bot W^{-1}]$. Then the assertion follows from Lemma \ref{lem.5.2}.

\end{proof}

\begin{definition}
\label{def.5.2}
Put  ${\bf D}_{l,i}= GL_l(\frko) (\vpi 1_i \bot 1_{l-i}) GL_l(\frko)$  for $0 \le i \le l$.  
We define the function $\pi_l$ on $M_l(\frko)^{{\rm nd}}$ as
\[\pi_l(W) =q^{i(i-1)/2}(-1)^i  \text{ for }   W \in {\bf D}_{l,i}\]
and
\[\pi_l(W)=0  \text{ if } W\not\in \bigcup_{i=0}^l {\bf D}_{l,i}.\]
\end{definition}

Then, using the same argument as in the proof of [\cite{Ki1}, Theorem 3.1], we obtain the following.

\bigskip

\begin{corollary} 
\label{cor.5.1}
  Let $A$ and $B$ be non-degenerate half-integral matrices of degree $m$ and $n$, respectively, over $\frko $ such that $m \ge n$. Then  we have
\[\sum_{W \in GL_l(\frko) \backslash M_l(\frko)^{{\rm nd}}} \pi_l(W)q^{(-m+n+1)\ord(\det W)}\alpha_{\frkp}(A,B[1_{n-l} \bot W^{-1}])=\beta_{\frkp}(A,B)^{(l)}.\]
\end{corollary}

Let $M$ be a free module of rank $n$ over $\frko$, and $Q$ a non-degenerate quadratic form on $M$ with values in $\frko$.
The pair $(M, Q)$ is called a quadratic module over $\frko$.
The symmetric bilinear form $(x, y)=(x, y)_Q$ associated to $Q$ is defined by
\[
 (x, y)_Q=Q(x+y)-Q(x)-Q(y), \qquad x, y\in M.
\]
When there is no fear of confusion, we simply denote $(x, y)$.
We denote by  $\mathrm{s}(M)$  the fractional ideal of $F$ generated by  $\{(x,y)\,|\, x,y\in M\}$, and call it the scale of $M$.
For a basis $\{z_1, \ldots, z_n\}$, we define a matrix $B=(b_{ij})\in\calh_n(\frko)$ by
\[
 b_{ij}=\frac12 (z_i,z_j).
\]
Two matrices are $GL_n(\frko)$-equivalent if and only if the associated quadratic modules are isomorphic.

An $\frko$-submodule $M'$ of a free $\frko$-free module $M$ is said to be primitive if $M'$ is a direct summand of $M$. Let $\{u_1,\ldots,u_m\}$ be a basis of $M$, and let $\{v_1,\ldots,v_l\}$ be a basis of a submodule $M'$.
Write 
\[v_j=\sum_{i=1}^m a_{ij} u_j \ (j=1,\ldots, l).\]
Then $M'$ is primitive if and only if $(a_{ij})_{1 \le i \le m, 1 \le j \le l}$ is primitive.  
For an element $B \in \calh_n(\frko)$ we denote by $\langle B \rangle$ the quadratic module $(M,Q)$ with a basis $\{z_1,\ldots,z_n\}$ such that
\[\left({1 \over 2}(z_i,z_j)\right)_{1 \le i,j \le n}=B.\]
We note that the isomorphism class of $\langle B \rangle$ is uniquely determined by the $GL_n(\frko)$-equivalence class of $B$.
To prove an induction formula for the local densities, we first prove the following lemma:

\bigskip
\begin{lemma} 
 \label{lem.5.3}
Let $k$ and $r$  be  positive integers such that $k \ge r$. Let $(M,Q)$ be a quadratic module over $\frko $ with a basis  $\{z_1,\ldots, z_{2k}\}$
such that 
\[\left({1 \over 2} (z_i,z_j)\right)_{1 \le i,j \le 2k}=H_k .\]
Let $z_{2k-r+1}',\ldots,z_{2k}'$ be  elements of $M$, and put 
\[A=\left({1 \over 2}(z_{2k-r+i}',z_{2k-r+j}')\right)_{1 \le i,j \le r}.\]
Assume that 
$\sum_{i=1}^r\frko  z_{2k-r+i}'$ is primitive, and $A \in \frkp\calh_r(\frko )$. Then there exist  elements $z_{2k-2r+1}',\ldots,z_{2k-r}'$ of $M$ and submodules $M_1'$ and $M_2'$ of $M$  such that 
\begin{itemize}
\item[(1)] $M_2'=\sum_{i=1}^{2r}\frko z_{2k-2r+i}'$ and $M_2'= \left\langle \mattwo(O;2^{-1}1_r;2^{-1}1_r;A) \right\rangle$. 
\item[(2)] $M_1'$ is isometric to $\langle H_{k-r} \rangle$ 
\item[(3)] $M= M_1' \bot M_2'$.
\end{itemize}
\end{lemma}

{\it Proof.} Since $\sum_{i=1}^r\frko  z_{2k-2r+2i}'$ is primitive, $A \in \frkp\calh_r(\frko )$, and $M$ is an orthogonal sum of hyperbolic spaces, there exist  elements $z_{2k-2r+1}',\ldots,z_{2k-r}'$ of $M$ such that 
\[\left({1 \over 2}(z_{2k-2r+i}',z_{2k-2r+j}')\right)_{1 \le i,j \le 2r}=\mattwo(O;2^{-1}1_r;2^{-1}1_r;A).\]
Hence the submodule 
\[M_2'=\sum_{i=1}^{2r}\frko z_{2k-2r+i}'\]
satisfies the condition (1). Moreover since $s(M)={1 \over 2}\frko $ and 
\[{1 \over 2}(u,v) \in {1 \over 2}\frko \]
for any $u \in M_2'$ and $v \in M$, Hence we have
\[M=M_2'^{\bot} \bot M_2'.\]
The module $M_2'$ is isometric to $\langle H_r \rangle$ and $M$ is isometric to $\langle H_k \rangle$. Hence, by the cancellation theorem, there exists  a submodule $M_1'$ of $M$ such that 
\[M_1' \cong \langle H_{k-r}\rangle,\]
and
\[M=M_1' \bot M_2'.\]
This proves the assertion.

\bigskip

\begin{remark}
A similar assertion has been proved in [\cite{Kat1}, Lemma 2.3]. 
\end{remark}

\bigskip
\begin{corollary} 
\label{cor.5.2} Let $\Xi \in M_{2k,r}^{\mathrm{prm}}(\frko)$ such that $H_k[\Xi] \in \frkp\calh_r(\frko)$. 
 Then there is an element $U \in GL_{2k}(\frko )$ of such that 
\begin{itemize}
\item[(1)] $H_k[U]=H_{k-r} \bot \mattwo(O;2^{-1}1_r;2^{-1}1_r;H_k[\Xi])$,
\item[(2)] $U^{-1}\Xi=\left(\begin{matrix} O \\ 1_r \end{matrix}\right)$.
\end{itemize}
\end{corollary}
\begin{lemma}
\label{lem.5.4}
Let $k$ and $r$ be positive integers such that $k \ge r$. Then for any $T \in \frkp \calh_r(\frko)$,  we have
\[\beta_{\frkp}(H_k,T)=(1-q^{-k})(1+q^{-k+r})\prod_{i=1}^{r-1}(1-q^{-2k+2i}).\]
\end{lemma}
\begin{proof} 
By (2.2) of  Lemma \ref {lem.5.1}, we have
\begin{align*}
& \beta_{\frkp}(H_k,T) \\
&=q^{r(r+1)/2}{\rm vol}(\calb_1(H_k,T))\\
&=q^{-2k+r(r+1)/2}\#\{X \in M_{2k,r}(\frko)/\frkp M_{2k,r}(\frko) \ | \ X \in \calb_1(H_k,T)\}
\end{align*}
Thus the assertion follows from [\cite{Ki2}, Lemma 5.6.9].
\end{proof}

\bigskip
\begin{theorem} 
\label{th.5.1} Let $k,n,r$ be positive integers such that 
$k \ge (n+r)/2$ and $n \ge r$. 
 Let $B=\left(\begin{matrix} T & 2^{-1}{}^t R \\ 2^{-1}R & A \end{matrix}\right) \in \calh_n(\frko )^{\rm{nd}}$  with
$T \in \calh_{n-r}(\frko ), R \in M_{r,n-r}(\frko )$ and $A \in \calh_{r}(\frko )$. Then the limit $\alpha_{\frkp}(H_{k-r},T,\vpi R,\vpi^2 A)$ exists and 
\[\beta_{\frkp}(H_k,B[1_{n-r} \bot \vpi 1_r])^{(r)}=\beta(H_k,\vpi^2 A) \alpha_{\frkp}(H_{k-r},T,\vpi R,\vpi^2 A).\]
\end{theorem}

\begin{proof} The above theorem can be proved in the same manner as [\cite{Kat1}, Proposition 2.4]. (See also the proof of [\cite{Kat-Kaw1}, Lemma 3.2].) But for the sake of completeness we give another proof. Put $B'=B[1_{n-r} \bot \vpi 1_r]$ and 
write $X \in M_{2k,n}(\frko )$ and $Y \in {\rm Sym}_n(F)$ as $X=\left(\begin{matrix} X_1 & X_2 \end{matrix}\right)$ with $X_1 \in M_{2k,n-r}(\frko ), X_2 \in M_{2k,r}(\frko )$,
and $Y=\mattwo(Y_{11};Y_{12};{}^tY_{12};Y_{22})$ with $Y_{11} \in {\rm Sym}_{n-r}(F), Y_{22} \in {\rm Sym}_r(F)$, and
$Y_{12} \in M_{n-r,r}(F)$. For each non-negative integer $e$ and $X_2 \in M_{2k,r}^{\mathrm{prm}}(\frko )$ put
\begin{align*}
I_{X_2,e}=&\int_{{\frkp}^{-e}{\rm Sym}_{n-r}(\frko) \times {\frkp}^{-e}M_{n-r,r}(\frko)} \int_{M_{2k,n-r}(\frko )}
\psi({\rm tr}(Y_{11}(H_k[X_1]-T)))\\
 &\times \psi({\rm tr}(Y_{12}(2 {}^t\! X_2H_kX_1-\vpi R)))dX_1dY_{11}dY_{12}.
 \end{align*}
Then 
\begin{align*}
\label{eq.a}  & q^{en(n+1)/2}{\rm vol}(\calb_e(H_k,B')^{(r)}) \tag{*}\\
 = &\int_{{\frkp}^{-e}{\rm Sym}_r(\frko)}\int_{M_{2k,r}^{\mathrm{prm}}(\frko )}\psi({\rm tr}(Y_{22}( H_k[X_2]-\vpi^2 A))) I_{X_2,e}   dX_2dY_{22} 
\end{align*}
We shall show that
\[I_{\Xi,e}=q^{e(n-r)(n-r+1)/2}\mathrm{vol}(\cala_e(H_{k-r},T,\vpi R,\vpi^2A))\]
for any $\Xi  \in \calb_e(H_k,\vpi^2A)$. Let $U$ be the matrix in Corollary \ref{cor.5.2}, and write
\[U^{-1}\left(\begin{matrix} X_1 & \Xi \end{matrix}\right)=\left(\begin{matrix} Y_1  & O\\
                         Y_2  & O \\
                         Y_3  & 1_r
                                   \end{matrix}\right),\]
with $Y_1 \in M_{2k-2r,n-r}(\frko ), Y_2,Y_3 \in M_{r,n-r}(\frko )$. 
Then, by the change of variable given by 
\[M_{2k,n-r}(\frko) \ni X_1 \longrightarrow (Y_1,Y_2,Y_3)  \in M_{2k-2r,n-r}(\frko ) \times M_{r,n-r}(\frko ) \times M_{r,n-r}(\frko ),\]
we have
\begin{align*}
 I_{\Xi,e}&=\int_{{\frkp}^{-e}{\rm Sym}_{n-r}(\frko) \times {\frkp}^{-e}M_{n-r,r}(\frko)} \Bigl(\int_{M_{2k-2r,n-r}(\frko ) \times M_{r,n-r}(\frko ) \times M_{r,n-r}(\frko ) } \\
& \psi(\mathrm{tr}(Y_{11}(H_{k-r}[Y_1]+2^{-1}({}^tY_2Y_3+{}^tY_3Y_2)+H_k[\Xi Y_3]-T)))\\
&\times \psi(Y_{12}(Y_2+2H_k[\Xi]Y_3-\vpi R))dY_1dY_2dY_3\Bigr) dY_{11}dY_{12}.
\end{align*}
Since  $H_k[\Xi] \equiv \vpi^2 A \text{ mod } \vpi^e\calh_r(\frko)$, we have
\begin{align*}
 I_{\Xi,e}&=\int_{{\frkp}^{-e}{\rm Sym}_{n-r}(\frko) \times {\frkp}^{-e}M_{n-r,r}(\frko)} \Bigl(\int_{M_{2k-2r,n-r}(\frko ) \times M_{r,n-r}(\frko ) \times M_{r,n-r}(\frko ) } \\
& \psi({\rm tr}(Y_{12} W)) \psi({\rm tr}(Y_{11}g(Y_1,Y_2,Y_3))) dY_1dY_2dY_3\Bigr)dY_{11}dY_{12},
\end{align*}
where
\[g(Y_1,Y_2,Y_3)=H_{k-r}[Y_1]+2^{-1}{}^tY_2Y_3+2^{-1}{}^tY_3Y_2+\vpi^2 A[Y_3] -T,\]
and 
\[W=Y_2+2\vpi^2 AY_3 -\vpi R.\]
By the change of variables given by $(Y_1,Y_2,Y_3) \mapsto (Y_1,W,Y_3),$
we have
\begin{align*}
 I_{\Xi,e} &=\int_{{\frkp}^{-e}{\rm Sym}_{n-r}(\frko) \times {\frkp}^{-e}M_{r,n-r}(\frko)} \Bigl(\int_{M_{2k-2r,n-r}(\frko ) \times M_{r,n-r}(\frko ) \times M_{r,n-r}(\frko )}\\
&\psi({\rm tr}(Y_{11}f(Y_1,Y_3) )) \psi({\rm tr}((Y_{11}{}^tY_3+Y_{12})W))dY_1dW dY_3\Bigr) dY_{11}dY_{12},
\end{align*}
where 
\begin{align*}
f(Y_1,Y_3)&=H_{k-r}[Y_1]- \vpi^2 A[Y_3]-2^{-1}\vpi {}^tR Y_3-2^{-1}\vpi {}^tY_3R-T\\
&=H_{k-r}[Y_1]-B'\left[\begin{pmatrix} 1_{n-r} \\ Y_3 \end{pmatrix}\right].
\end{align*}
Put $Z_{12}=Y_{11}\,{}^t Y_3+Y_{12}$.
Then we have
\begin{align*}
 I_{\Xi,e}&=\int_{{\frkp}^{-e}{\rm Sym}_{n-r}(\frko) } \int_{M_{2k-2r,n-r}(\frko )  \times M_{r,n-r}(\frko )}
\psi({\rm tr}(Y_{11}f(Y_1,Y_3))) dY_1dY_3dY_{11}\\
& \times  \int_{{\frkp}^{-e} M_{n-r,r}(\frko)} \int_{M_{r,n-r}(\frko )}  \psi({\rm tr}(Z_{12}W))dZ_{12}dW .
\end{align*}
We have
\[\int_{{\frkp}^{-e} M_{n-r,r}(\frko)} \int_{M_{r,n-r}(\frko )}  \psi({\rm tr}(Z_{12}W))dZ_{12}dW =1,\]
and hence
\[I_{\Xi,e}= q^{e(n-r)(n-r+1)/2} {\rm vol}(\cala_e(H_{k-r},T,\vpi R,\vpi^2 A))\]
for any $\Xi \in  M_{2k,r}^{\mathrm{prm}}(\frko )$. Hence by (\ref{eq.a}) we have
\begin{align*}
 &q^{en(n+1)/2}{\rm vol}(\calb_e(H_k,B')^{(r)} \\
=&q^{er(r+1)/2} {\rm vol}(\calb_e(H_k,\vpi^2 A)) q^{e(n-r)(n-r+1)/2} {\rm vol}(\cala_e(H_{k-r},T,\vpi R,\vpi^2 A)).
\end{align*}
By (2.1) of Lemma \ref{lem.5.1}, the limit $\beta_{\frkp}(H_k,B')^{(r)}$ exists,  and by Lemma \ref{lem.5.4}, the limit $\beta_{\frkp}(H_k,\vpi^2A)$ exists and non-zero. Hence, the assertion holds.
\end{proof}

\bigskip

\begin{remark} (1) The existence of $\alpha_\frkp(H_{k-r},T,\vpi R,\vpi^2A)$ does not follow from (1) of Lemma \ref{lem.5.1}
 because the condition there does not necessarily holds.

 (2) Let $F=\QQ_p$ and $R=O$. Then the above theorem is nothing but [\cite {Kat1}, Proposition 2.4].
\end{remark}

\bigskip
\begin{definition}
\label{def.5.3}
Let $T=(t_{ij}) \in \calh_n(\frko )^{\rm{nd}}$ and $n_1$ be a positive integer such that $n_1 \le n$.  Then
put $R_T^{(n_1)}=(t_{ij})_{n-n_1+1 \le i \le n,1 \le j \le n-n_1}$ and $A_T^{(n_1)}=(t_{ij})_{n-n_1+1 \le i,j \le n}$, 
and for $x \in M_{n_1,n-n_1}(\frko )$ put
\[T_x=T^{(n-n_1)}(\vpi R_T^{(n_1)},\vpi^2 A_T^{(n_1)},x),\]
where $T^{(n-n_1)}(\vpi R_T^{(n_1)},\vpi^2 A_T^{(n_1)},x)$ is the matrix in Definition \ref{def.5.1}.
Here we make the convention that $T_x$ is the empty matrix if $n_1=n$.
\end{definition}
Clearly $T_{O}=T^{(n-n_1)}$ for the  zero matrix $O$ of type  $n_1 \times (n-n_1)$. We note that 
\[T_x=T\left[\left(\begin{matrix} 1_{n-n_1} \\ \vpi x \end{matrix}\right)\right]. \]

\bigskip
\begin{theorem} 
\label{th.5.2}
 Let $B=(b_{ij})$ be an element of $\calh_n(\frko )^{{\rm nd}}$.
\begin{itemize} 
\item[(1)] Assume that there is a positive integer $l_0$ such that  
\[\ord(\det (2B_x)) \le l_0 \text{  for any } x \in M_{1,n-1}(\frko).\]
Then there is a positive integer $m_0$ such that 
\begin{align*}
& \alpha_{\frkp}(H_k, B[1_{n-1} \bot \vpi])- q^{n-2k+1}\alpha_{\frkp}(H_k,B) \\
&=\beta_{\frkp}(H_k, \vpi^2 b_{nn})\alpha_{\frkp}(H_{k-1},B^{(n-1)},(\vpi b_{n,1},\ldots, \vpi b_{n,n-1}),\vpi^2 b_{nn}) \\
& =q^{-m_0(n-1)}(1-q^{-k})(1+q^{-k+1})\sum_{x \in M_{1,n-1}(\frko )/\frkp^{m_0}  M_{1,n-1}(\frko )} \alpha_{\frkp}(H_{k-1},B_x).
\end{align*}
Here we understand that $\alpha_{\frkp}(H_{k-1},B_x)=1$ if $n=1$.
\item[(2)] Assume that there is a positive integer $l_0$ such that   
\[\ord(\det (2B_x)) \le l_0 \text{  for any } x \in M_{2,n-2}(\frko ).\] 
Then there is a positive integer $m_0$ such that 
\begin{align*}
& \alpha_{\frkp}(H_k,B[1_{n-2} \bot \vpi 1_2]) +q^{2n-4k+3}\alpha_{\frkp}(H_k,B) \\
& -q^{n-2k+1}\sum_{W \in {\bf D}_{2,1}/GL_2(\frko)} \alpha_{\frkp}(H_k,B[1_{n-2} \bot W]) \\
&=\beta_{\frkp}(H_k,\vpi^2 A)\alpha_{\frkp}(H_{k-2},B^{(n-2)},\vpi R,\vpi^2 A) \\
&=q^{-2m_0(n-2)}(1-q^{-k})(1+q^{-2k+2})(1+q^{-k+2})\\
& \times \sum_{x \in M_{2,n-2}(\frko )/\frkp^{m_0}  M_{2,n-2}(\frko )} \alpha_{\frkp}(H_{k-2},B_x),
\end{align*}
where $R=(b_{ij})_{n-1 \le i \le n, 1 \le j \le n-2}, A=\mattwo(b_{n-1,n-1};b_{n-1,n};b_{n-1,n};b_{nn})$. 
Here we understand that $\alpha_{\frkp}(H_{k-2},B_x)=1$ if $n=2$.
\end{itemize}
 \end{theorem}               

\bigskip

{\it Proof.} By Lemma \ref{lem.5.4}, for $a \in \frko $, we have
\[\beta_\frkp(H_k,\vpi^2 a)= (1-q^{-k})(1+q^{-k+1})\]
and for $A \in \calh_2(\frko )$.
We note that 
\[\frkl_{B_x} \le \ord(\det (2B_x)) \le l_0 \text{ for any } x \in M_{1,n-1}(\frko). \]
Hence, by (2.3) of Lemma \ref{lem.5.1}, there is a positive integer $c_{l_0}$ depending on $l_0$ satisfying the following condition:

If $B_{x'} \equiv B_x \text{ mod } \frkp^{c_{l_0}} S_{n-1}(\frko)$ for $x, x' \in M_{1,n-1}(\frko)$, then $B_{x'} \sim B_x$.

Put $m_0=\max(\ord(\det 2B)+1, c_{l_0})$. Then, by Lemma \ref{lem.5.1},  Corollary \ref{cor.5.1}, and Theorem \ref{th.5.1}, we see that $m_0$ satisfies the required condition in (1).  We also have
\[\beta_\frkp(H_k,\vpi^2 A)=(1-q^{-k})(1-q^{-2k+2})(1+q^{-k+2}).\]
Hence the assertion (2)  can be proved similarly.

\bigskip
\begin{corollary} 
\label{cor.5.3}
Let $B$ be as above. 
\begin{itemize}
\item[(1)] Let the notation and the assumption be as in (1) of Theorem \ref{th.5.2}.  Then 
\begin{align*}
F(B[1_{n-1} \bot \vpi],X) &=q^{n+1}X^2F(B,X)\\
&+q^{-m_0(n-1)}\sum_{x \in M_{1,n-1}(\frko )/\frkp^{m_0}  M_{1,n-1}(\frko )} K(X,x)F(B_x,qX),
\end{align*}
where
\[K(X,x)={(1-X)(1+qX)\gamma_q(B_{x},qX) \over \gamma_q(B,X)}.\]
Here we understand that $F(B_x,qX)=1$ and $\gamma_q(B_x,qX)=1$ if $n=1$.
\item[(2)] Let the notation and the assumption be as in (2) of Theorem \ref{th.5.2}. Then 
\begin{align*}
& F(B[1_{n-2} \bot \vpi 1_2],X) \\
& = -q^{2n+3}X^4 F(B,X)+q^{n+1}X^2\sum_{W \in {\bf D}_{2,1}/GL_2(\frko)} F(B[1_{n-2} \bot W],X) \\
& +q^{-2m_0(n-2)}\sum_{x \in M_{2,n-2}(\frko )/\frkp^{m_0}  M_{2,n-2}(\frko )} K(X,x)F(B_x,q^2X),
\end{align*}
where 
\[K(X,x)={(1-X)(1-q^2X^2)(1+q^2X)\gamma_q(B_{x},q^2X) \over \gamma_q(B,X)}.\]
Here we understand that $F(B_x,q^2X)=1$ and $\gamma_q(B_x,q^2X)=1$ if $n=2$.
\end{itemize}
\end{corollary}

\bigskip

Let $C(e,\widetilde e,\xi;Y,X)$ and $D(e,\widetilde e,\xi;Y,X)$ be rational functions in $X^{1/2}$ and $Y^{1/2}$ defined in Definition \ref{def.4.2}.
We often write 
\begin{align*}
C(e,\widetilde e,\xi;X)=&C(e,\widetilde e,\xi;q^{1/2},X),\\
D(e,\widetilde e,\xi;X)=&D(e,\widetilde e,\xi;q^{1/2},X)
\end{align*}
if there is no fear of confusion.  From now on we make the convention that we have $\widetilde F(A,X)=1$ if $\deg A=0$.

\bigskip
\begin{theorem} 
\label{th.5.3}
Let the notation and the assumption be as in (1) of Theorem \ref{th.5.2}. Put ${\frke}=\frke_B$ and $\widetilde{\frke}_x=\frke_{B_x}$ for $x \in M_{1,n-1}(\frko)$. 
\begin{itemize}
\item[(1)] Let $n$ be odd, and put $\xi_x=\xi_{B_x}$. Then we have
\begin{align*}
& \widetilde F(B,X) \\
&=q^{-m_0(n-1)}\sum_{x \in M_{1,n-1}(\frko )/\frkp^{m_0}  M_{1,n-1}(\frko )}\Bigl\{D(\frke,\widetilde {\frke}_x,\xi_x;X)\widetilde F(B_x,q^{1/2}X)\\
&+\eta_B D({\frke},\widetilde {\frke}_x,\xi_x;X^{-1})\widetilde F(B_x,q^{1/2}X^{-1})\Bigr\}.
\end{align*}
Here we make the convention that
$ \xi_x=1, \widetilde \frke_x=0$ and $\widetilde F(B_x,q^{1/2}X)=\widetilde F(B_x,q^{1/2}X^{-1})=1$ if $n=1$.
\item[(2)] Let $n$ be even. Then we have
\begin{align*}
& \widetilde F(B,X) \\
&=q^{-m_0(n-1)}\sum_{x \in M_{1,n-1}(\frko )/\frkp^{m_0}  M_{1,n-1}(\frko )}\Bigl\{C({\frke},\widetilde {\frke}_x,\xi_B;X)\widetilde F(B_x,q^{1/2}X)\\
&+C({\frke},\widetilde {\frke}_x,\xi_B;X^{-1})\widetilde F(B_x,q^{1/2}X^{-1})\Bigr\}.
\end{align*}
\end{itemize}
\end{theorem}

\bigskip

{\it Proof.} Put $B'=B[1_{n-1} \bot \vpi]$. Let $n$ be odd. Then, by Corollary \ref{cor.5.3},  we have
\begin{align*}
& \widetilde F(B',X)=X\widetilde F(B,X) \\
&+q^{-m_0(n-1)}\sum_{x \in M_{1,n-1}(\frko )/\frkp^{m_0}  M_{1,n-1}(\frko )} {1-X^2 \over 1- \xi_x X}q^{\widetilde \frke_x/4}X^{(-\frke+\widetilde \frke_x)/2-1}\widetilde F(B_x,q^{1/2}X).
\end{align*}
We also have
\begin{align*}
& \widetilde F(B',X^{-1})=X^{-1}\widetilde F(B,X^{-1}) \\
&+ q^{-m_0(n-1)}\sum_{x \in M_{1,n-1}(\frko )/\frkp^{m_0}  M_{1,n-1}(\frko )}{1-X^{-2} \over 1-q^{-1} \xi_x X^{-1}}q^{\widetilde \frke_x/4}X^{(\frke-\widetilde \frke_x)/2+1}\widetilde F(B_x,q^{1/2}X^{-1}).
\end{align*}
By Proposition \ref{prop.2.1} we have $\widetilde F(B',X^{-1})=\eta_B\widetilde F(B',X)$ and $\widetilde F(B,X^{-1})=\eta_B\widetilde F(B,X)$. This proves the assertion. 

Let $n$ be even.  
Then, by (2) of Corollary \ref{cor.5.3}, we have
\begin{align*}
& \widetilde F(B',X)=X\widetilde F(B,X) \\
&+ q^{-m_0(n-1)}\sum_{x \in M_{1,n-1}(\frko )/\frkp^{m_0}  M_{1,n-1}(\frko )}(1-q^{-1/2} \xi X)q^{\widetilde \frke_x/4}X^{(-\frke+\widetilde \frke_x)/2-1}\widetilde F(B_x,q^{1/2}X).
\end{align*}
We also have
\begin{align*}
 \widetilde F(B',X^{-1})&=X^{-1}\widetilde F(B,X^{-1})\\
& + q^{-m_0(n-1)}\sum_{x \in M_{1,n-1}(\frko )/\frkp^{m_0}  M_{1,n-1}(\frko )}(1-q^{-1/2} \xi X^{-1})\\
&\times q^{\widetilde \frke_x/4}X^{(\frke-\widetilde \frke_x)/2+1}\widetilde F(B_x,q^{1/2}X^{-1}).
\end{align*}
By Proposition \ref{prop.2.1} we have $\widetilde F(B',X^{-1})=\widetilde F(B',X)$ and $\widetilde F(B,X^{-1})=\widetilde F(B,X)$. This proves the assertion.

\bigskip
\begin{theorem} 
\label{th.5.4}
Let $B$ be an element of $\calh_n(\frko )$. Assume that $B^{(n-2)}$ is non-degenerate and that
 $\widetilde F(B_y,X)=\widetilde F(B^{(n-2)},X)$ for any $y \in M_{2,n-2}(\frko )$. 
Put $\frke=\frke_B$ and $\hat \frke=\frke_{B^{(n-2)}}$. 
\begin{itemize}
\item[(1)] Let $n$ be an even integer such that $n \ge 4$, and assume that $\xi_B=\xi_{B^{(n-2)}}=0$.  
\begin{itemize}
\item[(1.1)] We have
\begin{align*}
&\widetilde F(B[1_{n-2} \bot \vpi 1_2],X)=q \widetilde F(B,X)\\
&+{q^{\hat \frke/2} \over X^{-1}-X} \Bigl\{X^{(\hat \frke-\frke)/2-3}\widetilde F(B^{(n-2)},qX)(1-qX^2)\\
&-X^{(-\hat \frke+\frke)/2+3}\widetilde F(B^{(n-2)},qX^{-1})(1-qX^{-2})\Bigr\}.
\end{align*}
\item[(1.2)] Assume that $B^{(n-1)}$ is non-degenerate  and that $\ord(\det B_x)=\ord(\det B^{(n-1)})$ for any $x \in M_{1,n-1}(\frko )$. Put $\widetilde \frke=\frke_{B^{(n-1)}}$. Then 
\begin{align*}
& \widetilde F(B,X)(X^{(-\frke+2\widetilde\frke- \hat\frke)/2-2} -X^{(\frke-2\widetilde\frke+ \hat\frke)/2 +2})\\
& =\widetilde F(B[1_{n-1} \bot \vpi],X)(X^{(-\frke+2\widetilde\frke- \hat\frke)/2-1}
-X^{(\frke-2\widetilde\frke+ \hat\frke)/2+1}) \\
&+q^{\hat \frke/2}X^{\hat \frke-\widetilde \frke}\widetilde F(B^{(n-2)},qX)
-q^{\hat \frke/2}X^{-\hat \frke+\widetilde \frke}\widetilde F(B^{(n-2)},qX^{-1}).
\end{align*}
\end{itemize}

\item[(2)] Let $n$ be an odd integer such that $n \ge 3$.  Then 
\begin{align*}
& \widetilde F(B,X)=(q(X^{-1}+X))^{-1}\Bigl\{ \sum_{W \in {\bf D}_{2,1}/GL_2(\frko )}\widetilde F(B[1_{n-2} \bot W],X)  \\
& -q^{\hat \frke/2}X^{-(\frke-\hat \frke)/2-1}\widetilde F(B^{(n-2)},qX)-\eta_Bq^{\hat \frke/2}X^{(\frke-\hat \frke)/2+1}  \widetilde F(B^{(n-2)},qX^{-1}) \Bigr\}.
\end{align*}
\end{itemize}
\end{theorem}
\begin{proof} (1) We note that $\frke_{B_y}=\frke_{B^{(n-2)}}$ for any $y \in M_{2,n-2}(\frko)$, and hence there is a positive integer $l_0$ such that
\[\ord(\det (2B_y)) \le l_0 \text{ for any } y \in M_{2,n-2}(\frko).\]
Hence by  Corollary \ref{cor.5.3}, we have,  
\begin{align*}
& \widetilde F(B[1_{n-2} \bot \vpi 1_2],X)X^{-1} \\
& = -qX \widetilde F(B,X)+ \sum_{W \in {\bf D}_{2,1}/GL_2(\frko)} \widetilde F(B[1_{n-2} \bot W],X), \\
& +q^{\hat e/2}X^{\hat e/2 -e/2-3} (1-qX^2) \widetilde F(B^{(n-2)},qX),
\end{align*}
and  
\begin{align*}
& \widetilde F(B[1_{n-2} \bot \vpi 1_2],X)X \\
& = -qX^{-1} \widetilde F(B,X)+ \sum_{W \in {\bf D}_{2,1}/GL_2(\frko)} \widetilde F(B[1_{n-2} \bot W],X) \\
& +q^{\hat e/2}X^{-\hat e/2 +e/2+3} (1-qX^{-2})\widetilde F(B^{(n-2)},qX^{-1}). 
\end{align*}
Then the assertion (1.1) follows from these equality.

We prove (1.2). By the assumption, we can take a positive integer $m_0$ in (1) of Theorem \ref{th.5.3}. Then, as in  the proof of Theorem \ref{th.5.3}, we have
\begin{align*}
&\widetilde F(B[1_{n-1} \bot \vpi],X)=X\widetilde F(B,X) \\
&+ q^{-m_0(n-1)}\sum_{x \in M_{1,n-1}(\frko )/\frkp^{m_0}  M_{1,n-1}(\frko )}q^{\widetilde \frke/4}X^{(-\frke+\widetilde \frke)/2-1}\widetilde F(B_x,q^{1/2}X)
\end{align*}
for any $x \in M_{1,n-1}(\frko)$. 
By the remark after Definition \ref{def.5.3},  for any  $y \in M_{1,n-2}(\frko)$ we have
\begin{align*}
&(B_x)_y=(B\left[\left(\begin{matrix} 1_{n-1} \\ \vpi x \end{matrix} 
 \right)\right])\left[\left(\begin{matrix} 1_{n-2} \\ \vpi y \end{matrix} 
 \right)\right]=B\left[\left(\begin{matrix} 1_{n-2} \\ \vpi z \end{matrix} 
 \right)\right],
\end{align*}
where for $x=(x_1,\ldots,x_{n-1})$ and $y=(y_1,\ldots,y_{n-2})$ we define $z$ by 
\[z=\left(\begin{matrix} y_1 & \ldots & y_{n-2} \\
 x_1+\vpi x_{n-1}y_1 &\ldots & x_{n-2}+\vpi x_{n-1}y_{n-2}
 \end{matrix}\right).\] 
Hence, by the assumption we have $F((B_x)_y,X)=F(B^{(n-2)},X)$. Hence, by (1) of Theorem \ref{th.5.3}, we have
\begin{align*}
& \widetilde F(B_x,q^{1/2}X) \\
&=q^{{\widetilde e}/4}(q^{1/2}X)^{(-\widetilde \frke+\hat \frke)/2} \widetilde F(B^{(n-2)},qX)+\eta_x
q^{{\widetilde e}/4}(q^{1/2}X)^{(\widetilde \frke-\hat \frke)/2} \widetilde F(B^{(n-2)},X^{-1}),
\end{align*}
where $\eta_x=\eta_{B_x}$.
Hence
\begin{align*}
&\widetilde F(B[1_{n-1} \bot \vpi],X)=X\widetilde F(B,X) \\
&+q^{\hat \frke/2}X^{(-\frke+\hat \frke)/2-1}\widetilde F(B^{(n-2)},qX) \\
& +q^{-m_0(n-1)}\sum_{x \in M_{1,n-1}(\frko )/\frkp^{m_0}  M_{1,n-1}(\frko )}\eta_x
q^{\widetilde \frke/2}X^{(-\frke+2\widetilde \frke-\hat \frke)/2-1}\widetilde F(B^{(n-2)},X^{-1}).
\end{align*}
We also have
\begin{align*}
&\widetilde F(B[1_{n-1} \bot \vpi],X^{-1})=X^{-1}\widetilde F(B,X^{-1})\\
&+q^{\hat \frke/2}X^{(\frke-\hat \frke)/2+1}\widetilde F(B^{(n-2)},qX^{-1})\\
&+q^{-m_0(n-1)}\sum_{x \in M_{1,n-1}(\frko )/\frkp^{m_0}  M_{1,n-1}(\frko )}\eta_x
q^{\widetilde \frke/2}X^{(\frke-2\widetilde \frke+\hat \frke)/2+1}\widetilde F(B^{(n-2)},X).
\end{align*}
By Proposition \ref{prop.2.1} we have 
\[\widetilde F(B[1_{n-1} \bot \vpi],X^{-1})=\widetilde F(B[1_{n-1} \bot \vpi],X),\]
\[\widetilde F(B,X^{-1})=\widetilde F(B,X),\]
 and 
\[\widetilde F(B^{(n-2)},X^{-1})=\widetilde F(B^{(n-2)},X).\]
Thus the assertion (1.2) holds. \\
(2) The assertion (2) can be proved by using the same argument as in the proof of (1.2).
\end{proof}

\section{Proof of Theorem \ref{th.1.1}-non dyadic case- } 

In this section and the next, we give a proof of Theorem 1.1. Main tools for the proof are refined versions of Theorem \ref{th.5.3}, from which we can obtain an explicit formula of $\widetilde F(B,X)$
 for any $B \in \calh_n(\frko )$.  In this section, we assume that $q$ is odd. 

 
{\bf Proof of Theorem \ref{th.1.1}.}  We prove the induction on $n$.  If $n=1$,  by Theorem \ref{th.5.3}, we easily see that 
\[\widetilde F(B,X)=\sum_{i=0}^{a_1} X^{i-(a_1/2)}\]
This proves the assertion for $n=1$. Let $n \ge 2$ and assume that the assertion holds for $n'=n-1$. We may assume that  $B \in \calh_n(\frko )$ is  a diagonal Jordan form with 
$\GK(B)=(a_1,\ldots,a_n)$.  Then  $B^{(n-1)}$ is also a diagonal Jordan form with $\GK(B^{(n-1)})=(a_1,\ldots,a_{n-1})$.
Then, we prove the following theorem, by which we prove Theorem 1.1.

\bigskip
\begin{theorem} 
\label{th.6.1}
Under the above notation and the assumption, we have
\begin{align*}
&\widetilde F(B,X)=D(\frke_n,\frke_{n-1},\xi_{B^{(n-1)}};X)\widetilde F(B^{(n-1)},q^{1/2}X) \\
&+\eta_B D(\frke_n, \frke_{n-1},\xi_{B^{(n-1)}};X^{-1})\widetilde F(B^{(n-1)},q^{1/2}X^{-1})
\end{align*}
if $n \ge 3$ is odd, and 
\begin{align*}
&\widetilde F(B,X)=C(\frke_n,\frke_{n-1},\xi_B;X)\widetilde F(B^{(n-1)},q^{1/2}X) \\
&+ C(\frke_n, \frke_{n-1},\xi_B;X^{-1})\widetilde F(B^{(n-1)},q^{1/2}X^{-1})
\end{align*}
if $n$ is even. 
\end{theorem}

\begin{proof}  Let $n$ be odd. Then, by [\cite{Ike-Kat}, Theorem 0.1], we have $\frke_B=\frke_n$. For $x \in M_{1,n-1}(\frko)$ let $B_x$ be the matrix in Definition \ref{def.5.3}, and put 
$\frke_x=\frke_{B_x}$ and $\xi_x=\xi_{B_x}$. 
We note that 
\[B_x \equiv B^{(n-1)} \text{ mod } \frkp^{\frkl_{B^{(n-1)}}+1}S_{n-1}(\frko) \text{ for any } x \in M_{1,n-1}(\frko).\]
Hence, by  (2.3) of Lemma \ref{lem.5.1}, $B_x \sim B^{(n-1)}$ and in particular  
\[\frke_x=\frke_{n-1}, \  \xi_x=\xi_{B^{(n-1)}} \text{ and } \widetilde F(B_x,X)=\widetilde F(B^{(n-1)},X)\]
for any $x \in M_{1,n-1}(\frko)$.  
Hence, by Theorem \ref{th.5.3}, we have
\begin{align*}
&\widetilde F(B,X)=D(\frke_n,\frke_{n-1},\xi_{B^{(n-1)}};X)\widetilde F(B^{(n-1)},q^{1/2}X)\\
&+\eta_B D(\frke_n, \frke_{n-1},\xi_{B^{(n-1)}};X^{-1})\widetilde F(B^{(n-1)},q^{1/2}X^{-1}).
\end{align*}
This proves the formula in the case that  $n \ge 3$ is odd. 
Similarly the induction formula can be proved  in the case that  $n \ge 2$ is even. 
\end{proof}

\section{Proof of Theorem \ref{th.1.1}-dyadic case- }
Next we must consider a more complicated case where $q$ is even. 
Throughout this section we assume that $q$ is even. 
Let $B$ be a reduced  form in $\calh_n({\mathfrak o})$ with GK-type $((a_1,\ldots,a_n),\sig)$.
Put $\ua=(a_1,\ldots,a_n)$.
 We say that $(\ua,\sig)$ belongs to category (I) if $n=\sig(n-1)$ and $a_{n-1}=a_n$.  We say that  $\sig$ belongs to category (II) if
$B$ does not belong to category (I). We note that $(\ua,\sig)$ belongs to category (II) if and only if  
 $a_{n-1}<a_n$ or $\sig(n)=n$. In particular, $(\ua,\sig)$ belongs to category (II) if $n=1$.  We also say that $B$ belongs to category (I) or (II) according as $(\ua,\sig)$ belongs to category (I) or (II).  We note that if two reduced forms are of the same GK-type, then they belong to the same category. 

\begin{definition}
\label{def.7.1}
 Let $B=(b_{ij})$ be a reduced  form in $\calh_n(\frko)$ with GK-type $(\ua,\sig)$, and put $\ua=(a_1,\ldots,a_n)$, and $\calp^0$ be the set in Definition \ref{def.3.3}. Let $n$ be odd. Then ${\calp}^0$ consists of exactly one element, which will be denoted by $i_0=i_0(B)$.  
Let  $n$ be even. Then $\#({\calp}^0)=0$ or $2$.
If $a_1+\cdots+a_n$ is odd, then ${\calp}^0$ consists of two elements, which will be denoted by $i_1=i_1(B)$ and $i_2=i_2(B)$. In this case we note that for $k=1,2$,  we have $b_{i_ki_k}=\vpi^{a_{i_k}}u_{i_ki_k}$ with $u_{i_ki_k} \in \frko ^{\times}$, and for $j \not=i_k$, we have $b_{i_kj}=2^{-1}\vpi^{[(a_{i_k}+a_j+2)/2]}u_{i_kj}$ with $u_{i_kj} \in \frko $.  We say that $(\ua,\sig)$ is good if $\calp^0$ is empty. We  also say that $B$ is a good form in this case. If $B$ is a good form, then we remark that 
\[\langle (-1)^{n/2} \det B,x \rangle =\xi_B^{\ord(x)}\]
for any $x \in F^{\times}$. We also note that $B$ is a good form if and only if both $n$  and $a_1+\cdots +a_n$ are even. 
\end{definition}
We say that an element $K$ of $\calh_2(\frko)$ is a primitive unramified binary form if $\GK(B)=(0,0)$.  
We note that $K=(c_{ij}) \in \calh_2(\frko)$ is a primitive unramified binary form if and only if $2c_{12} \in \frko^{\times}$.
The following assertion can easily be proved.

\begin{lemma} 
\label{lem.7.1} 
Let  $K=(c_{ij})$ be a  primitive unramified binary form such that $c_{11}c_{22} \in \frkp$.
Then, for any $a,b \in \frko$ such that $ab \in \frkp$, there is an element $U \in GL_2(\frko)$ such that $B[U]=\left(\begin{smallmatrix} a & 1/2 \\ 1/2 &b \end{smallmatrix}\right)$. 
 \end{lemma}
 
For a non-decreasing sequence $\ua=(a_1,\ldots,a_n)$ of integers, let $G_{\ua}$ be the group defined by 
\[
G_{\ua}=
\{g=(g_{ij})\in \GL_n(\frko) \,|\, \text{ $\ord(g_{ij})\geq (a_j-a_i)/2$, if $a_i<a_j$}\}.
\]
We  say a reduced form $B=(b_{ij})$ of GK type $(\ua, \sig)$ is a strongly reduced form if the following condition hold:
\begin{itemize}
\item[(SR)] If $i\notin \calp^0$, then $b_{ij}=0$ for $j>\max\{ i, \sig(i)\}$.
\end{itemize}
We note that a reduced form of GK type $(\ua, \sig)$ is $G_\ua$-equivalent to
 a strongly reduced form of GK type $(\ua, \sig)$ (see [\cite{Ike-Kat}, Remark 4.1]). 
Let ${\bf D}_{2,1}$ be the subset of $M_2(\frko)$ in Definition \ref{def.5.2}.

\begin{lemma} 
\label{lem.7.2} 
Let $n$ be an odd integer, and  let $B =(b_{ij}) \in \calh_n(\frko )$ be a strongly reduced form of $\GK$-type $(\ua,\sig)$ with  $\ua=(a_1,\ldots,a_n)$. 
 Assume that $B$ belongs to category {\rm (I)} and that $a_1+\cdots+a_{n-1}$ is even.  
Then for any $W \in {\bf D}_{2,1}$ the matrix $B[1_{n-2} \bot W]$ is  $GL_n(\frko)$-equivalent to  a reduced  form $C$ belonging to category {\rm (I)} of $\GK$ type $(a_1,\ldots,a_{n-2},a_{n-1}+1,a_{n-1}+1)$ such that 
$C^{(n-2)}=B^{(n-2)}$. 

\end{lemma}
\begin{proof} Take $U_1, U_2 \in GL_2(\frko)$ such that $W=U_1\smallmattwo(1;0;0;\vpi)U_2$. Then,
$B[1_{n-2} \bot U_1]$ is also a strongly reduced form, and $B[1_{n-2} \bot W]$ is $GL_n(\frko)$-equivalent to
$B[1_{n-2} \bot U_1]\left[1_{n-2} \bot \smallmattwo(1;0;0;\vpi)\right]$. Therefore, it suffices to prove the assertion for $W= \smallmattwo(1;0;0;\vpi)$.  Let $i_0=i_0(B)$ be the integer defined in Definition \ref{def.7.1}.
 Put $a_i'= a_i$ for $1 \le i \le n-2$, and $a_{n-1}'=a_n'=a_{n-1}+1$. 
 Put $B'=(b_{ij}')=B\left[1_{n-2} \bot \smallmattwo(1;0;0;\vpi) \right]$. 
Then 
\[b_{ij}'=b_{ij}\]
for any $(i,j)$ such that $1 \le i,j \le n-1$,
\[b_{i,n}'=\vpi b_{i,n}\]
for any $1 \le i \le n-1$, and 
\[b_{n,n}'=\vpi^2 b_{n,n}.\]
Hence, 
\[\ord(2b_{n-1,n}') =a_{n-1}', \ \ord(b_{n,n}') \ge a_n',\]
and
\[\ord(2b_{i,n}') > (a_i'+a_{n}')/2\]
for any $1 \le i \le n-2$.
Since $B$ is strongly reduced and $a_{i_0}+a_{n-1}$ is even, we have 
\[b_{i,n-1}' =b_{i,n-1}=0\]
for any $1 \le i \le n-2$ such that $i \not=i_0$, and
\[\ord(2b_{i_0,n-1}')=\ord(2b_{i_0,n-1}) \ge (a_{i_0} +a_{n-1})/2+1 >(a_{i_0}+a_{n-1}')/2.\]
 We note that $B^{(n-2)}$ is a  reduced  form by [\cite{Ike-Kat}, Lemma 4.1].

Now first assume that $\ord(b_{n-1,n-1}) >a_{n-1}$. Then, $\ord(b_{n-1,n-1}') \ge a_{n-1}'$, and 
this implies that $B'$ is a  reduced  form with $\GK(B')=(a_1',\ldots,a_n')$.

 Next assume that $\ord(b_{n-1,n-1}) =a_{n-1}$. Since we have $\sig(i_0)=i_0$, we have 
$\ord(b_{i_0,i_0})=a_{i_0}$ and $\ord(2b_{i_0,n-1}) > {a_{i_0}+a_{n-1} \over 2}$. Moreover since 
$a_{n-1}-a_{i_0}$ is even, by [\cite{Ike-Kat}, Lemma 4.4], we can take $x \in \frko $ such that 
\[\ord(x) \ge {a_{n-1} -a_{i_0}  \over 2}, \ \ord(b_{n-1,n-1} +2b_{i_0,n-1}x+b_{i_0,i_0}x^2) > a_{n-1}.\]
 Take 
\[
V_1=
\left(\begin{smallmatrix}  1  & 0          & 0              & 0 &  0                            & 0 \\
\noalign{\vskip -2pt}
                           0           &  \smallddots  & 0  & 0 & 0                        & 0 \\
                           0 & 0  &  1            & 0 & \!\! u_{i_0,n-1} &  0 \\
\noalign{\vskip -6pt}
                           0 & 0          &  0         & \;\smallddots\; &0                              &0 \\
                           0 & 0          &  0         & 0 &1                              &0 \\
\noalign{\vskip 2pt}
                           0 & 0          & 0          & 0 &0                              &1 
                                                      \end{smallmatrix} \right)
\]
with $u_{i_0,n-1}=x$. Put $B''=(b_{ij}'')=B'[V_1]$. Then we have 
\[b_{ij}'' =b_{ij}'   \                  \text{ if } 1 \le i,j \le n-2, \text { or } i=n, \text { or } j=n,\]
and
\[\ord(2^{1-\delta_{i,n-1}}b_{i,n-1}'')=\ord(2^{1-\delta_{i,n-1}}b_{n-1,i}'') >  (a_i+a_{n-1})/2\]
for any  $1 \le i \le n-1$. Then, by [\cite{Ike-Kat}, Lemma 4.3], $B''$ is $GL_n(\frko)$-equivalent to  a reduced form $C=(c_{ij})$ such that 
\begin{align*}
&c_{ij} =b_{ij}''   \                  \text{ if } 1 \le i,j \le n-2, \text { or } i=n, \text { or } j=n,\\
&c_{i,n-1}=c_{n-1,i}=0 \text{ if } 1 \le i \le n-2 \text{ and } i\not=i_0,
\end{align*}
and 
\[\ord(2^{1-\delta_{i,n-1}}c_{i,n-1})=\ord(2^{1-\delta_{i,n-1}}c_{n-1,i}) >  (a_i+a_{n-1})/2 \]
if $i=i_0$ or $n-1$.
Hence we can prove that $C$ satisfies the required conditions in the  same way as  in the first case.
\end{proof}
 \begin{definition}
\label{def.7.2} Let  $B=(b_{ij}) \in \calh_n(\frko)$ be a strongly reduced form with ${\rm GK}(B) =(a_1,\ldots,a_n)$. Assume that
$B$ belongs to category (I) and that $a_1+\cdots+a_{n-2}$ is odd, and let $i_1=i_1(B)$ and $i_2=i_2(B)$ be those defined in Definition \ref{def.7.1}.
We then define $\kappa(B)$ as 
\[\kappa(B)=\min\{2\ord(2b_{i_k,m})-a_{i_k}-a_m \ | \ k=1,2, m=n-1,n \} .\]
Here we make the convention that $\kappa(B)=\infty$ if $b_{i_k,m}=0$ for any $k=1,2$ and $m=n-1,n$. 
We remark that
\[\kappa(B[1_{n-2} \bot \vpi^rU])=\kappa(B)\]
for any non-negative integer $r$ and $U \in GL_2(\frko)$.
\end{definition}
\begin{lemma}
\label{lem.7.3}
 Let $m$ be an even integer . Let  $C =(c_{ij}) \in \calh_m(\frko)$ be a reduced form of $\GK$ type $(\ua,\sigma)$ with $\ua=(a_1,\ldots,a_m)$. Assume that
 $a_1+\cdots+a_m$ is odd, and let $i_1=i_1(C),i_2=i_2(C)$ be the integers $i_1, i_2 $ defined in Definition \ref{def.7.1}. 
Let $A$ be an integer such that $A  \ge a_m$ and $A+a_{i_1}+\ord(\frkD_C)$ is odd.  
For each $i$ such that $1 \le i \le m$ put 
\[x_i=\begin{cases} 2^{-1}c\vpi^{(A+a_{i_1}+\ord(\frkD_C)-1)/2} & \text{ if } i=i_1 \\
0 & \text{ otherwise},
\end{cases}\]
with $c \in \frko^{\times}$ and ${\bf y}=\left(\begin{matrix}y_1 \\ \vdots \\ y_m \end{matrix}\right)=C^{-1}\left(\begin{matrix} x_1 \\ \vdots \\ x_m \end{matrix}\right)$. 
Then
\[\ord(y_i) \ge (A-a_i)/2\]
 for any $1 \le i \le m$, and 
\[\ord(C[{\bf y}]) =A.\]
\end{lemma}
\begin{proof}
Write $C^{-1}=(c_{ij}^*)_{1 \le i,j \le m}$ and put $\frkd=\ord(\frkD_{C})$. Then by 
[\cite{Ike-Kat}, Lemma 3.12], we have the following.
\begin{align}
\label{eq.7.3.1} 
&\ord(c_{ii}^*)= 2e_0+1-\frkd-a_i & (i =i_k \text{ with } k=1,2), \tag{7.3.1} \\
\label{eq.7.3.2} 
&\ord(c_{ii}^*) > 2e_0+1-\frkd-a_i & (i\not=i_1,i_2), \tag{7.3.2}\\
 \label{eq.7.3.3}
&
\ord(c_{ij}^*) \ge  (2e_0+1-\frkd-a_i-a_j)/2 & (i=i_1, j=i_2), \tag{7.3.3}\\
 \label{eq.7.3.4}
&
\ord(c_{ij}^*) >  (2e_0+1-\frkd-a_i-a_j)/2 & (i \not=j, \{ i,j \} \not=\{i_1,i_2\}). \tag{7.3.4}
\end{align}
Hence, by a simple computation,  we have $\ord(y_i) \ge (A-a_i)/2$ for any $1 \le i \le m$, and
\[\ord(C[{\bf y}])=\ord\left(C^{-1}\left[\left(\begin{matrix} x_1 \\ \vdots \\ x_m \end{matrix}\right)\right]\right)=\ord(c_{i_1,i_1}^*x_{i_1}^2)=A.\]
\end{proof}

\begin{lemma}
\label{lem.7.4}
  Let $n$ be an even integer . Let  $B$ be a reduced form of $\GK$ type $(\ua,\sigma)$ with $\ua=(a_1,\ldots,a_n)$. Assume that
$B$  belongs to category {\rm (I)} and that $a_1+\cdots+a_n$ is odd, and let $i_1=i_1(B),i_2=i_2(B)$ be the integers $i_1, i_2 $ defined in Definition \ref{def.7.1}. 
Then, we have $n \ge 4$ and $i_1, i_2 \le n-2$, and  there is a strongly reduced form $C=(c_{ij})_{n \times n}$ which is $G_{\ua}$-equivalent to  $B$  and satisfies the following three conditions: 
\begin{itemize}
\item [(CEI-0)] $C^{(n-2)}=B^{(n-2)}$.\\
\item [(CEI-1)] $ \ \ (c_{ij})_{n-1 \le i,j \le n}=\vpi^{a_{n-1}} \left(\begin{smallmatrix} 0 & 1/2 \\ 1/2 & 0 \end{smallmatrix}\right).$
\item [(CEI-2)]  $\kappa(C)=2\ord(2c_{i_k,n-1})-a_{i_k}-a_{n-1}<\ord(\frkD_{C^{(n-2)}})$ with some $k=1,2$. 
  \end{itemize}
\end{lemma}

\begin{proof} Since $B$ belongs to category (I), $B^{(n-2)}$ is a reduced  form with GK-type $(a_1,\ldots,a_{n-2})$ and $i_1,i_2 \le n-2$. Then either $a_{i_1} \equiv a_{n-1} \text{ mod $2$}$ or $a_{i_2} \equiv a_{n-1} \text{ mod  $2$}$.
We note that the matrix $B^{(n-2)}(i_1,i_2;i_1,i_2)$ is a good form.
Hence, in view of [\cite{Ike-Kat}, Lemma 4.3] and Lemma \ref{lem.7.1},  in the same way as in the proof of Lemma \ref{lem.7.2}, we may assume that
$B$ is a strongly reduced form such that $ (b_{ij})_{n-1 \le i,j \le n}=\vpi^{a_{n-1}} \left(\begin{smallmatrix} a & 1/2 \\ 1/2 & 0 \end{smallmatrix}\right)$ with some $a \in \frko$.
First assume that $\kappa(B)<\ord(\frkD_{B^{(n-2)}})$. Then, again by Lemma \ref{lem.7.1}, we may assume that $a=0$.
Moreover, if $\ord(2b_{i_l,n-1}) -a_{i_l}/2  > \ord(2b_{i_k,n})-a_{i_k}/2$
 for any $l=1,2$, replacing $B$ by $B\left[1_{n-2} \bot \smallmattwo(0;1;1;0)\right]$,
we obtain the  matrix $C$ satisfying the condition (CEI-2). 
Next assume that $\kappa(B) \ge \ord(\frkD_{B^{(n-2)}})$.
Without loss of generality, we may assume that $a_{i_1}+a_{n-1}+\ord(\frkD_{B^{(n-2)}})$ is odd. 
For each $i$ such that $1 \le i \le n-2$ put 
\[x_i=\begin{cases} 2^{-1}\vpi^{(a_{n-1}+a_{i_1}+\ord(\frkD_{B^{(n-2)}})-1)/2} & \text{ if } i=i_1 \\
0 & \text{ otherwise},
\end{cases}\]
and ${\bf y}=\left(\begin{matrix}y_1 \\ \vdots \\ y_{n-2} \end{matrix}\right)=(B^{(n-2)})^{-1}\left(\begin{matrix} x_1 \\ \vdots \\ x_{n-2} \end{matrix}\right)$. 
Then, by Lemma\ref{lem.7.3}, we have 
\[\ord(y_i) \ge (a_{n-1}-a_i)/2 \text{ for any } 1 \le i \le n-2,\]
\[2\ord(2(x_{i_1}+b_{i_1,n-1}))-a_{i_1}-a_{n-1}<\ord(\frkD_{B^{(n-2)}})\] 
and 
\[\ord(B^{(n-2)}[{\bf y}]+2b_{i_1,n-1}y_{i_1}+2b_{i_2,n-1}y_{i_2}) \ge a_{n-1}.\]
This implies that the matrix 
$B'=B\left[\left(\begin{matrix} 1_{n-2} & {\bf y} & O \\ O & 1 & 0 \\ O & 0 & 1 \end{matrix}\right)\right]$ 
 is a strongly reduced form, satisfies the conditions (CEI-1) and (CEI-2), and its lower right $2 \times 2$ block is $\vpi^{a_n}\left(\begin{matrix} a' & 1/2 \\ 1/2 & 0 \end{matrix}\right)$ with $a' \in \frko$. Thus, by Lemma \ref{lem.7.1}, we obtain the matrix $C$ satisfying the condition (CEI-1).
\end{proof}

\begin{lemma}
\label{lem.7.5} Let $n$ be an even integer.  Let $B \in \calh_n(\frko)$ be a strongly reduced form belonging to category {\rm (I)} with $\GK(B)=(a_1,\ldots,a_n)$. Assume  that $a_1+\cdots +a_n$ is odd and let $i_1=i_1(B)$ and $i_2=i_2(B)$ be the integers in Definition \ref{lem.7.1}, and put $r_k=\ord(2b_{i_k,n-1})$ for $k=1,2$. Moreover assume that $B$ satisfies the conditions {\rm (CEI-1)} and 
{\rm (CEI-2)}.
Then we have the following:
\begin{itemize}
\item[(1)] $B^{(n-1)}$ is non-degenerate and 
\[e_{B^{(n-1)}}=a_{n-1}+e_{B^{(n-2)}}+\kappa(B)+1, \]
and 
\begin{align*}
&\eta_{B^{(n-1)}}=\eta_{B^{(n-2)}}  \langle D_{B^{(n-2)}(i_1;i_1)},D_{B^{(n-2)}} \rangle. 
\end{align*}
\item[(2)] Assume that 
\[a_{n-1} \ge \max(2\lambda_{\frkl_{B^{(n-2)}}}+2e_0, 2\ord(\det B^{(n-2)}) +2e_0(n+1)+2),\]
where  $\lambda_{\frkl_{B^{(n-2)}}}$ is the integer defined in (2.3) of Lemma \ref{lem.5.1}. Then for any $x =(x_1,\ldots,x_{n-1}) \in M_{1,n-1}(\frko)$, $B_x $ is non-degenerate, and 
\[e_{B_x}=a_{n-1}+e_{B^{(n-2)}}+\kappa(B)+1. \]
Moreover, if $\kappa(B)=1$, then
\[\eta_{B_x}=\eta_{B^{(n-1)}}
\langle 1+x_{n-1} v_1 \vpi^{\ord(\frkD_{B^{(n-2)}})-1}, D_{B^{(n-2)}} \rangle,\]
where $v_1$ is an element of $\frko^{\times}$ independent of $x$. In particular, $\eta_{B_x}$ is uniquely determined by $x_{n-1} \ {\rm mod} \ \frkp$.
\end{itemize}
\end{lemma}
\begin{proof} Without loss of generality we may assume that $\kappa(B)=2r_1-a_{n-1}-a_{i_1}$. \\
(1) Put 
$A_{kl}=-\det B^{(n-2)}(i_k;i_l)$, and  $D_{kl}=A_{kl}b_{i_k,n-1}b_{i_l,n-1}$  for $k,l=1,2$. 
Then we have  
\begin{align*}
&\det B^{(n-1)}=A_{11}b_{i_1,n-1}^2+2A_{12}(-1)^{i_1+i_2}b_{i_1,n-1}b_{i_2,n-1}+A_{22}b_{i_2,n-1}^2. 
\end{align*}
First we prove 
\begin{align}
\label{eq.7.4.1}
&\ord(\det B^{(n-1)})=\ord(A_{11})+2(r_1-e_0).\tag{7.4.1}
\end{align}
To prove this, put $(b_{ij}^*)_{1 \le i,j \le n-2}=(B^{(n-2)})^{-1}$. Then we note that $A_{kl}=(-1)^{i_k+i_l+1}b_{i_l,i_k}^* \det B^{(n-2)}$ and by [\cite{Ike-Kat}, Theorem 0.1] we have 
\begin{align}
\label{eq.7.4.2}
\ord(\det B^{(n-2)})=\sum_{i=1}^{n-2} a_i +\ord(\frkD_{B^{(n-2)}})-1-(n-2)e_0.\tag{7.4.2}
\end{align}
Hence, by (\ref{eq.7.3.1}),  we have
\begin{align}
\label{eq.7.4.3}
\ord(A_{kk})=\sum_{1 \le i \le n-2 \atop i \not=i_k} a_i -(n-4)e_0  \tag{7.4.3}
\end{align}
and 
\begin{align}
\label{eq.7.4.4}
&\ord(D_{kk}) \tag{7.4.4} \\
&=\ord(A_{kk})+2(r_k-e_0) \notag\\
&=\sum_{1 \le i \le n-2 \atop i \not=i_k} a_i -(n-2)e_0+ 2r_k  \notag
\end{align}
for $k=1,2$, and by (\ref{eq.7.3.3})
\begin{align*}
\ord(2D_{12})&= e_0+\ord(A_{12})+(r_1-e_0)+(r_2-e_0) \\
& \ge \sum_{i=1}^{n-2}a_i -(a_{i_1}+a_{i_2})/2+r_1+r_2\\
&+(\ord(\frkD_{B^{(n-2)}})-1)/2-(n-2)e_0.
\end{align*}
We note that  we have $r_1-a_{i_1}/2 <r_2 -a_{i_2}/2$. 
Hence we have $\ord(D_{11}) < \ord(D_{22})$ and $\ord(D_{11}) < \ord(2D_{12})$, and this proves (\ref{eq.7.4.1}). 
Hence we prove  the former half of (1) in view of (\ref{eq.7.4.4}).  
We prove the latter half of (1).
To prove this, we recall that  
\[\eta_{B^{(n-1)}}=\eta_{B^{(n-2)}}\langle (-1)^{(n-2)/2}\det B^{(n-1)},(-1)^{(n-2)/2}\det B^{(n-2)}\rangle \]
in view of  [\cite{Ike-Kat}, Lemma 3.4].  By Example 3 in Chapter II, Section 5 of \cite{Sat}, we have 
\[\det \left(\begin{matrix} A_{11} & A_{12} \\
                            A_{12} & A_{22}
       \end{matrix}\right)
=\det B^{(n-2)} \det B^{(n-2)}(i_1,i_2;i_1,i_2).\]
Hence we have
\[\det B^{(n-1)}=A_{11}u^2+A_{11}\det B^{(n-2)}\det B^{(n-2)}(i_1,i_2;i_1,i_2)v^2\]
with some $u,v \in F$. 
Hence we have
\[\langle A_{11}^{-1} \det B^{(n-1)},-\det B^{(n-2)}\det B^{(n-2)}(i_1,i_2;i_1.i_2)\rangle=1.\]
Moreover we note that 
$(b_{i_1,n-1}^2A_{11})^{-1}\det B^{(n-1)} \in \frko^{\times}$ and 
$B^{(n-2)}(i_1,i_2;i_1,i_2)$ is a good  form (for the definition of `good form', see Definition \ref{def.7.1}).
 Then,  by the remark in Definition \ref{def.7.1}, 
 we have
\[\langle A_{11}^{-1} \det B^{(n-1)}, (-1)^{(n-4)/2}\det B^{(n-2)}(i_1,i_2;i_1,i_2) \rangle=1.\]         
This proves the latter half of (1).  \\
(2) For $x =(x_1,\ldots,x_{n-1}) \in M_{1,n-1}(\frko)$,  write  $B_x=(b_{x;ij})_{(n-1) \times (n-1)}$.
 We have 
 \[B_x=B^{(n-1)}+ \vpi {}^tx {\bf b} + \vpi {}^t{\bf b}x,\]
 where ${\bf b}=(b_{i,n})_{1 \le i \le n-1} \in M_{1,n-1}(\frko)$.  Hence, for any $1 \le i,j \le n-1$, we have
 \[b_{x;ij}=b_{ij}+x_i\vpi b_{jn}+x_j\vpi b_{in}.\]
 We have
\[b_{in}=
\begin{cases}
b_{i_k,n}  & \text{ if } i=i_k \text{ with } k=1,2 \\

 0                                          & \text{ if } i\not=i_1,i_2, i \le n-2 \\
 2^{-1} \vpi^{a_n}                     & \text{ if } i=n-1,
 \end{cases}\]
 and $\ord(b_{i_k,n}) > (a_n+a_{i_k})/2-e_0$.  We note that $a_n=a_{n-1}$. Hence,  by the 
 assumption on $a_{n-1}$,
we have 
\[\ord(b_{i_k,n}) \ge \ord(\det B^{(n-2)}) +(n-2)e_0\] 
for any $1 \le k \le 2$, and 
\[\ord(2^{-1}\vpi^{a_n}) \ge  \ord(\det B^{(n-2)}) +(n-2)e_0.\]
 Hence we have
\begin{align}
\label{eq.7.4.5}
b_{x;i_1,n-1} b_{i_1,n-1}^{-1} \equiv 1 \text{ mod } \frkp, \tag{7.4.5}
\end{align}
and 
 \[b_{x;ij} \equiv b_{ij} \ {\rm mod} \  \det B^{(n-2)}\frkp^{(n-2)e_0+1}\]
for any $1 \le i,j \le n-2$. We note that $b_{ij}, b_{x;ij} \in 2^{-1}\frko$, and hence
\begin{align}
\label{eq.7.4.6}
\det B_x^{(n-2)}(i;j) \equiv \det B^{(n-2)}(i;j) \ {\rm mod} \ \det B^{(n-2)}\frkp^{2e_0+1}
\tag{7.4.6}
\end{align}
for any $1 \le i,j \le n-2$, and 
\begin{align}
\label{eq.7.4.7}
\det B_x^{(n-2)} \equiv \det B^{(n-2)} \ {\rm mod} \ \det B^{(n-2)}\frkp^{e_0+1}.
\tag{7.4.7}
\end{align}
 Hence, by (\ref{eq.7.4.2}) and (\ref{eq.7.4.3}), we have 
\[
\ord(\det B_x^{(n-2)}(i_k;i_k))=\ord(\det B^{(n-2)}(i_k;i_k))\]
for $k=1,2$. Moreover we have 
\[B_x^{(n-2)}  \equiv  B^{(n-2)} \ {\rm mod} \ \vpi^{\lambda_{\frkl_{B^{(n-2)}}}}S_{n-2}(\frko),\]
and hence, by (2.3) of Lemma \ref{lem.5.1},  $B_x^{(n-2)}$ is $GL_{n-2}(\frko)$-equivalent to $B^{(n-2)}$.
We have  
 \[b_{x;i,n-1}=
\begin{cases} 2^{-1}\vpi^{a_n+1}x_i                              & \text{ if } i\not=i_1,i_2, i \le n-2\\
b_{i_k,n-1}+ \vpi b_{i_k,n}x_{n-1}+ 2^{-1}\vpi^{a_n+1}x_{i_k}    & \text{ if } i=i_k \text{ with } k=1,2\\
\vpi^{a_{n-1}+1} x_{n-1}& \text { if } i=n-1,
\end{cases}\]
\[\ord(b_{i_k,n-1}) > (a_{n-1}+a_{i_k})/2-e_0\]
 and  
\[a_{n-1}/2 \ge \ord(\det B^{(n-2)})+e_0(n-2)+2e_0+1  \ge 2e_0+1,\]
 and hence $\ord(b_{i_k,n-1}) \ge e_0+1$.
Similarly $\ord(b_{i_k,n}) \ge e_0+1$ and $\ord(2^{-1}\varpi^{a_n}) \ge e_0+1$. Hence
$\ord(b_{x;i,n-1}) \ge e_0+1$ for any $1 \le i \le n-2$.  
Hence
\[\det B_x^{(n-2)}(i;j) b_{x;i,n-1}b_{x;j,n-1} \equiv 0 \text { mod } \det B^{(n-2)}\frkp^{a_{n-1}+2}\]
for any integers  $1 \le i,j \le n-2$ such that $i \not= i_1,i_2$ or $j \not= i_1,i_2$.
Similarly we have 
\begin{align*}
&\det B_x^{(n-2)}(i_k;i_l)b_{x;i_k,n-1}b_{x;i_l,n-1} \\
&\equiv (b_{i_k,n-1}+\vpi x_{n-1}b_{i_k,n})(b_{i_l,n-1}+\vpi x_{n-1}b_{i_l,n})\\
&\times  \det B^{(n-2)}(i_k;i_l) \text { mod } \det B^{(n-2)}\frkp^{a_{n-1}+2} 
\end{align*}
for any $k,l=1,2$, and  
\[\det B_x^{(n-2)}b_{x;n-1,n-1} \equiv \det B^{(n-2)} \vpi^{a_{n-1}+1} x_{n-1}  \text  { mod }  \det B^{(n-2)}\frkp^{a_{n-1}+2} .\]
Put 
\[\widetilde {B_x}= \mattwo( {B_x^{(n-2)}} ; {\begin{matrix}O \\ b_{x;i_1,n-1} \\ O \\ b_{x;i_2,n-1}\\ O \end{matrix}};{\begin{matrix}O & b_{x;i_1,n-1} & O & b_{x;i_2,n-1}& O \end{matrix}}; 0).\]
Then we have 
\[\det B_x=\det \widetilde B_x +\det B_x^{(n-2)}b_{x;n-1,n-1}.\]
By  (\ref{eq.7.3.2}),(\ref{eq.7.3.3}),(\ref{eq.7.3.4}), (\ref{eq.7.4.6}) and (\ref{eq.7.4.7}), we have 
\begin{align}
\label{eq.7.4.8}
&\det \widetilde B_x \tag{7.4.8} \\
&=\sum_{1 \le i,j \le n-2} (-1)^{i+j+1} \det B_x^{(n-2)}(i;j)b_{x;i,n-1}b_{x;j,n-1} \notag \\
& \equiv \sum_{1 \le k,l \le 2} (b_{i_k,n-1}+\vpi x_{n-1}b_{i_k,n})(b_{i_l,n-1}+\vpi x_{n-1}b_{i_l,n}) \notag \\ 
& \times (-1)^{i_k+i_l}A_{kl}   \text{ mod }   
 \det B^{(n-2)}\frkp^{a_{n-1}+2}. \notag
\end{align}
and 
\begin{align*}
&\det B_x^{(n-2)}b_{x;n-1,n-1} \equiv \det B^{(n-2)} x_{n-1}\vpi^{a_{n-1}+1} \text{ mod }    \det B^{(n-2)}\frkp^{a_{n-1}+2},
\end{align*}
where $A_{kl}$ is that in the proof of (1).  Hence, by (1) and  (\ref{eq.7.4.2}), we have 
\[\ord(\det B_x^{(n-2)} b_{x;n-1,n-1})- \ord(\det \widetilde {B_x})=\ord(\frkD_{B^{(n-2)}})-\kappa(B),\]
and in particular
\[\ord(\det \widetilde {B_x})=\ord(\det B^{(n-1)})<\ord(\vpi^{a_{n-1}+1}\det B_x^{(n-2)}).\] 
Hence $\ord(\det B_x)=\ord(\det B^{(n-1)})$.
This proves the former half of (2).  

Now assume that $\kappa(B)=1$. Then 
\begin{align}
\label{eq.7.4.9}
&\ord(\det B^{(n-2)} \vpi^{a_{n-1}+1})-\ord(b_{i_1,n-1}^2A_{11})=\ord(\frkD_{B^{(n-2)}})-1 \tag{7.4.9} 
\end{align}
By (\ref{eq.7.4.5}) and (\ref{eq.7.4.8}) we have 
\[\det \widetilde {B_x}(b_{i_1,n-1}^2A_{11})^{-1} \equiv  1 \ {\rm mod} \ \frkp. \]
Put $v_1=(b_{i_1,n-1}^2A_{11}\vpi^{\ord(\frkD_{B^{(n-2)}})-1}  )^{-1}\det B^{(n-2)} \vpi^{a_{n-1}+1}$. Then, by  (\ref{eq.7.4.9}),  we have $v_1 \in \frko^{\times}$ and  
\[\det B_x (\det \widetilde {B_x})^{-1} \equiv 1+x_{n-1}v_1 \vpi^{\ord(\frkD_{B^{(n-2)}})-1} \text{ mod } \frkD_{B^{(n-2)}}.\]
We have
\begin{align*}
\eta_{B_x} &=\eta_{B_x^{(n-2)}}\langle B_{B_x},D_{B_x^{(n-2)}} \rangle \\
&=\eta_{B^{(n-2)}}\langle (-1)^{(n-2)/2}( \det \widetilde B_x +\det B_x^{(n-2)}b_{x;n-1,n-1}),D_{B^{(n-2)}} \rangle.
\end{align*}
By using the same argument as in the proof of the latter half of (1), we have 
\[\eta_{B^{(n-2)}}\langle D_{\widetilde {B_x}},D_{B^{(n-2)}}\rangle =\eta_{B^{(n-1)}},\] 
and  this proves the equality of in the latter half of (2).
We note that the conductor of the character 
\[\frko^{\times} \ni z \mapsto \langle z, D_{B^{(n-2)}} \rangle\] 
is $\frkD_{B^{(n-2)}}$. Hence we prove that $\eta_{B_x}$ is uniquely determined by $x_{n-1} \text{ mod } \frkp$.
\end{proof}

 For integers $e,\widetilde e$,  a real number $\xi$, let $C(e,\widetilde e,\xi;Y,X)$ and $D(e,\widetilde e,\xi;Y,X)$ be the rational functions in $Y^{1/2},X^{1/2}$ defined in Definition \ref{def.4.3}, and for an $\EGK$ datum $G$, let $\widetilde \calf(G;Y,X)$ be the Laurent polynomial 
defined in Section 4.   

{\bf Proof of Theorem \ref{th.1.1}.}  

We prove 
\begin{align*}
\tag{$\mathrm{EF}_n$}
&\widetilde F(B,X)=\widetilde \calf(\EGK(B);q^{1/2},X) \\
&\text{ for any reduced form } B \in \calh_n(\frko)
\end{align*}
by induction on $n$. This proves Theorem \ref{th.1.1} in view of Theorem \ref{th.3.2}.  Let $B\in \calh_n(\frko )$ be a reduced form of type $(\ua, \sig). $
Put $\ua=(a_1,\ldots,a_n)$. For a non-negative integer $i \le n$ let $\frke_i=\frke(\ua)_i$ be the integer in Definition \ref{def.4.3}, and 
let $\calm^{0}(\ua)$ be the set in Definition \ref{def.3.4}. 
By [\cite{Ike-Kat}, Theorem 0.1], we have $\frke_B=\frke_n$.

Assume that $n=1$. 
Then, by Theorem \ref{th.5.3}, we have
\[
\widetilde F(B,X)=\sum_{i=0}^{a_1} X^{i-(a_1/2)}.
\]
Thus ($\mathrm{EF}_1$) holds.
Now assume  that $n \ge 2$, and assume that $(\mathrm{EF}_{n'}$) holds for any positive integer $n'<n$.

\noindent
{\bf Step 1:} \ First we assume that $B$ belongs to category (II) or that 
$B$ belongs to category (I) and $n+a_1+\cdots+a_{2[n/2]}$ is even. Then, $B$ is $GL_n(\frko)$-equivalent to a reduced form $B_1$ such that $B_1^{(n-1)}$ is a reduced form with $\GK(B_1^{(n-1)})=\ua^{(n-1)}$. In fact, $B^{(n-1)}$ is automatically a reduced form and $\GK(B^{(n-1)})=\ua^{(n-1)}$ if $B$ belongs to category (II).
Assume that $B$ belongs to category (I).
Assume that both $n$ and $a_1+\cdots+a_{n-1}$ are odd.
By Lemma \ref{lem.7.1},  we may assume that $\ord(b_{n-1, n-1})=a_{n-1}$.
Then $B^{(n-1)}$ is a reduced form with $\GK(B^{(n-1)})=\ua^{(n-1)}$.
The case that  both $n$ and $a_1+\cdots+a_{n-1}$ are even is similar. Therefore we may assume that $B^{(n-1)}$ is a reduced form with $\GK(B^{(n-1)})=\ua^{(n-1)}$.
By Proposition \ref{prop.4.5}, there exists $H\in\mathcal{NEGK}_n$ such that $\Ups_n(H)=\EGK(B)$ and $\Ups_{n-1}(H')=\EGK(B^{(n-1)})$.

For each  $x \in M_{1,n-1}(\frko)$, let $B_x$ be the matrix in Definition \ref{def.5.3}. 
By definition, $B_{x}-B^{(n-1)} \in \calm^0(\ua^{(n-1)})$ for any $x \in M_{1,n-1}(\frko)$. 
Hence, by Theorem \ref{th.3.3},  $B_x$ is also a reduced  form with $\EGK(B_x)=\EGK(B^{(n-1)})$.
In particular, we have $\frke_{B_x}=\frke_{B^{(n-1)}}=\frke_{n-1}$ for any $x \in M_{1,n-1}(\frko)$.
Hence, $\ord(\det (2B_x)) \le \frke_{n-1}+2e+1$ for any $x \in M_{1,n-1}(\frko)$, and we can take a  positive integer $m_0$ satisfying the condition in (1) of Theorem \ref{th.5.2}.

Assume that $n$ is odd.
Then, by Theorem \ref{th.5.3}, we have
\begin{align*}
&\widetilde F(B,X) \\
=&q^{-m_0(n-1)}\sum_{x \in M_{1,n-1}(\frko )/\frkp^{m_0}M_{1,n-1}(\frko )}\Bigl\{ D(\frke_n,\frke_{n-1},\xi_{B^{(n-1)}};X)\widetilde F(B_{x},q^{1/2}X) \\
&\quad +\eta_B D(\frke_n,\frke_{n-1},\xi_{B^{(n-1)}};X^{-1})\widetilde F(B_{x},q^{1/2}X^{-1})\Bigr\}. 
\end{align*}
By ($\mathrm{EF}_{n-1}$), we have 
\begin{align*} 
\widetilde F(B_{x},X)&=\widetilde \calf(\EGK(B^{(n-1)});q^{1/2},X)  \\
&=\widetilde F(B^{(n-1)},X)=\calf(H', q^{1/2}, X) 
\end{align*}
for any $x \in M_{1,n-1}(\frko)$. 
Hence we have 
\begin{align}
\label{eq.1.1}
\widetilde F(B,X)
=&D(\frke_n,\frke_{n-1},\xi_{B^{(n-1)}};X)\widetilde F(B^{(n-1)}, q^{1/2} X) \tag{1.1}\\
&+\eta_B D(\frke_n,\frke_{n-1},\xi_{B^{(n-1)}};X^{-1})\widetilde F(B^{(n-1)},q^{1/2} X). \notag
\end{align}
Thus by Definition \ref{def.4.4}
\begin{align*}
\widetilde F(B,X)=\calf(H, q^{1/2}, X) =\calf(\EGK(B);q^{1/2},X).
\end{align*}
In the case that $n$ is even, similarly we have 
\begin{align}
\label{eq.1.2}
\widetilde F(B,X)
=&C(\frke_n,\frke_{n-1},\xi_{B};X)\widetilde F(B^{(n-1)}, q^{1/2} X) \tag{1.2}\\
&+ C(\frke_n,\frke_{n-1},\xi_{B};X^{-1})\widetilde F(B^{(n-1)},q^{1/2} X), \notag
\end{align}
and 
\begin{align*}
\widetilde F(B,X)=\calf(\EGK(B);q^{1/2},X).
\end{align*}

{\bf Step 2:} \ Let $n$ be odd, and assume that $B$ belongs to category (I) and that  $a_1+\cdots+a_{n-1}$ is even.
Then, $\GK(B)$ satisfies the condition of (Case 1) of Proposition \ref{prop.4.4}.
In this case, $B^{(n-2)}$ is a reduced form with $\GK(B^{(n-2)})=\ua^{(n-2)}$. For any $y \in M_{2,n-2}(\frko)$, we have
$B_y - B^{(n-2)} \in \calm^0(\GK(B^{(n-2)}))$, and by Theorem \ref{th.3.1} and the induction assumption ($\mathrm{EF}_{n-2}$), we have
\[\widetilde F(B_y,X)=\widetilde F(B^{(n-2)},X).\]
By (2) of [\cite{Ike-Kat}, Lemma 6.2 (2)], we have $\eta_{B^{(n-2)}}=\eta_B$. In view of Lemma \ref{lem.7.2}, we may assume that
$B$ is a strongly reduced. By (2) of Theorem \ref{th.5.4}, we have
\begin{align}
\label{eq.2.1}
&  \widetilde F(B,X) \tag{2.1}
\\
&=q^{-1}(X^{-1}+X)^{-1}\Bigl\{ \sum_{W \in {\bf D}_{2,1}/GL_2(\frko )} \widetilde F(B[1_{n-2} \bot W],X) \notag \\
&\quad -q^{\frke_{n-2}/2}X^{(-\frke_n+\frke_{n-2})/2-1}\widetilde F(B^{(n-2)},qX) \notag \\
&\quad  -\eta_B q^{\frke_{n-2}/2}X^{(\frke_n-\frke_{n-2})/2+1}\widetilde F(B^{(n-2)},qX^{-1})
\vphantom{\sum_{W \in {\bf D}_{2,1}/GL_2(\frko )}}
\Bigr\}.\notag
\end{align}
By (2) of Lemma \ref{lem.7.2}, $B[1_{n-2} \bot W]$ is $GL_n(\frko)$-equivalent to  a reduced  form whose GK invariant is $(a_1,\ldots,a_{n-2},a_{n-1}+1,a_{n-1}+1)$
for any $W \in {\bf D}_{2,1}$. 
Moreover, we have $\eta_{B[1_{n-2} \bot W]}=\eta_B$ for any $W \in {\bf D}_{2,1}$.
Hence, by a direct calculation using (\ref{eq.1.1}) and (\ref{eq.1.2}),  we have 
\begin{align}\label{eq.2.2}
&\widetilde F(B[{1}_{n-2}\bot W],X) \tag{2.2} \\
&=q^{(\frke_{n-2}-1)/2} 
\Biggl\{ 
\frac{X^{(-\frke_n+\frke_{n-2})/2-2}} {(q^{1/2}X)^{-1}-q^{1/2}X } 
\widetilde F(B^{(n-2)},qX)  \notag\\
&\quad  +\eta_B
\frac{X^{(\frke_n-\frke_{n-2})/2+2}} {(q^{1/2}X^{-1})^{-1}-q^{1/2}X^{-1} }\widetilde F(B^{(n-2)},qX^{-1})\Biggr\} \notag \\
&+\eta_B \frac{q^{\frke_{n-1}/2}(q-1)(X+X^{-1}) } { ((q^{1/2}X)^{-1}-q^{1/2}X)((q^{1/2}X^{-1})^{-1}-q^{1/2}X^{-1})} \notag \\
&\times \widetilde F(B^{(n-2)},X) . \notag
\notag
\end{align}
for any $W \in {\bf D}_{2,1}$. 
Hence, by (\ref{eq.2.1}) and (\ref{eq.2.2}) we have 
\begin{align}
\label{eq.2.3}
\widetilde F(B,X)=&q^{\frke_{n-2}/2-1/2} 
\Biggl\{{X^{(-\frke_n+\frke_{n-2})/2-1} \over (q^{1/2}X)^{-1}-q^{1/2}X } \widetilde F(B^{(n-2)},qX) \tag{2.3}\\
&+ \eta_B{X^{(\frke_n-\frke_{n-2})/2+1} \over (q^{1/2}X^{-1})^{-1}-q^{1/2}X^{-1} }\widetilde F(B^{(n-2)},qX^{-1})\Biggr\} \notag\\
&+\eta_B {q^{\frke_{n-1}/2}(q-q^{-1}) \over ((q^{1/2}X)^{-1}-q^{1/2}X)((q^{1/2}X^{-1})^{-1}-q^{1/2}X^{-1})} \notag\\
&\times \widetilde F(B^{(n-2)},X). \notag
\end{align}
Hence 
\[\widetilde F(B, X)=\widetilde \calf(\EGK(B), q^{1/2}, X)\]
 by ($\mathrm{EF}_{n-2}$), (1) of Proposition \ref{prop.4.3} and (2) of Proposition \ref{prop.4.5}.

\noindent
{\bf Step 3:} \ Let $n$ be even, and  assume that $B$ belongs to category (I) and that  $a_1+\cdots +a_n$ is odd.  Then, $\GK(B)$ satisfies the condition of (Case 2) of Proposition \ref{prop.4.4}. 
Then we prove the following equality:
 \begin{align}
\label{eq.3.1}
\widetilde F(B,X) &=q^{\frke_{n-2}/2} \Biggl \{{X^{(-\frke_n+\frke_{n-2})/2 -1}  \over X^{-1}-X}\widetilde F(B^{(n-2)},qX)  \tag{3.1} \\
&+{X^{(\frke_n- \frke_{n-2})/2 +1}  \over X-X^{-1}}\widetilde F(B^{(n-2)},qX^{-1})\Biggr \}.\notag
\end{align}
This combined with (2) of Proposition \ref{prop.4.3} and (2) of Proposition \ref{prop.4.5} proves that we have 
\begin{align*}
\widetilde F(B,X)=\calf(\EGK(B);q^{1/2},X).
\end{align*}
 Let $i_1=i_1(B)$ and $i_2=i_2(B)$ be those defined in Definition \ref{def.7.1}.  To prove (\ref{eq.3.1}), we may assume that $B$ is a strongly reduced form and that it satisfies the conditions (CEI-1) and (CEI-2)  in   Lemma \ref{lem.7.4}.
Similarly to Step 2,  by Theorem \ref{th.3.1} and the induction assumption ($\mathrm{EF}_{n-2}$), we have
\[\widetilde F(B_y,X)=\widetilde F(B^{(n-2)},X).\]
Let $r$ be a positive integer. Then, by (1.1) of Theorem \ref{th.5.4} and by a simple computation,
we see that (3.1) holds for $\widetilde F(B[1_{n-2} \bot \vpi^{r-1} 1_2],X)$ if  it  holds for  
$\widetilde F(B[1_{n-2} \bot \vpi^r 1_2],X)$. Therefore, to show  (\ref{eq.3.1}) we may assume that 

\bigskip

\noindent
(CEI-3) $a_{n-1} \ge \max(2\lambda_{\frkl_{B^{(n-2)}}}+2e_0, 2\ord(\det B^{(n-2)}) +2e_0(n+2))$.

\bigskip

\noindent
Let $k=\kappa(B)$. Then $k$ is a positive integer such that $k<\ord(\frkD_{B^{(n-2)}})$. We prove (\ref{eq.3.1}) by induction on $k$. First we prove (\ref{eq.3.1}) for any strongly reduced form $B \in \calh_n(\frko)$ with $\kappa(B)=1$ satisfying the conditions (CEI-1),(CEI-2) and (CEI-3).  In this case, by (2) of Lemma \ref{lem.7.5}, 
$\frke_{B_x}=2+(\frke_{n-2}+\frke_n)/2$ for any $x \in M_{1,n-1}(\frko)$ and it does not depend on the choice of $x$, which will be denoted by $\widetilde \frke$. Hence we can take a
 positive integer satisfying the condition in Theorem \ref{th.5.2}. Then, as in the proof of Theorem \ref{th.5.3}, we have
\begin{align*}
\widetilde F(B[1_{n-1} \bot \vpi],X)&=X\widetilde F(B,X) \\
&+q^{\widetilde \frke/4}X^{(-\frke_n+\widetilde \frke)/2-1}\\
&\times q^{-m_0(n-1)}\sum_{x \in M_{1,n-1}(\frko )/\frkp^{m_0}  M_{1,n-1}(\frko )}\widetilde F(B_x,q^{1/2}X).
\end{align*}
By (2) of Theorem \ref{th.5.3}, we have
\begin{align*}
 \widetilde F(B_x,q^{1/2}X) &=q^{{\widetilde \frke}/4}(q^{1/2}X)^{(-\widetilde \frke+\frke_{n-2})/2} \widetilde F(B^{(n-2)},qX)\\
&+\eta_x
q^{{\widetilde \frke}/4}(q^{1/2}X)^{(\widetilde \frke-\frke_{n-2})/2} \widetilde F(B^{(n-2)},X^{-1}),
\end{align*}
where $\eta_x=\eta_{B_x}$.
Hence
\begin{align*}
\widetilde F(B[1_{n-1} \bot \vpi],X)&=X\widetilde F(B,X) \\
&+q^{\frke_{n-2}/2}X^{(-\frke_n+\frke_{n-2})/2-1}\widetilde F(B^{(n-2)},qX) \\
& +\widetilde F(B^{(n-2)},X^{-1})q^{\widetilde \frke/2}X^{(-\frke_n+2\widetilde \frke-\frke_{n-2})/2-1}\\
&\times q^{-m_0(n-1)}\sum_{x \in M_{1,n-1}(\frko )/\frkp^{m_0}  M_{1,n-1}(\frko )}\eta_x.
\end{align*}
Put 
\[I=\sum_{x \in M_{1,n-1}(\frko )/\frkp^{m_0}  M_{1,n-1}(\frko )}\eta_x.\]
Then, by (2) of Lemma \ref{lem.7.5}, we have
\begin{align*}
I =cq^{m_0(n-1)-1}\sum_{y \in \frko/\frkp} \langle 1+y\vpi^{\ord(\frkD_{B^{(n-2)}})-1},D_{B^{(n-2)}}\rangle,
\end{align*}
where $c$ is a constant independent of $y$. 
By (2) of Lemma \ref{lem.7.5}, the homomorphism 
\[\frko/\frkp \ni y \mapsto \langle 1+y\vpi^{\ord(\frkD_{B^{(n-2)}})-1}, (-1)^{(n-2)/2} \det B^{(n-2)} \rangle \in \{\pm 1 \}\]
is non-trivial, and we have $I=0$.
Hence 
\begin{align*}
\widetilde F(B[1_{n-1} \bot \vpi],X)&=X\widetilde F(B,X) \\
&+q^{\frke_{n-2}/2}X^{(-\frke_n+\frke_{n-2})/2-1}\widetilde F(B^{(n-2)},qX) .
\end{align*}
We also have
\begin{align*}
\widetilde F(B[1_{n-1} \bot \vpi],X^{-1})&=X^{-1}\widetilde F(B,X^{-1})\\
&+q^{ \frke_{n-2}/2}X^{(\frke_n-\frke_{n-2})/2+1}\widetilde F(B^{(n-2)},qX^{-1}).
\end{align*}
By Proposition \ref{prop.2.1} we have 
\[\widetilde F(B[1_{n-1} \bot \vpi],X^{-1})=\widetilde F(B[1_{n-1} \bot \vpi],X), \]
\[\widetilde F(B,X^{-1})=\widetilde F(B,X),
\] and 
\[\widetilde F(B^{(n-2)},X^{-1})=\widetilde F(B^{(n-2)},X).\]
Hence we prove the equality (\ref{eq.3.1}) in the case $\kappa(B)=1$.\\
 Let $2 \le k < \ord(\frkD_{B^{(n-2)}})$, and  assume that (\ref{eq.3.1}) holds for any $B' \in \calh_n(\frko)$ with  $\kappa(B')<k$ satisfying the conditions (CEI-1),(CEI-2) and (CEI-3).  Let $B$ be a strongly reduced form  satisfying the conditions (CEI-1),(CEI-2) and (CEI-3) such that  $\kappa(B)=k$.  Then, by (2) of Lemma \ref{lem.7.5}, $\frke_{B_x}=k+(\frke_{n-2}+\frke_n)/2+1$ for any $x \in M_{1,n-1}(\frko)$.  Hence, by (1.2) of Theorem \ref{th.5.4},
\begin{align*}
\widetilde F(B,X)(X^{k-1}-X^{1-k}) &=\widetilde F(B[1_{n-1} \bot \vpi],X)(X^k-X^{-k}) \\
&+q^{\frke_{n-2}/2}X^{-k-1+(\frke_{n-2}-\frke_n)/2}\widetilde F(B^{(n-2)},qX)\\
&-q^{\frke_{n-2}/2}X^{k+1+(-\frke_{n-2}+\frke_n)/2}\widetilde F(B^{(n-2)},qX^{-1}).
\end{align*}
We note that $B[1_{n-1} \bot \vpi]$ is a strongly reduced form  satisfying the conditions (CEI-1), (CEI-2) and (CEI-3) with  $\kappa(B[1_{n-1} \bot \vpi])=k-1$. Hence, by the induction assumption,  we prove (\ref{eq.3.1}) for $B$. This completes the induction.
$\hfill\Box$

\section{Examples}
(1)  Let $G=(n_1,\ldots,n_r;m_1,\ldots,m_r;\zeta_1,\ldots,\zeta_r)$ be an EGK datum of length $n$. 
For $1 \le i \le n$ we define $\widetilde m_i$ as
\[\widetilde m_i =m_j  \text{ if } n_1+\cdots +n_{j-1} +1 \le i \le n_1+\cdots +n_j, \]
and for such $\widetilde m_1,\ldots,\widetilde m_n$ we define the integers 
 $\frke_1,\ldots,\frke_n$ as in Definition \ref{def.4.3}.

(1.1) An EGK datum of length $2$ is one of the following forms
\begin{itemize}
\item[(a)] $G=(1,1;m_1,m_2;1,\zeta_2)$ with $m_1 <m_2$ and $\zeta_2 \in \calz_3$
\item[(b)] $G=(2;m_1;\zeta_1)$ with $\zeta_2 \in \{\pm 1\}$.
\end{itemize}
 Put $\xi=\zeta_2$ or $\xi=\zeta_1$ according as case (a) or case (b). Then 
\[H=(\widetilde m_1,\widetilde m_2;1,\xi)\]
 is a naive EGK datum such that $\Ups_2(H)=G$, and by a simple calculation combined with  Proposition  \ref{prop.4.1}, $\widetilde \calf(G;Y,X)$ can be expressed as 
\begin{align*}
\widetilde \calf(G;Y,X) &=\sum_{i=0}^{\frke_1} Y^i X^{-\frke_2/2+i}\sum_{j=0}^{\frke_2/2-i} X^{2j}\\
&- \xi \sum_{i=0}^{\frke_1} Y^{i-1}X^{-\frke_2/2+i+1}\sum_{j=0}^{\frke_2/2-i-1}X^{2j}. 
\end{align*} 
Let $B \in \calh_2(\frko )^{\rm{nd}}$. Then by Theorem \ref{th.1.1}, we have
\[\widetilde F(B,X)=\widetilde \calf(\EGK(B);q^{1/2},X).\]
This coincides with [\cite{Ot}, Corollary 5.1].
\bigskip 

\noindent
(1.2) An EGK datum of length $3$ is one of the following forms:
\begin{itemize}
\item[(a)] $G=(1,1,1;m_1,m_2,m_3;1,\zeta_2,\zeta_3)$ with $\zeta_2 \in \calz_3$, and $\zeta_3 \in \{\pm 1 \}$
\item[(b)] $G=(1,2;m_1,m_2;1,\zeta_2)$ with  $\zeta_2 \in \{\pm 1 \}$
\item[(c)] $G=(2,1;m_1,m_2;\zeta_1,\zeta_2)$ with $\zeta_1 \in \calz_3$ and $\zeta_2 \in \{\pm 1 \}$
\item[(d)]$G=(3;m_1;1)$.
\end{itemize}
We put 
\[\xi=
\begin{cases} \zeta_2 &  \text{ in case (a) }\\
\zeta_1 & \text{ in case (c)} \\
1  &  \text { in case (b) or case (d), and $\widetilde m_1+\widetilde m_2$ is even } \\
0  & \text { in case (b) or case (d), and $\widetilde m_1+\widetilde m_2$ is odd, }
\end{cases}\]
and 
\[\eta=
\begin{cases}
 \zeta_3 & \text{ in case (a) }\\
\zeta_2 & \text{ in case (b) or (c)} \\
1  &  \text { in case  (d).}
\end{cases}\]
Moreover put $\frke_2'=2[(a_1+a_2+1)/2]$.  Then, 
\[H=(\widetilde m_1,\widetilde m_2,\widetilde m_3;1,\xi,\eta)\]
 is a naive EGK datum such that $\Ups_3(H)=G$, and  by a simple calculation combined with Proposition \ref{prop.4.1}, $\widetilde \calf(G;Y,X)$ can be expressed as 
\begin{align*}
\widetilde \calf(G;Y,X)&=X^{-\frke_3/2}\Biggl \{\sum_{i=0}^{\frke_1} (Y^2X)^i \sum_{j=0}^{\frke_2'/2-i-1} (YX)^{2j} \\
&+ \eta X^{\frke_3}\sum_{i=0}^{\frke_1} (Y^2X^{-1})^i \sum_{j=0}^{\frke_2'/2-i-1} (YX^{-1})^{2j} \\
&+\xi^2 Y^{\frke_2'}X^{\frke_2'-\frke_1} \sum_{j=0}^{\frke_3-2\frke_2'+\frke_1} (\xi X)^j  \sum_{i=0}^{\frke_1}X^i \Biggr \}.
\end{align*}
Let $B \in \calh_3(\frko )^{\rm{nd}}$. Then by Theorem \ref {th.1.1}, we have
\[\widetilde F(B,X)=\widetilde \calf(\EGK(B);q^{1/2},X).\]
This essentially coincides with [\cite{Kat1}, Example (3)] and [\cite{Wed}, (2.8)] in the case $F=\QQ_p$.

\medskip
\noindent
(2) Let $q$ be odd, and let
\[B \sim \vpi^{a_1}u_1 \bot \cdots \bot \vpi^{a_n}u_n \quad (a_1 \le \cdots \le a_n, \ u_1,\ldots, u_n \in \frko^{\times})\]
be a diagonal Jordan decomposition of  $B \in \calh_n(\frko)^{{\rm nd}}$.
 Put
\[\vep_i=\begin{cases} \xi_{B^{(i)}} & \text{ if $i$ is even} \\
 \eta_{B^{(i)}} & \text{ if $i$ is odd.}
 \end{cases}
 \]
 Then $H=(a_1,\ldots,a_n;\vep_1,\ldots,\vep_n)$ is a naive EGK datum  such that $\Ups_n(H)=\EGK(B)$, and by Proposition
 \ref{prop.4.1} and Theorem \ref{th.1.1}, we can get an explicit formula for $\widetilde F(B,X)$ in terms of $H$, which essentially coincides with
 [\cite{Kat1}, Theorem 4.3] in the case  $F=\QQ_p$. If $F$ is an unramified extension of  $\QQ_2$,  we have an algorithm for giving a naive EGK datum associated with $\EGK(B)$ for  $B \in  \calh_n(\frko)^{{\rm nd}}$ from its Jordan decomposition, and  we  can also give an explicit formula for $\widetilde F(B,X)$ in terms of it . This essentially coincides with
 [\cite{Kat1}, Theorem 4.3] in the case  $F=\QQ_2$ (cf. \cite{C-I-K-L-Y}).

\medskip
\noindent
(3) Fourier coefficients of the Schottky form.

Let $S_k(\mathrm{Sp}_g(\mathbb{Z}))$ be the space of all Siegel cusp form of weight $k$ and genus $g$.
The Schottky form $J\in S_8(\mathrm{Sp}_4(\mathbb{Z}))$ is defined by
\[
J=-\frac{1}{2^{14}\cdot 3^2\cdot 5\cdot 7}\left(\Theta_{E_8\oplus E_8}-\Theta_{D_{16}^+}\right),
\]
where $E_8$ and $D_{16}^+$ are unique indecomposable positive definite even unimodular lattices of size $8$ and $16$, respectively.
It is known that $\dim_\mathbb{C} S_8(\mathbb{Sp}_4(\mathbb{Z}))=\CC\cdot J$.
Let 
\[
J(Z)=\sum_{B\in \Lambda_4^+} A_J(B)\exp(2\pi\iu \mathrm{tr}BZ)
\]
be the Fourier expansion of $J$, where $\Lambda_4^+$ is the set of all positive definite half-integral symmetric matrices of size $4$.
It is known (cf. \cite{Ike1}) that  $J$ is a lift of $\Delta(z)=q\prod_{n=1}^\infty(1-q^n)^{24}\in S_{12}(\SL_2(\mathbb{Z}))$ to $S_8(\mathrm{Sp}_4(\mathbb{Z}))$.
Let $S_{k+(1/2)}^+(\Gamma_0(4))$ be the Kohnen plus space of weight $k+(1/2)$.
Define $\delta(z)=\sum_{n=1}^\infty c(n) q^n\in S_{13/2}^+(\Gamma_0(4))$ by
\[
\delta(z)=\frac{1}{12}\left(\theta(z)E_6(4z)-120\calh_{5/2}(z)E_4(4z)\right)
=q-56q^4+120q^5+\cdots,
\]
where $\tht(z)=\sum_{n\in\ZZ} q^{n^2}$ and $\calh_{5/2}(z)$ is the Cohen Eisenstein series (cf. \cite{cohen}) of weight $5/2$.
Then we have $S_{13/2}^+(\mathrm{SL}_2(\mathbb{Z}))=\CC\cdot \delta$.
Let  $\{\alpha_p, \alpha_p^{-1}\}$ be the Satake parameter of $\Delta(z)$.
Then by \cite{Ike1}, the Fourier coefficient $A_J(B)$ is given by
\[
A_J(B)=\frac{1}{120}c(\frkD_B)\frkf_B^{11/2}\prod_{p|\frkf_B} \tilde F_p(B, \alpha_p),
\]
where $\frkD_B$ is the discriminant of $\mathbb{Q}(\sqrt{\det B})$ and $\frkf_B$ is the positive integer such that $\det(2B)=\frkD_B\frkf_B^2$.
Let $B_i\in\Lam_4^+$ \ ($i=1,\ldots, 10$) be as follows:
\[
\begin{array}{|c|c|c|c|c|}
\noalign{\hrule\vskip -4.85pt}
\vphantom{B^{2^{2^{2^2}}}}
2B_1 & 2B_2 & 2B_3 & 2B_4 & 2B_5 \\
\noalign{\hrule}
\vphantom{\left(\begin{smallmatrix} 
0 \\ 0 \\ 0 \\ 0  \end{smallmatrix}\right)^{2^2}_{2_2}}
 { \left(\begin{smallmatrix} 2 & 1 & 1 & 1 \\ 1 & 2 & 1 & 1 \\ 1 & 1 & 2 & 1 \\ 1 & 1 & 1 & 2  \end{smallmatrix}\right)}
&
  {\left(\begin{smallmatrix} 2 & 1 & 1 & 1 \\ 1 & 2 & 1 & 0 \\ 1 & 1 & 2 & 1 \\ 1 & 0 & 1 & 6  \end{smallmatrix}\right)}
 &
  { \left(\begin{smallmatrix} 2 & 1 & 0 & 1 \\ 1 & 2 & 0 & 0 \\ 0 & 0 & 2 & 0 \\ 1 & 0 & 0 & 4  \end{smallmatrix}\right)}
 &
  {\left(\begin{smallmatrix} 2 & 0 & 0 & 1 \\ 0 & 2 & 0 & 1 \\ 0 & 0 & 2 & 1 \\ 1 & 1 & 1 & 4  \end{smallmatrix}\right)}
 &
  {\left(\begin{smallmatrix} 2 & 0 & 0 & 0 \\ 0 & 2 & 0 & 0 \\ 0 & 0 & 4 & 2 \\ 0 & 0 & 2 & 6  \end{smallmatrix}\right)}
 \\
 \noalign{\hrule \smallskip \hrule}
\vphantom{B_7^{2^2}} 2B_6 & 2B_7 & 2B_8 & 2B_{9} & 2B_{10} \\
\noalign{\hrule}
\vphantom{\left(\begin{smallmatrix} 
0 \\ 0 \\ 0 \\ 0  \end{smallmatrix}\right)^{2^2}_{2_2}}
  { \left(\begin{smallmatrix} 2 & 1 & 1 & 1 \\ 1 & 2 & 0 & 0 \\ 1 & 0 & 6 & 2 \\ 1 & 0 & 2 & 6  \end{smallmatrix}\right)}
 & {\left(\begin{smallmatrix} 2 & 1 & 1 & 1 \\ 1 & 4 & 0 & 0 \\ 1 & 0 & 4 & 0 \\ 1 & 0 & 0 & 4  \end{smallmatrix}\right)}
 & {\left(\begin{smallmatrix} 4 & 2 & 2 & 2 \\ 2 & 4 & 2 & 2 \\ 2 & 2 & 4 & 2 \\ 2 & 2 & 2 & 4  \end{smallmatrix}\right)}
 & {\left(\begin{smallmatrix} 2 & 0 & 0 & 0 \\ 0 & 6 & 2 & 2 \\ 0 & 2 & 6 & 2 \\ 0 & 2 & 2 & 6  \end{smallmatrix}\right)}
 & {\left(\begin{smallmatrix} 4 & 0 & 0 & 2 \\ 0 & 4 & 0 & 2 \\ 0 & 0 & 4 & 2 \\ 2 & 2 & 2 & 8  \end{smallmatrix}\right)}  
  \\
 \noalign{\hrule}
\end{array}
\]
Then we have $\frkD_{B_1}=\cdots=\frkD_{B_{10}}=5$.
We show invariants of these matrices as follows.
The calculation of $A_J(B)$\rq{}s are due to Breulmann-Kuss \cite{br-ku}.
\[
\begin{array}{|r|r|r|r|r|r|}
\noalign{\hrule\vskip -7.2pt}
\vphantom{B^{2^{2^{2^{2^2}}}}}
& B_{1} &  B_{2} &  B_{3} &  B_{4} &  B_{5} 
\\  \noalign{\hrule}
\vphantom{(B)^{2^2}}\frkf_B & 1 & 2 & 2 & 2 & 2^2 
\\ \noalign{\hrule}
\vphantom{(B)^{2^2}}\EGK_2(B) & {\scriptsize 
\left(\begin{smallmatrix} 0 & 0 & 0 & 0 \\ & & & - \end{smallmatrix}\right) }&  
{\scriptsize \left(\begin{smallmatrix} 0 & 0 & 1 & 1 \\ & - & & - \end{smallmatrix}\right) }&  {\scriptsize 
\left(\begin{smallmatrix} 0 & 0 & 0 & 2 \\ & & + & - \end{smallmatrix}\right) }&  {\scriptsize 
\left(\begin{smallmatrix} 0 & 0 & 1 & 1 \\ & + &  & - \end{smallmatrix}\right) }& {\scriptsize  
\left(\begin{smallmatrix} 0 & 1 & 1 & 2 \\ + & & + & - \end{smallmatrix}\right) }
\\ \noalign{\hrule}
\vphantom{(B)^{2^2}}A_J(B) & 1 & -56 & 8 & 72 & 1856 
\\ \noalign{\hrule\smallskip\hrule}
\vphantom{B_7^{2^2}} &  B_{6} &  B_{7} &  B_{8} &  B_{9} &  B_{10} 
\\ \noalign{\hrule}
\vphantom{(B)^{2^2}}\frkf_B & 2^2 & 2^2 & 2^2 & 2^4 & 2^4  
\\ \noalign{\hrule}
\vphantom{(B)^{2^2}}\EGK_2(B) & {\scriptsize 
\left(\begin{smallmatrix} 0 & 0 & 2 & 2 \\ & - & & - \end{smallmatrix}\right) }&  {\scriptsize 
\left(\begin{smallmatrix} 0 & 0 & 2 & 2 \\  & + &  & - \end{smallmatrix}\right) }&  {\scriptsize 
\left(\begin{smallmatrix} 1 & 1 & 1 & 1 \\ & & & - \end{smallmatrix}\right) }
&  {\scriptsize 
\left(\begin{smallmatrix} 0 & 2 & 2 & 2 \\ + &  &  & - \end{smallmatrix}\right) }& {\scriptsize  
\left(\begin{smallmatrix} 1 & 1 & 2 & 2 \\  & + &  & - \end{smallmatrix}\right) }
\\ \noalign{\hrule}
\vphantom{(B)^{2^2}}A_J(B) & 1344 & 2368 & 2880 
& -159232 & 143872 
\\ \noalign{\hrule}
\end{array}
\]
Here, $\EGK_2(B)$ is  the $2$-adic EGK datum.
The first row $\EGK_2(B)$ is the Gross-Keating invariant of $B$.
In the second row of $\EGK_2(B)$, the sign of $\zeta_r$ is indicated at the $n_1+\cdots+n_r$-th entry.
Put $\tilde F_i(X)=\widetilde\calf(\EGK_2(B_i); q^{1/2}, X)$ for $i=1, \ldots, 10$.
Then we have
\begin{align*}
&\tilde F_1(X)=1, 
&\tilde F_2(X)=
S_1+q^{-1/2}(1-q),
\\
&\tilde F_3(X)=
S_1+q^{-1/2},
&\tilde F_4(X)=
S_1+q^{-1/2}(1+q),
\end{align*}
\begin{align*}
\tilde F_5(X)=
&S_2+q^{-1/2}S_1+1+q,
\\
\tilde F_6(X)=
&S_2+q^{-1/2}(1-q)S_1+q,
\\
\tilde F_7(X)=
&S_2+q^{-1/2}(1+q)S_1+2+q,
\\
\tilde F_8(X)=
&S_2+q^{-1/2}(1+q^2)S_1+1+q^2,
\\
\tilde F_{9}(X)=
&S_3+q^{-1/2}S_2+(1+q^2)S_1+q^{-1/2},
\\
\tilde F_{10}(X)=
&S_4+q^{-1/2}(1+q^2)S_3+(1+q+2q^2)S_2 \\
&\qquad\quad +q^{-1/2}(1+2q^2+q^3)S_1+1+q+2q^2.
\end{align*}
where $S_k=X^k+X^{-k}$ \ ($k\in\ZZ$).
These are consistent with the calculation of \cite{br-ku}.

\end{document}